\documentclass[11pt]{amsart}

\usepackage{subcaption}
\usepackage{subfiles}
\usepackage{comment}
\usepackage{float}
\usepackage{marginnote}
\usepackage{tabu}
\usepackage{euscript}
\usepackage{mathdots}
\usepackage[dvipsnames]{xcolor}
\usepackage{graphicx}			
\usepackage{amssymb}
\usepackage{mathrsfs}
\usepackage{amsthm}
\usepackage{amsmath}
\usepackage{stmaryrd}
\usepackage{tikz}
\usepackage{tikz-cd}
\usepackage{accents}
\usepackage{upgreek}
\usepackage{enumerate}
\usepackage{bm}
\usepackage{mathtools}
\usetikzlibrary{patterns}
\usepackage[all]{xy}
\usepackage{caption}
\usepackage{url}
\usepackage{float}
\usepackage{colonequals}
\usepackage{bbm}
\usepackage{longtable}
\usepackage[full]{textcomp}
\usepackage[cal=cm]{mathalfa}
\usepackage{xparse}
\usepackage{cite}
\usepackage{enumitem}

\newcommand{\cev}[1]{\reflectbox{\ensuremath{\vec{\reflectbox{\ensuremath{#1}}}}}}

\usetikzlibrary{calc}
\usetikzlibrary{fadings}
\usetikzlibrary{decorations.pathmorphing}
\usetikzlibrary{decorations.pathreplacing}
\usepackage{tikz,tikz-cd,tikz-3dplot}
\usepackage{pgfplots}
\usetikzlibrary{arrows,shadows,positioning, calc, decorations.markings, 
	hobby,quotes,angles,decorations.pathreplacing,intersections,shapes}
\usepgflibrary{shapes.geometric}
\usetikzlibrary{fillbetween,backgrounds}

\usepackage[margin=0.9in]{geometry}

\tikzset{
	commutative diagrams/.cd, 
	arrow style=tikz, 
	diagrams={>=stealth}
}
\tikzset{
	arrow/.pic={\path[tips,every arrow/.try,->,>=#1] (0,0) -- +(0,4pt);},
	pics/arrow/.default={triangle 90}
}
\tikzset{->-/.style={decoration={
			markings,
			mark=at position .6 with {\arrow{latex}}},postaction={decorate}}
}
\tikzset{
	c/.style={every coordinate/.try}
}

\theoremstyle{theorem}
\newtheorem{innercustomthm}{Theorem}
\newenvironment{customthm}[1]
{\renewcommand\theinnercustomthm{#1}\innercustomthm}
{\endinnercustomthm}

\theoremstyle{theorem}

\theoremstyle{theorem}

\makeatletter
\def\@tocline#1#2#3#4#5#6#7{\relax
	\ifnum #1>\c@tocdepth 
	\else
	\par \addpenalty\@secpenalty\addvspace{#2}%
	\begingroup \hyphenpenalty\@M
	\@ifempty{#4}{%
		\@tempdima\csname r@tocindent\number#1\endcsname\relax
	}{%
		\@tempdima#4\relax
	}%
	\parindent\z@ \leftskip#3\relax \advance\leftskip\@tempdima\relax
	\rightskip\@pnumwidth plus4em \parfillskip-\@pnumwidth
	#5\leavevmode\hskip-\@tempdima
	\ifcase #1
	\or\or \hskip 1em \or \hskip 2em \else \hskip 3em \fi%
	#6\nobreak\relax
	\dotfill\hbox to\@pnumwidth{\@tocpagenum{#7}}\par
	\nobreak
	\endgroup
	\fi}
\makeatother

\newcounter{marginnote}
\DeclareMathAlphabet{\mathpzc}{OT1}{pzc}{m}{it}

\usepackage[backref=page]{hyperref}
\hypersetup{
  colorlinks   = true,          
  urlcolor     = olive,          
  linkcolor    = olive,          
  citecolor   = olive             
}

\usepackage{cleveref}

\theoremstyle{theorem}
\newtheorem{theorem}{Theorem}[section]

\newtheorem{conjecture}[theorem]{Conjecture}
\newtheorem{corollary}[theorem]{Corollary}
\newtheorem{lemma}[theorem]{Lemma}
\newtheorem{proposition}[theorem]{Proposition}

\theoremstyle{definition}

\newtheorem{remark}[theorem]{Remark}

\newtheorem*{runningexample*}{Running example}

\newtheorem*{aside*}{Aside}

\newtheorem{definition}[theorem]{Definition}

\newtheorem{notation}[theorem]{Notation}
\newtheorem{proposition-definition}[theorem]{Proposition-Definition}

\newcommand{\MapsRefinedEvals}{\Mbar_{g,s,\upbeta}(S)^\prime}
\newcommand{\xdashleftrightarrow}[2][]{\ext@arrow 3359\leftrightarrowfill@@{#1}{#2}}

\DeclareMathOperator{\ev}{ev}

\newcommand{\id}{\mathrm{Id}}

\newcommand{\GW}{\mathsf{GW}}

\newcommand{\RR}{\mathbb{R}}

\newcommand{\op}[1]{\operatorname{#1}}
\newcommand{\ol}[1]{\overline{#1}}

\newcommand{\bcd}{\begin{center}\begin{tikzcd}}
		\newcommand{\ecd}{\end{tikzcd}\end{center}}

\newcommand{\Aaff}{\mathbb{A}}
\newcommand{\Aone}{\mathbb{A}^{\! 1}}
\newcommand{\Azero}{\mathbb{A}^{\! 0}}

\newcommand{\bfc}{\mathbf{c}}
\newcommand{\PP}{\mathbb{P}}
\newcommand{\OO}{\mathcal{O}}
\newcommand{\N}{\mathbb{N}}
\newcommand{\Z}{\mathbb{Z}}
\newcommand{\Q}{\mathbb{Q}}
\newcommand{\A}{\mathbb{A}}
\newcommand{\R}{\mathbb{R}}

\newcommand{\virt}{\op{virt}}

\newcommand{\Speck}{\operatorname{Spec}\kfield}
\newcommand{\kfield}{\Bbbk}

\newcommand{\Ocal}{\mathcal{O}}

\newcommand{\Dcal}{\mathcal{D}}
\newcommand{\Rcal}{\mathcal{R}}
\newcommand{\Xcal}{\mathcal{X}}

\newcommand{\Mbar}{\ol{M}}
\newcommand{\Nbar}{\ol{N}}

\newcommand{\FF}{\mathbb{F}}

\newcommand{\calO}{\mathcal{O}}

\newcommand{\pt}{\mathrm{pt}}

\newcommand{\scrD}{\EuScript{D}}
\newcommand{\scrE}{\EuScript{E}}

\newcommand{\scrL}{\EuScript{L}}

\newcommand{\scrS}{\EuScript{S}}

\newcommand{\scrY}{\EuScript{Y}}
\newcommand{\scrZ}{\EuScript{Z}}

\newcommand{\tropf}{\mathsf{f}}
\newcommand{\ftrop}{\mathsf{f}}

\newcommand{\frgt}{\upvarsigma}

\crefname{equation}{eq.}{eqs.}
\crefname{eqnarray}{eq.}{eqs.}
\crefname{conjecture}{conjecture}{conjectures}
\crefname{lemma}{lemma}{lemmas}
\crefname{theorem}{theorem}{theorems}
\crefname{claim}{claim}{claims}
\crefname{remark}{remark}{remarks}
\crefname{proposition}{proposition}{propositions}
\crefname{section}{section}{sections}
\crefname{appendix}{appendix}{appendices}
\crefname{corollary}{corollary}{corollaries}
\crefname{figure}{figure}{figures}
\crefname{table}{table}{tables}
\crefname{example}{example}{examples}
\crefname{assumption}{assumption}{assumptions}
\crefname{definition}{definition}{definitions}
\crefname{innercustomthm}{theorem}{theorems}
\crefname{innercustomcor}{corollary}{corollaries}
\crefname{innercustomconj}{conjecture}{conjectures}

\setlist[enumerate,1]{label=(\roman*),itemsep=0.9ex}
\setlist[itemize]{itemsep=0.9ex}

\begin{document}
	
	\title{Gromov--Witten theory of bicyclic pairs}
	\author{Michel van Garrel, Navid Nabijou, Yannik Schuler}
	
	\begin{abstract} A bicyclic pair is a smooth surface equipped with a pair of smooth divisors intersecting in two reduced points. Resolutions of self-nodal curves constitute an important special case. We investigate the logarithmic Gromov--Witten theory of bicyclic pairs. We establish correspondences with local Gromov--Witten theory and open Gromov--Witten theory in all genera, a correspondence with orbifold Gromov--Witten theory in genus zero, and correspondences between all-genus refined Gopakumar--Vafa invariants and refined quiver Donaldson--Thomas invariants. For self-nodal curves in $\PP(1,1,r)$ we obtain closed formulae for the genus zero invariants and relate these to the invariants of local curves. We also establish a conceptual relationship between invariants relative a self-nodal plane cubic and invariants relative a smooth plane cubic. The technical heart of the paper is a qualitatively new analysis of the degeneration formula for stable logarithmic maps, involving a tight intertwining of tropical and intersection-theoretic vanishing arguments.
	\end{abstract}
	
	\vspace{-2.5cm}
	
	\maketitle
	\tableofcontents
	
	\vspace{-1cm}
	
\section*{Introduction}

\noindent Consider a self-nodal plane cubic $D \subseteq \PP^2$. We study curves in $\PP^2$ with prescribed tangency orders along $D$, via the framework of logarithmic Gromov--Witten theory. We work in the more general setting of bicyclic pairs.

A \textbf{bicyclic pair} is a pair $(S \, | \, D+E)$ consisting of a smooth surface $S$ and a pair of smooth divisors $D,E \subseteq S$ intersecting in two reduced points:
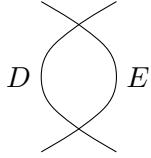
\[
\begin{tikzpicture}

	\draw (-1,1) to[out=210,in=40] (-1.75,0.5) to[out=220,in=90] (-2,0) to[out=270,in=140] (-1.75,-0.5) to[out=320,in=150] (-1,-1);
	\draw (-2,0) node[left]{$D$};
	
	\draw (-2,1) to[out=330,in=140] (-1.25,0.5) to[out=320,in=90] (-1,0) to[out=270,in=40] (-1.25,-0.5) to[out=220,in=30] (-2,-1);
	\draw (-1,0) node[right]{$E$};
	
\end{tikzpicture}
\]
These arise in two distinct contexts:	
\begin{enumerate}
\item \textbf{Looijenga pairs} with two boundary components \cite{Looijenga}. In this case $D+E \in |\!-\!\!K_S|$ with $D$ and $E$ both rational. When $D$ and $E$ are nef there is a simple classification \cite[Table~1]{BBvG2}. The general classification proceeds via the minimal model program, see \cite[Lemma~3.2]{FriedmanMiranda} or \cite[Section~2]{FriedmanAnticanonical}.
\item \textbf{Self-nodal pairs.} Consider a surface containing an irreducible curve with a single nodal singularity. Let $S$ be the blowup of the surface at the node, $D$ the strict transform of the irreducible curve, and $E$ the exceptional divisor. Then $(S\, | \, D+E)$ is a bicyclic pair. Note that $E$ is not nef.
\end{enumerate}
We study the Gromov--Witten theory of bicyclic pairs; by \cite{AbramovichWiseBirational} this includes the Gromov--Witten theory of self-nodal pairs. We investigate connections to local, open, and orbifold geometries, obtain closed formulae in important special cases, and probe the behaviour under smoothing of the divisor. The main theories we consider are:
\[
\begin{tikzcd}
\fbox{$\left(S\, |\, D+E\right)$} \ar[rr,leftrightarrow,"\text{\Cref{sec: bicyclic pairs}}"] & & \fbox{$\big(\OO_S(-D)\, | \, \hat{E}\big)$\footnotemark} \ar[rr,leftrightarrow,"\text{\Cref{sec: nef pairs}}"] \ar[d,leftrightarrow, "\text{\Cref{sec: toric pairs}}"] & & \fbox{$\OO_S(-D) \oplus \OO_S(-E)$} \\
& & \fbox{$\OO_S(-D)|_{S \setminus E}$} \ar[rr,rightarrow,"\text{\Cref{sec: nodal cubic}}"] & & \fbox{$\OO_{\PP^1}(r)\oplus\OO_{\PP^1}(-r-2)$}
\end{tikzcd}
\]
\footnotetext{\label{footnote: vb notation}For $\uppi \colon V \to S$ a vector bundle and $E \subseteq S$ a divisor, we will write $\hat{E}\coloneqq \pi^{-1}(E)$.}
We work in several different settings. The precise hypotheses are stated clearly at the start of each section. Globally, the paper proceeds from most general to most specific. Our motivating example is the resolution of a self-nodal plane cubic (see \Cref{sec: self-nodal curve introduction}).


\subsection{Logarithmic-local correspondence (Sections~\ref{sec: setup}--\ref{sec: degeneration formula analysis})} Our first main result establishes an all-genus correspondence between the Gromov--Witten theories of the following pairs:
\[
(S \, | \, D+E) \leftrightarrow (\OO_S(-D) \, | \, \hat{E}).
\]
If $D\!+\!E \in |\!-\!K_S|$ then the right-hand side is a logarithmically Calabi--Yau threefold.

The proof proceeds via degeneration to the normal cone of the divisor $D \subseteq S$. This is a well-established technique \cite{vGGR,TsengYouHigherGenus,BFGW} but severe complications arise due to the non-trivial logarithmic structure on the general fibre.

Consider a bicyclic pair $(S\, | \, D+E)$ and fix a curve class $\upbeta \in A_1(S;\Z)$. We assume:
\begin{itemize}
\item $D \cong \PP^1$ and $D^2 \geq 0$.
\item $D \cdot \upbeta > 0$ and $E \cdot \upbeta \geq 0$.\footnote{The case $E \cdot \upbeta=0$ can occur, e.g. for resolutions of self-nodal pairs. In this case, the spaces of stable logarithmic maps to $(S\, |\, D+E)$ and $(S|D)$ have the same virtual dimension, but their logarithmic Gromov--Witten invariants typically differ. See \Cref{rmk: E times beta can be zero}.}	
\end{itemize}
Fix a genus $g$ and tangency data $\bfc$ with respect to $E$. From $\bfc$ we build tangency data $\hat{\bfc}$ with respect to $D+E$ by appending an additional marked point with tangency $D \cdot \upbeta$ along $D$.

Finally let $\upgamma$ be a collection of $D$-avoidant insertions (\Cref{def: E disjoint insertions}) of the correct codimension. This permits markings with no insertions, as well as descendants of evaluation classes disjoint from~$D$.

\begin{customthm}{A}[\Cref{thm: local-log}] \label{thm: local-log introduction} We have the following equality of generating functions:		
\[ \dfrac{(-1)^{D \cdot \upbeta -1}}{2 \sin\left( \frac{D \cdot \upbeta}{2} \hbar \right)}\left( \sum_{g \geq 0} \GW_{g,\hat{\bfc},\upbeta}(S\, | \, D+E) \langle (-1)^g \uplambda_g \upgamma \rangle \cdot \hbar^{2g-1} \right) = \sum_{g \geq 0} \GW_{g,\bfc,\upbeta}( \OO_S(-D)\, |\, \hat{E})\langle \upgamma \rangle \cdot \hbar^{2g-2}.\]
\end{customthm}

The proof proceeds via the degeneration formula for stable logarithmic maps \cite{ACGSDecomposition,RangExpansions} applied to the degeneration to the normal cone of $D \subseteq S$. The general fibre has non-trivial logarithmic structure corresponding to $E$, which greatly complicates the analysis.

We develop combinatorial techniques for constraining the shapes of rigid tropical types, and intersection theoretic techniques for establishing the vanishing of certain contributions (see \Cref{sec: degeneration formula analysis}). A novel feature is the close intertwining of these modes, and we switch frequently back and forth between tropical and intersection theoretic arguments. We expect this to serve as a useful blueprint for future calculations.

Having established \Cref{thm: local-log introduction}, we explore two proximate results in genus zero.

\subsection{Nef pairs (\Cref{sec: nef pairs})} Combining \Cref{thm: local-log introduction} with \cite[Theorem~1.1]{vGGR} we obtain:
\begin{customthm}{B}[\Cref{thm: nef correspondence}] \label{thm: nef correspondence introduction} Suppose that $E^2 \geq 0$ and $E \cdot \upbeta > 0$ and let $\bfc=(E \cdot \upbeta)$ be maximal tangency contact data to $E$. Then we have:
\[	\GW_{0,\hat{\bfc},\upbeta}(S\, | \, D+E)\langle \upgamma \rangle = (-1)^{(D+E)\cdot \upbeta} (D\cdot\upbeta)(E\cdot\upbeta) \cdot \GW_{0,\bfc,\upbeta}(\OO_S(-D) \oplus \OO_S(-E))\langle \upgamma \rangle.\]	
\end{customthm}
This gives another instance of the numerical logarithmic-local correspondence for normal crossings pairs \cite{BBvG1,BBvG2,BBvG3,TsengYouMirror}. See \Cref{rmk: orbifold correspondence strong} for the importance of $D$-avoidant insertions.

\subsection{Root stacks and self-nodal pairs (\Cref{sec: exceptional divisors})}
We next impose that $E \cdot \upbeta=0$. This includes the case of self-nodal pairs discussed above. In this setting we provide an alternative proof of \Cref{thm: local-log introduction}, by passing through the logarithmic-orbifold \cite{RootLog} and orbifold-local \cite{BNTY} correspondences (\Cref{thm: exceptional divisors correspondence}). This allows us to remove the assumptions that $D$ is rational and $\upgamma$ is $D$-avoidant. The key intermediate result is:

\begin{customthm}{C}[\Cref{prop: log orb coincide}] \label{thm: log orbifold introduction} Suppose that $E \cdot \upbeta=0$. Then there is an equality between logarithmic and orbifold Gromov--Witten invariants:
	\[ \GW^{\mathsf{log}}_{0,\hat{\bfc},\upbeta}(S\, | \, D+E)\langle \upgamma \rangle = \GW^{\mathsf{orb}}_{0,\hat{\bfc},\upbeta}(S\, | \, D+E)\langle \upgamma \rangle.\]
	This holds without assuming that $D \cong \PP^1$ or that $\upgamma$ is $D$-avoidant.
\end{customthm}
This result is proved using \cite[Theorem~X]{RootLog}. The main technical step is to strongly constrain the shapes of tropical types of maps to $(S\, |\, D+E)$.

\subsection{Toric and open geometries (Sections~\ref{sec: setup toric}--\ref{sec: localisation calculation})} \label{sec: toric introduction} We next specialise to toric Calabi--Yau pairs. We assume:
\begin{itemize}
\item $S$ is toric.
\item $E$ is a toric hypersurface.
\item $D+E \in |\!-\!\!K_S|$.
\item $E \cdot \upbeta=0$.	
\end{itemize}
We do not require that $D$ is toric. In this setting, we obtain a logarithmic-open correspondence:
\begin{customthm}{D}[\Cref{thm: local open correspondence}] \label{thm: local open introduction} For all $g\geq 0$ we have:
		\[	\GW_{g,0,\upbeta} (\OO_S(-D)|\hat{E}) = \GW_{g,0,\upiota^\star \upbeta}^T(\OO_S(-D)|_{S \setminus E})	\]
where the Gromov--Witten invariant on the right-hand side is defined by localising with respect to the action of the Calabi--Yau torus $T$ (see \Cref{sec: open invariants}).
\end{customthm}
The proof proceeds by localisation on both sides. The difficult step is to establish the vanishing of certain contributions, by isolating a weight zero piece of the obstruction bundle. For this it is crucial that we localise with respect to the Calabi--Yau torus $T$.

The target $\OO_S(-D)|_{S \setminus E}$ is a toric Calabi--Yau threefold and hence its invariants can be computed using the topological vertex formalism \cite{LiLiuLiuZhouTopVert}. \Cref{thm: local open introduction} combined with \Cref{thm: local-log introduction} thus provides a new means to efficiently compute the all-genus logarithmic invariants of $(S\, |\, D+E)$, which is otherwise a highly tedious endeavour.

\subsection{Applications to GV, BPS, and quiver DT invariants (\Cref{sec: GV BPS DT})} \label{sec: applications introduction} Applying \Cref{thm: local open introduction}, we obtain correspondences
\[
\begin{tikzcd}
\fbox{$\mathrm{BPS}\left(S\, |\, D+E\right)$} \ar[rrrr,leftrightarrow,"\text{\Cref{cor: refined}}"] \ar[rrd,leftrightarrow,"\text{\cite{BousseauQuantumTropicalVertex}}" {xshift=-20pt,yshift=-17pt}] &&&& \fbox{$\mathrm{GV}\left(\OO_S(-D)|_{S \setminus E}\right)$} \ar[dll,leftrightarrow,"\text{\Cref{cor:divisibility-quiver}}"] \\ 
&& \fbox{$\mathrm{DT}(Q)$} && 
\end{tikzcd}
\]
where $Q$ is a quiver built from the pair $(S\,|\,D+E)$. \Cref{cor:divisibility-quiver} in particular establishes two surprising facts. First, the space of quiver representations is virtually a projective bundle, which is not the case geometrically. Second, $Q$ differs from the standard quiver associated to the Calabi--Yau threefold $\OO_S(-D)|_{S \setminus E}$ (although $Q$ does arise in the theory of exponential networks on the mirror curve, see \Cref{sec: prospects networks}).

\subsection{Self-nodal curves: calculations (\Cref{sec: scattering})} \label{sec: self-nodal curve introduction} We now specialise to our motivating example. Consider $S_r \colonequals \PP(1,1,r)$ and let $D_r \in |\!-\!\!K_{S_r}|$ be an irreducible curve with a single nodal singularity at the singular point of $S_r$. We consider degree $d$ curves in the pair $(S_r|D_r)$ which meet $D_r$ in a single point of maximal tangency order $d(r+2)$.
\begin{customthm}{E}[\Cref{thm: nodal cubic invariants}] \label{thm: nodal cubic introduction}
	We have:
	\[ \GW_{0,(d(r+2)),d}(S_r|D_r) =  \dfrac{r+2}{d^2} {(r+1)^2 d-1 \choose d-1}. \] 
\end{customthm}
\begin{remark} The numerator $(r+2)$ is the $d=1$ invariant. The remaining factors strongly resemble the multiple cover formula of \cite[Proposition~6.1]{GPS}, however for a curve of tangency order $(r+1)^2+1$.\end{remark}

\Cref{thm: nodal cubic introduction} is proved via the Gromov--Witten/quiver correspondence \cite{BousseauQuantumTropicalVertex}. Passing to the resolution by the Hirzebruch surface $\mathbb{F}_r \to S_r$ and combining Theorems~\ref{thm: local open introduction}~and~\ref{thm: nodal cubic introduction} we immediately obtain the following formula, already known in the physics literature \cite[Equation~(4.53)]{CaporasoGriguoloMarinoPasquettiSeminara}:

\begin{customthm}{F}[\Cref{thm: local P1 invariants}] \label{thm: local P1 introduction}
	We have:
	\[ \GW^T_{0,0,d}\big(\Ocal_{\PP^1}(r)\oplus\Ocal_{\PP^1}(-r-2)\big) = \dfrac{(-1)^{rd - 1}}{d^3} {(r+1)^2 d-1 \choose d-1}. \] 
\end{customthm}

\subsection{Self-nodal curves: smoothings (\Cref{sec: degenerating hypersurfaces})}
In the final section we focus on the case $r=1$. We compare the Gromov--Witten invariants of $(\PP^2|D)$ and $(\PP^2|E)$ where $D$ and $E$ are nodal and smooth plane cubics. Experimentally we observe that the former are always smaller than the latter. We provide a conceptual explanation for this defect, via the enumerative geometry of degenerating hypersurfaces. As in \cite{BarrottNabijou} we degenerate both $D$ and $E$ to the toric boundary $\Delta \subseteq \PP^2$ and consider the invariants of the logarithmically singular central fibre.

\begin{customthm}{G}[\Cref{thm: nodal cubic invariants are contributions}] \label{thm: nodal cubic invariants are contributions introduction} The invariants of $(\PP^2|D)$ are precisely the central fibre contributions to the invariants of $(\PP^2|E)$ arising from multiple covers of a single coordinate line.
\end{customthm}
Finally (\Cref{thm: BN conjecture holds}) we apply \Cref{thm: nodal cubic invariants are contributions introduction} to settle a conjecture in \cite{BarrottNabijou}.

\begin{remark} Taken together, Theorems~\ref{thm: nodal cubic introduction}, \ref{thm: local P1 introduction}, \ref{thm: nodal cubic invariants are contributions introduction} relate the Gromov--Witten invariants arising from two non-standard obstruction theories on the space of stable maps to $\PP^1$: the local geometry and the degenerated hypersurface.\end{remark}

\subsection{Context} \label{sec: context}

\subsubsection{Enumerative correspondences}

The present paper fits into the broader body of work on logarithmic-local-open-quiver correspondences \cite{TakahashiMirror,GathmannMirror,vGWZ,vGGR,CvGKTlocalBPS,CvGKTBPS,CvGKTDegenerateContributions,BBvG1,BBvG2,BBvG3,MaxContacts,BFGW,LiuYu,YuBPS,FanWu,BousseauQuiver,BW,BriniSchulerQuasiTameLooijenga,WangGS,AL23}. By \cite{BNTY} this is intimately connected to the logarithmic-orbifold correspondence, another area of intense study \cite{Cadman,CadmanChen,AbramovichCadmanWise,TsengYouHigherGenus,RootLog,BNR24}. This area enjoys close connections to Mirror Symmetry \cite{BousseauScatt,BousseauTakahashi,LiuYu2,FanTsengYouMirror,TsengYouMirror,YouMM,Graefnitz,Grae,GRZ,YouLG,ShafiQuasimaps}.

In \cite{MaxContacts} it is shown that the na\"ive logarithmic-local correspondence fails for normal crossings pairs, and a corrected form is established. It is then observed that in many situations, the insertions cap trivially with the correction terms, collapsing the corrected correspondence back to the na\"ive correspondence. The results of this paper provide another instance of this phenomenon.

\subsubsection{Logarithmic Fanos} \label{sec:comp-log-Fano}

If $(S\, |\, D+E)$ is Looijenga then $(S\,|\,D)$ is a logarithmically Fano surface. When in addition $D$ is ample of virtual genus zero, \cite[Theorem~1.7]{BW} relates the all genus Gromov--Witten invariants of $(S\,|\,D)$ and $\OO_S(-D)$ by applying the more general \cite[Theorem~1.1]{BFGW}. 

Our techniques can be used to recover \cite[Theorem~1.7]{BW}. Following \cite[Proposition~3.1]{BousseauQuiver} we pass to a deformation and identify the invariants of $(S\, |\, D)$ with the invariants of $(S\, |\, D+E)$ with tangency orders $(1,\ldots,1)$ along $E$. Similarly we identify the invariants of $\OO_S(-D)$ with the invariants of $(\OO_S(-D)\, |\, \hat{E})$. The correspondence \cite[Theorem~1.7]{BW} then follows from \Cref{thm: local-log introduction}.

For a del Pezzo surface with smooth anticanonical divisor, the divisibility statement analogous to \Cref{cor: refined} was conjectured in \cite[Conjecture~1.2]{CvGKTlocalBPS} and proven for the projective plane in \cite{GW24}. The relationship to counts of 1-dimensional sheaves on local del Pezzo surfaces was conjectured in \cite[Conjecture~0.3]{BousseauQuantumTropicalVertex}.

\subsubsection{Mirror Symmetry}

Following \cite{GHK}, the mirror family to $(S_r|D_r)$ is built by constructing their toric models as in \Cref{sec: scattering}. As the discrete Legendre transform \cite{GrossSiebertlogdegI} of the toric model is itself, it thus follows that $(S_r|D_r)$ is isomorphic to a generic member of its mirror family (and not just diffeomorphic as is expected for hyperkähler surfaces).

Moreover, the scattering diagrams of $(S_r|D_r)$ form the $(r+2)$-local scattering diagrams \cite{GPS,BousseauScatt} for the mirror symmetry construction of \cite{GHK}. Hence finding an effective description of the genus zero logarithmic Gromov--Witten invariants of $(S_r|D_r)$ has potential applications in the explicit description of mirror families to logarithmically Calabi--Yau surfaces, an often intractable problem.

\subsection{Prospects}

\subsubsection{Logarithmic-open}
Combining Theorems~\ref{thm: local-log introduction}~and~\labelcref{thm: local open introduction} we obtain a logarithmic-open correspondence for two-component Looijenga pairs \cite[Conjecture~1.3]{BBvG2} in the setting where the curve class pairs trivially with one of the components. If the curve class pairs non-trivially with both components, we still expect \Cref{thm: local-log introduction} to lead to progress on the logarithmic-open correspondence. We intend to investigate this in future work.

\subsubsection{Higher dimensions}
We restrict to surface geometries in this paper. This ensures a simplified balancing condition at codimension-$2$ strata, used crucially in \Cref{sec: third vanishing} to eliminate the contributions of rigid tropical types which are not star-shaped. The resulting \Cref{thm: third vanishing} strongly resembles \cite[Theorem~1.1]{BFGW}. Intriguingly, however, the latter result holds in all dimensions.

We speculate that the arguments of \Cref{sec: third vanishing} may be refined to produce a higher-dimensional analogue of \Cref{thm: third vanishing}, leading to a higher-dimensional analogue of \Cref{thm: local-log introduction}. Continuing in this vein, we speculate that the same arguments may be adapted to treat pairs $(S\, | \, D+E)$ such that the intersection $D \cap E$ consists of three or more smooth connected components.

\subsubsection{Scattering}
\label{sec:intro-scatt}
\Cref{thm: nodal cubic introduction} computes the genus zero Gromov--Witten invariants of $\PP(1,1,r)$ relative to a self-nodal anticanonical divisor. We achieve this by equating these invariants with the Donaldson--Thomas invariants of the $(r+2)$-Kronecker quiver with dimension vector $(d,d)$. The latter are encoded in the wall-crossing function of the central ray of a local scattering diagram, with incoming ray directions $\uprho_1,\uprho_2$ satisfying $|\uprho_1 \wedge \uprho_2|=r+2$. We then use \cite[Equation (1.4)]{GP} proven by Reineke \cite{Reineke} for $\ell_1=\ell_2$ to arrive at \Cref{thm: nodal cubic introduction}, from which we deduce \Cref{thm: local P1 introduction}. 

It may be possible to reverse this logic, proving \Cref{thm: local P1 introduction} first via a direct analysis of the local invariants (as in \cite{CaporasoGriguoloMarinoPasquettiSeminara}) and using this to deduce \Cref{thm: nodal cubic introduction} and hence \cite[Equation (1.4)]{GP}. Speculatively, it may also be possible to study the case $\ell_1\neq\ell_2$, by extending the arguments of \Cref{sec: nodal cubic} to weighted projective planes of the form $\PP(1,a,b)$.

Parts of \Cref{sec: nodal cubic} extend readily to higher genus using Bousseau's quantum scattering \cite{BousseauQuantumTropicalVertex}. An analysis of the central ray of the quantum scattering diagram of the Kronecker quivers, via a correspondence with the all-genus Gromov--Witten generating function of local $\mathbb{P}^1$, is an attractive prospect.

\subsubsection{Stable Lagrangians and spectral networks} \label{sec:A-branes} \label{sec: prospects networks}

The correspondence between logarithmic BPS and open Gopakumar--Vafa invariants (\Cref{sec: applications introduction}) has a heuristic explanation arising from mirror symmetry, as an identity of counts of curves in the open geometry with counts of stable Lagrangians in the mirror geometry. This is analogous to the heuristics for the local projective plane described in \cite[Section~8]{BousseauTakahashi} and elaborated in \cite[Section~4.1]{ABQuiver}, and with the general formulation given in \cite{BousseauHoloFloer}.

This expectation is supported by the analysis of the moduli space of certain Lagrangians in the mirror geometry \cite{BLRAbranes,BLRfoliations} and the link to exponential networks \cite{BLRExpI,BLRExpII,ESWExp,BLRExplore,GLPY} and more generally spectral networks \cite{GMNfourdim,GMNwallcrossing,BridgelandSmith}. In this context, the link to the representation theory of Kronecker quivers is established in \cite[Section~2.2]{BLRfoliations}.

This suggests that there may exist a correspondence between the scattering diagram of $(S\,|\,D+E)$ and the spectral networks of the mirror family. It would be interesting to pursue this line of enquiry.

\subsection{Assumptions}
Theorems \ref{thm: local-log introduction}--\ref{thm: local open introduction} concern the Gromov--Witten theory of bicyclic pairs $(S\,|\,D+E)$ with curve class $\upbeta$ and genus $g$ satisfying varying assumptions. To first approximation, our assumptions get more restrictive as the paper progresses and our results become more specific. For the convenience of the reader, below we summarise our assumptions together with sections in which they apply:
\begin{center}
	\def\arraystretch{1.25}
\begin{tabular}{|l|l|l|l|}
	\hline
	Section & Assumptions & Insertions & Main results \\ \hline \hline
	\ref{sec: setup}--\ref{sec: degeneration formula analysis} & $D\cong \PP^1$, $D^2\geq 0$, $D\cdot \upbeta >0$, $E\cdot \upbeta \geq 0$ & $D$-avoidant & \Cref{thm: local-log introduction} \\ \hline
	\ref{sec: nef pairs} & $D \cong \PP^1$, $D^2 \geq 0$, $D \cdot \upbeta > 0$, $E^2 \geq 0$, $E \cdot \upbeta > 0$ & $D$-avoidant & \Cref{thm: nef correspondence introduction} \\ \hline
	\ref{sec: exceptional divisors} & $D^2\geq 0$, $D \cdot \upbeta >0$, $E\cdot \upbeta = 0$, $g=0$ & arbitrary & \Cref{thm: log orbifold introduction} \\ \hline
	\ref{sec: toric pairs} & $S$ toric, $E$ toric hypersurface, $D+E\in|-K_S|$, $E\cdot \upbeta = 0$ & none &  \Cref{thm: local open introduction}\\ \hline
\end{tabular}
\end{center}
In \Cref{sec: nodal cubic} we then apply our results to self-nodal pairs. More precisely, \Cref{sec: scattering} (\Cref{thm: nodal cubic introduction} and \ref{thm: local P1 introduction}) concerns pairs $(S_r|D_r)$ where $S_r=\PP(1,1,r)$ and $D_r\in|-K_{S_r}|$ is an irreducible curve with a single nodal toric singularity at the singular point of the surface (or at one of the torus fixed points if $r=1$). Lastly, in \Cref{sec: degenerating hypersurfaces}, we further specialise to $r=1$ to prove \Cref{thm: nodal cubic invariants are contributions introduction}.

\subsection*{Acknowledgements} This project originated in joint discussions with Andrea~Brini, Pierrick~Bousseau, and Tim~Gr\"afnitz. We are grateful for their generosity and insights. We have benefited from conversations with Luca~Battistella, Francesca~Carocci, Cyril~Closset, Ben~Davison, Samuel~Johnston, Naoki~Koseki and Dhruv~Ranganathan. Special thanks are owed to Qaasim~Shafi and Longting~Wu, for extremely helpful clarifications at key points. We thank the anonymous referees for several valuable comments.
	
	Parts of this work were carried out during research visits at the Isaac Newton Institute, the University of Sheffield, the University of Birmingham, Imperial College London, and Queen Mary University of London. We are grateful to these institutions for hospitality and financial support.
		
\section{Bicyclic pairs}\label{sec: bicyclic pairs}
\noindent We introduce bicyclic pairs $(S\, |\, D+E)$. The main result of this section (\Cref{thm: local-log}) establishes a precise correspondence between the logarithmic Gromov--Witten theories of the pairs $(S\, |\, D+E)$ and~$(\OO_S(-D)\, |\, \hat{E})$. This occupies Sections~\ref{sec: setup}--\ref{sec: degeneration formula analysis}. Applications and variations are discussed in Sections~\ref{sec: nef pairs} and \ref{sec: exceptional divisors}.

\subsection{Setup and statement of correspondence} \label{sec: setup}
		
\begin{definition} A \textbf{bicyclic pair} $(S\, |\, D+E)$ consists of a smooth projective surface $S$ and smooth divisors $D,E \subseteq S$ such that $D$ and $E$ intersect in two reduced points, denoted $q_1$ and $q_2$. The normal crossings pair $(S\, | \, D+E)$ has tropicalisation:

\begin{equation}\label{eqn: tropicalisation}
		\begin{tikzpicture}[scale=0.8,baseline=(current  bounding  box.center)]
			\draw[fill=black] (0,0) circle[radius=2pt];
			
			\draw[->] (0,0) -- (3,0);
			\draw (3,0) node[right]{$D$};
			
			\draw[->] (0,0) -- (0,3);
			\draw (0,3) node[above]{$E$};
			
			\draw[->] (0,0) -- (-3,0);
			\draw (-3,0) node[left]{$D$};
			
			
			\draw (1.5,0) node{$\mid\mid$};
			\draw (-1.5,0) node{$\mid\mid$};
			
			\draw (1.5,1.5) node{$q_2$};
			\draw (-1.5,1.5) node{$q_1$};
		\end{tikzpicture}
\end{equation}
\end{definition}

\noindent Fix a bicyclic pair $(S\, |\, D+E)$ and a curve class $\upbeta \in A_1(S;\Z)$ such that:
\begin{itemize}
\item $D \cong \PP^1$ and $D^2 \geq 0$.
\item $D \cdot \upbeta > 0$ and $E \cdot \upbeta \geq 0$.	
\end{itemize}
We consider stable logarithmic maps to $(S \, | \, D+E)$ with genus $g$, class $\upbeta$ and markings $x,y_1, \ldots, y_s$ carrying the following tangency conditions:
\begin{itemize}
\item $x$ has maximal tangency $D \cdot \upbeta$ with respect to $D$.
\item $y_j$ has tangency $\upalpha_j \geq 0$ with respect to $E$ (with $\Sigma_{j=1}^s \upalpha_j = E \cdot \upbeta$).
\end{itemize}
We let $\hat{\bfc}$ denote the matrix of tangency data. The evaluation space $\mathrm{Ev}_j$ corresponding to the marking $y_j$ is defined as:
\begin{equation} \label{eqn: definition evaluation space} \mathrm{Ev}_j \colonequals \begin{cases} E \qquad & \text{if $\upalpha_j>0$} \\ S \qquad & \text{if $\upalpha_j = 0$}.\end{cases} \end{equation}
We restrict to the following class of insertions.
\begin{definition} \label{def: E disjoint insertions} A class $\upgamma_j \in A^\star(\mathrm{Ev}_j)$ is \textbf{$D$-disjoint} if there exists a regularly embedded subvariety $Z_j \subseteq \mathrm{Ev}_j$ such that $[Z_j] = \upgamma_j$ and $Z_j \cap D = \emptyset$. An assembly of insertions
\[ \upgamma = \prod_{j=1}^s \ev_{y_j}^\star(\upgamma_j) \psi_{y_j}^{k_j} \]
is \textbf{$D$-avoidant} if for each $j \in \{1,\ldots,s\}$ one of the following two conditions holds:
\begin{enumerate}
\item $\upgamma_j = \mathbbm{1}_{\mathrm{Ev}_j}$ and $k_j=0$.
\item $\upgamma_j$ is $D$-disjoint.
\end{enumerate}
This provides an enlargement of the stationary sector.
\end{definition}

With the above setup, we use the moduli space of stable logarithmic maps \cite{ChenLog,AbramovichChenLog,GrossSiebertLog} to define logarithmic Gromov--Witten invariants with a lambda class insertion
\begin{equation} \label{eqn: log invariant} \GW_{g,\hat{\bfc},\upbeta}(S\, | \, D+E)\langle (-1)^g \uplambda_g \upgamma \rangle \colonequals (-1)^g \uplambda_g \upgamma \cap [\Mbar_{g,\hat{\bfc},\upbeta}(S\, | \, D+E)]^{\virt} \in \Q. \end{equation}
We next consider the pair $(\OO_S(-D)\, |\, \hat{E})$ where $\hat{E}$ is the preimage of $E$ under the projection $\OO_S(-D) \rightarrow S$. We obtain tangency data $\bfc$ from $\hat{\bfc}$ by deleting the marking $x$, and define logarithmic Gromov--Witten invariants of the local target
\begin{equation} \label{eqn: local invariant} \GW_{g,\bfc,\upbeta}(\OO_S(-D)\, |\, \hat{E})\langle \upgamma \rangle \colonequals \upgamma \cap [\Mbar_{g,\bfc,\upbeta}(\OO_S(-D)\, |\, \hat{E})]^{\virt} \in \Q. \end{equation}
Since $D^2 \geq 0$ it follows that $D$ is nef, and then $D \cdot \upbeta > 0$ implies that $H^0(C,f^\star \OO_S(-D))=0$ for any stable map $f \colon C \to S$ of class $\upbeta$. The local theory of $\OO_S(-D)$ is thus well-defined. The main result of this section is a correspondence between the invariants \eqref{eqn: log invariant} and \eqref{eqn: local invariant}.
\begin{theorem}[\Cref{thm: local-log introduction}] \label{thm: local-log} We have the following equality of generating functions:	
\[ \dfrac{(-1)^{D \cdot \upbeta - 1}}{2 \sin\left( \frac{D \cdot \upbeta}{2} \hbar \right)} \sum_{g \geq 0} \GW_{g,\hat{\bfc},\upbeta}(S\, | \, D+E) \langle (-1)^g \uplambda_g \upgamma \rangle \hbar^{2g-1} = \sum_{g \geq 0} \GW_{g,\bfc,\upbeta}( \OO_S(-D)\, |\, \hat{E})\langle \upgamma \rangle \hbar^{2g-2}.\]
\end{theorem}

\begin{remark} \label{rmk: E times beta can be zero} The case $E \cdot \upbeta=0$ is possible, and in fact one of the most interesting (see \Cref{sec: nodal cubic}). In this case there are no markings tangent to $E$, but we emphasise that the spaces
\[ \Mbar_{g,\bfc,\upbeta}(S\, | \, D+E) \qquad \text{and} \qquad \Mbar_{g,\hat{\bfc},\upbeta}(S|D)\]
are \emph{not} the same. The difference is clearly visible for $E \subseteq \FF_1$ the $(-1)$-curve and $\upbeta$ a curve class pulled back along the morphism $\FF_1 \to \PP^2$ contracting $E$. The moduli space
\[ \Mbar_{0,n,\upbeta}(\FF_1)\]
has a large number of excess components. On the other hand the moduli space
\[ \Mbar_{0,(0,\ldots,0),\upbeta}(\FF_1|E)\]
is irreducible of the expected dimension. Although the markings carry no tangency, the logarithmic structure imposes tangency conditions at the nodes, which cut down the excess components to produce a space of the expected dimension. The Gromov--Witten invariants also differ, with the theory of $(\FF_1|E)$ being essentially equivalent to the theory of $\PP^2$.
\end{remark}

\subsection{Target degeneration}  \label{sec: degeneration target} The proof of \Cref{thm: local-log} follows the degeneration argument of \cite{vGGR}. Because our degeneration has logarithmic structure on the general fibre, much greater care is required when enumerating rigid tropical types and performing gluing. This accounts for the significantly more involved proof. Despite this, the shape of the final formula is relatively simple (\Cref{thm: local-log}), because we are able to strongly constrain the rigid tropical types which contribute.

Consider the degeneration of $S$ to the normal cone of $D$ as illustrated in Figure~\ref{fig:degeneration}. This is a family
	\begin{equation} \label{eqn: family deformation to normal cone} \scrS \to \Aone\end{equation}
	with general fibre $S$. Let $P$ denote the projective completion of the normal bundle of $D \subseteq S$. The central fibre $\scrS_0$ of \eqref{eqn: family deformation to normal cone} is obtained by gluing $S$ and $P$ along the divisors $D \subseteq S$ and the zero section $D_0 \subseteq P$. We write $D_0 \subseteq \scrS_0$ for the gluing divisor.
	
	Let $\scrE \subseteq \scrS$ denote the strict transform of $E \times \Aone$. This intersects the component $S$ of the central fibre in $E \subseteq S$, and the component $P$ of the central fibre in the union of the two fibres $E_1,E_2 \subseteq P$ of the $\PP^1$-bundle $P \to D$ over the points $\{q_1,q_2\} = D \cap E$. We equip the total space $\scrS$ with the divisorial logarithmic structure corresponding to $\scrS_0+\scrE = S + P + \scrE$.
	
	Now take the strict transform of $D \times \Aone$ under the blowup $\scrS \to S \times \Aone$ and let $\scrL$ be the inverse of the corresponding line bundle on $\scrS$. There is a flat morphism
	\[ \scrL \to \scrS \]
	which we use to pull back the logarithmic structure on $\scrS$. On the general fibre, the resulting logarithmic scheme is
	\[ (\OO_S(-D)\, |\, \hat{E}).\]
	On the central fibre we have $\scrL|_S=\OO_S$ and $\scrL|_P = \OO_P(-D_\infty)$. The central fibre $\scrL_0$ is therefore obtained by gluing
	\[ (S \times \Aone\, |\, \hat{D} + \hat{E}) \qquad \text{and} \qquad (\OO_P(-D_\infty)\, |\, \hat{D}_0+\hat{E}_1+\hat{E}_2)\]
	along the divisors $\hat{D} = D \times \Aone \subseteq S \times \Aone$ and $\hat{D}_0 = D_0 \times \Aone \subseteq \OO_P(-D_\infty)$. (See Footnote~\ref{footnote: vb notation} regarding our notation for divisors on the total space of a line bundle.)
	
		\begin{figure}[H]
		\centering
		\begin{tikzpicture}
			\node at (-0.5,1) {$\scrS\,$:};
			\node at (-0.5,3.5) {$\scrL\,$:};
			\draw (0,0) -- (3,0) -- (4,2) -- (1,2) -- (0,0);
			\draw[purple] (3.125,0.25) .. controls (3.125,0.25) and (0.4,0.2) .. (0.8,1)  .. controls (1.2,1.8) and (3.875,1.75) .. (3.875,1.75);
			\node at (2,1) {$S$};
			\node[purple] at (0.8,0.3) {$E$};
			\node at (2,3.5) {$\Ocal_S(-D)$};
			\draw[-stealth] (2,3.2) -- (2,2.5);
			\node at (4.2,0.8) {\LARGE$\rightsquigarrow$};
			\draw (5,1) -- (8,-1) -- (9,1) -- (6,3) -- (5,1);
			\draw[purple] (8,-1) -- (9,1);
			\draw (9,1) -- (11,3) -- (10,1) -- (8,-1);
			\draw[purple] (8.125,-0.75) .. controls (8.125,-0.75) and ($(5.1,1.2) + (0.21,-0.14)$) .. ($(5.35,1.9) + (0.21,-0.14)$)  .. controls ($(5.9,2.8) + (0.21,-0.14)$) and (8.875,0.75) .. (8.875,0.75);
			\draw[purple] (8.125,-0.75) -- (10.125,1.25);
			\draw[purple] (8.875,0.75) -- (10.875,2.75);
			\node at (7.3,0.8) {$S$};
			\node at (7.3,3.5) {$\Ocal_S$};
			\draw[-stealth] (7.3,3.2) -- (7.3,2.5);
			\node[purple] at (5.5,1) {$E$};
			\node[purple] at (9.5,0) {$E_1$};
			\node[purple] at (10.1,2.6) {$E_2$};
			\node at (9.5,1) {$P$};
			\node at (9.5,3.5) {$\Ocal_P(-D_\infty)$};
			\draw[-stealth] (9.5,3.2) -- (9.5,2.5);
			\node[purple] at (8.3,0.25) {$D_0$};
			\node at (11,1.9) {$D_\infty$};
			\node[purple] at (8.5,-0.9) {$q_1$};
			\node[purple] at (8.125,-0.75) {$\bullet$};
			\node[purple] at (9.05,0.55) {$q_2$};
			\node[purple] at (8.875,0.75) {$\bullet$};
		\end{tikzpicture}
		\caption{The degeneration $\scrS \to \Aone$ and the line bundle $\scrL$ on $\scrS$. The logarithmic structure is indicated in purple.}
		\label{fig:degeneration}
	\end{figure}

	\subsection{Tropical types and balancing} \label{sec: tropical types and balancing} The logarithmic morphism $\scrS \to \Aone$ induces a morphism of tropicalisations
	\[ \Sigma(\scrS) \to \Sigma(\Aone) = \R_{\geq 0}. \]
The fibre over $1 \in \R_{\geq 0}$ is a polyhedral complex which we denote $\Sigma$:
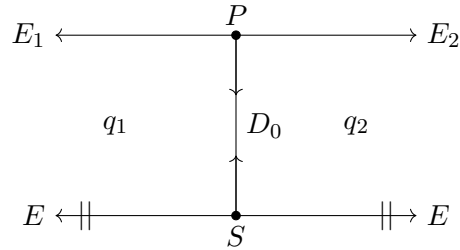
	\begin{equation}\label{eqn: central fibre tropicalisation}
		\begin{tikzpicture}[scale=0.8, baseline=(current  bounding  box.center)]
			\draw[fill=black] (0,0) circle[radius=2pt];
			\draw (0,0) node[below]{$S$};
			
			\draw[->] (0,0) -- (3,0);
			\draw (3,0) node[right]{$E$};
			
			\draw (0,0) -- (0,3);
			\draw[->] (0,0) -- (0,1);
			\draw[->] (0,3) -- (0,2);
			\draw (0,1.5) node[right]{$D_0$};
			
			\draw[fill=black] (0,3) circle[radius=2pt];
			\draw (0,3) node[above]{$P$};
			\draw[->] (0,3) -- (3,3);
			\draw (3,3) node[right]{$E_2$};
			
			\draw[->] (0,3) -- (-3,3);
			\draw (-3,3) node[left]{$E_1$};
			
			\draw[->] (0,0) -- (-3,0);
			\draw (-3,0) node[left]{$E$};
			
			
			\draw (2.5,0) node{$\mid\mid$};
			\draw (-2.5,0) node{$\mid\mid$};
			
			\draw (-2,1.5) node{$q_1$};
			\draw (2,1.5) node{$q_2$};
		\end{tikzpicture}
	\end{equation}

\begin{notation}
We use the notation $S,P,D_0,E,E_1,E_2,q_1,q_2$ to refer both to the strata of $\scrS_0$ and to the corresponding polyhedra in $\Sigma$.
\end{notation}

The moduli space of stable logarithmic maps to the central fibre $\scrL_0$ decomposes into virtual irreducible components indexed by rigid tropical types of maps to $\Sigma$.	
	\begin{definition}[{\! \cite[Definition~2.23]{ACGSDecomposition}}] A \textbf{tropical type of map to $\Sigma$} consists of:
		\begin{enumerate}
			\item \textbf{Source graph.} A finite graph $\Gamma$ with vertices $V(\Gamma)$, finite edges $E(\Gamma)$, unbounded legs $L(\Gamma)$, and a genus assignment $g \colon V(\Gamma) \to \N$.
			\item \textbf{Polyhedra assignments.} An inclusion-preserving function $\upsigma \colon V(\Gamma) \sqcup E(\Gamma) \sqcup L(\Gamma) \to \Sigma$. We often write $\ftrop(w) \in \upsigma$ instead of $\upsigma_w=\upsigma$.
			\item \textbf{Curve classes.} A curve class $\upbeta_v \in A_1(\scrS_0(\upsigma_v))$ for every $v \in V(\Gamma)$. Here $\scrS_0(\upsigma_v) \subseteq \scrS_0$ is the closed stratum corresponding to the polyhedron $\upsigma_v \in \Sigma$.
			\item \textbf{Slopes.} Vectors $m_{\vec{e}} \in N(\upsigma_e)$ for every oriented edge $\vec{e} \in \vec{E}(\Gamma)\sqcup L(\Gamma)$, satisfying $m_{\vec{e}} = -m_{\cev{e}}$. Here $N(\upsigma_e)$ is the lattice associated to the polyhedron $\upsigma_e$.
		\end{enumerate}
	\end{definition}
A tropical type has an associated tropical moduli space, parametrising choices of edge lengths and vertex positions. A tropical type is \textbf{rigid} if its tropical moduli space is a point. See \cite[Sections~2.5 and 3.2]{ACGSDecomposition} for details.

The above data is required to be balanced at each vertex. For vertices in $q_1$ and $q_2$ this means that the sum of outgoing slopes is zero. For vertices in $S$ and $P$ it is the usual balancing condition for logarithmic maps to the normal crossings pairs $(S \, | \, D+E)$ and $(P\, |\, D_0+E_1+E_2)$ as in \cite[Proposition~1.15]{GrossSiebertLog}. We now explain balancing for vertices in $E,D_0,E_1$, and $E_2$.

	\subsubsection{Vertices on $E$} \label{sec: balancing on E} First, let $v$ be a vertex of $\Gamma$ with $\mathsf{f}(v)\in E$. Then local to $v$ the tropical target $\Sigma$ has the following structure:
	
	
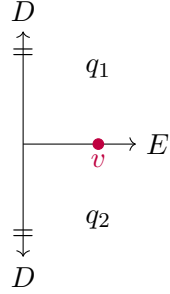
\[
	\begin{tikzpicture}
		
		\draw [<->] (0,-1.5) -- (0,1.5);
		
		\draw (0,-1.5) node[below]{$D$};
		\draw (1,-1) node{$q_2$};
		
		\draw (0,1.5) node[above]{$D$};
		\draw (1,1) node{$q_1$};
		
		\draw [->] (0,0) -- (1.5,0);
		\draw (1.5,0) node[right]{$E$};
		
		\draw[color=purple,fill=purple] (1,0) circle[radius=2pt];
		\draw [purple] (1,0) node[below]{$v$};
		
		\draw (0,-1.2) node{$=$};
		\draw (0,1.2) node{$=$};
		
	\end{tikzpicture}
	\]
	For $i \in \{1,2\}$ let $m_D^i \in \N$ denote the sum of vertical slopes of outgoing edges which enter $q_i$. Letting $m_D \in \Z$ denote the sum of vertical slopes of all outgoing edges, we have
	\[ m_D = m_D^1 + m_D^2.\]
	On the other hand, let $m_E \in \Z$ denote the total sum of horizontal slopes of all outgoing edges. The curve class $\upbeta_v \in A_1(E)$ must be of the form
	\[ \upbeta_v = kE \]
	for some $k \in \N$. We have $D \cdot \upbeta_v = kDE = 2k$ and $E \cdot \upbeta_v = kE^2$ where $E^2 \in \Z$ is the self-intersection inside $S$. The balancing condition therefore gives
	\[ m_D = 2k, \qquad m_E = kE^2.\]
	However, there is a stronger constraint. Since $f_v \colon C_v \to E$ is a degree $k$ cover, the tangency orders $m_D^1,m_D^2$ give the ramification profiles of $f_v$ over the two intersection points $q_1,q_2$. Therefore we must have
	\[ m_D^1 = m_D^2 = k.\]
	Geometrically, this means that when we lay the tropicalisation \eqref{eqn: central fibre tropicalisation} flat along the ``spine'' $E$, the sum of outgoing vertical slopes is zero.
	
	\subsubsection{Vertices on $D_0$} \label{sec: balancing on D=D0} This is similar to the previous case. Local to $v$ the tropical target $\Sigma$ has the following structure:
	
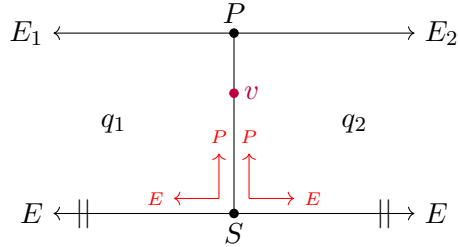
\[
	\begin{tikzpicture}[scale=0.8]
		\draw[fill=black] (0,0) circle[radius=2pt];
		\draw (0,0) node[below]{$S$};
		
		\draw[->] (0,0) -- (3,0);
		\draw (3,0) node[right]{$E$};
		
		\draw (0,0) -- (0,3);
		
		\draw[color=purple,fill=purple] (0,2) circle[radius=2pt];
		\draw[color=purple] (0,2) node[right]{$v$};
		
		\draw[red,->] (0.25,0.25) -- (1,0.25);
		\draw[red] (1,0.25) node[right]{\tiny$E$};
		\draw[red,->] (0.25,0.25) -- (0.25,1);
		\draw[red] (0.25,1) node[above]{\tiny$P$};
		
		\draw[red,->] (-0.25,0.25) -- (-1,0.25);
		\draw[red] (-1,0.25) node[left]{\tiny$E$};
		\draw[red,->] (-0.25,0.25) -- (-0.25,1);
		\draw[red] (-0.25,1) node[above]{\tiny$P$};

		
		\draw[fill=black] (0,3) circle[radius=2pt];
		\draw (0,3) node[above]{$P$};
		
		\draw[->] (0,3) -- (3,3);
		\draw (3,3) node[right]{$E_2$};
		
		\draw[->] (0,3) -- (-3,3);
		\draw (-3,3) node[left]{$E_1$};
		
		\draw[->] (0,0) -- (-3,0);
		\draw (-3,0) node[left]{$E$};
		
		
		\draw (2.5,0) node{$\mid\mid$};
		\draw (-2.5,0) node{$\mid\mid$};
		
		\draw (-2,1.5) node{$q_1$};
		\draw (2,1.5) node{$q_2$};
	\end{tikzpicture}
	\]
	We choose coordinates for the adjacent polyhedra $q_1,q_2$ as indicated in red. For $i \in \{1,2\}$ we let $m_E^i \in \N$ denote the sum of horizontal slopes of outgoing edges which enter $q_i$. Letting $m_E \in \Z$ denote the sum of horizontal slopes of all outgoing edges, we have
	\[ m_E = m_E^1 + m_E^2.\]
	On the other hand, let $m_P \in \Z$ denote the sum of vertical slopes of all outgoing edges (note that if instead we choose coordinates corresponding to $S,E$ then we have $m_S=-m_P$). The curve class $\upbeta_v \in A_1(D_0)$ necessarily takes the form
	\[ \upbeta_v = kD_0 \]
	for some $k \in \N$. We have $E \cdot \upbeta_v = 2k$ and $\deg f_v^\star N^\vee_{D_0|P} = \deg f_v^\star N_{D|S} = kD^2$ where $D^2 \geq 0$ is the self-intersection inside $S$. The balancing condition is therefore
	\[ m_E = 2k, \qquad m_P = kD^2.\]
	As in the previous case, we have the stronger constraint
	\[ m_E^1 = m_E^2 = k.\]
	This means that when we lay the tropicalisation \eqref{eqn: central fibre tropicalisation} flat along the spine $D_0$, the sum of outgoing horizontal slopes is zero.
	
	\subsubsection{Vertices on $E_1,E_2$} \label{sec: balancing on E1 E2} This is the simplest case. We restrict to $E_1$ without loss of generality. Local to $v$ there is a single maximal polyhedron $q_1$, coordinatised by $E_1,D_0$:
	
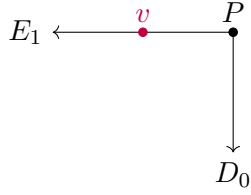
\[
	\begin{tikzpicture}[scale=0.8]
		\draw[->] (0,3) -- (0,1);
		\draw (0,1) node[below]{$D_0$};
		
		\draw[fill=black] (0,3) circle[radius=2pt];
		\draw (0,3) node[above]{$P$};
		
		\draw[->] (0,3) -- (-3,3);
		\draw (-3,3) node[left]{$E_1$};
		
		\draw[color=purple,fill=purple] (-1.5,3) circle[radius=2pt];
		\draw[purple] (-1.5,3) node[above]{$v$};
		
		\draw (-1.5,2) node{$q_1$};
		
	\end{tikzpicture}	
	\]
	We let $m_{E_1} \in \Z$ and $m_{D_0} \in \N$ denote, respectively, the sums of horizontal and vertical outgoing slopes. The curve class is $\upbeta_v = kE_1$ giving $E_1 \cdot \upbeta_v=0$ and $D_0 \cdot \upbeta_v = k$. The balancing condition gives:
	\[ m_{E_1} = 0, \qquad m_{D_0} = k.\]
	Unlike the previous cases there is no stronger constraint, as there is only one adjacent polyhedron.
	
	\subsection{Degeneration formula analysis}\label{sec: degeneration formula analysis} The blowup morphism $p \colon \scrS \to S \times \Aone$ satisfies $p^{-1}(E \times \Aone) = \scrE$. Restricting to the central fibre, we obtain a logarithmic morphism $\scrS_0 \to (S|E)$. Combined with the line bundle projection $\scrL_0 \to \scrS_0$ this induces a pushforward morphism
	\[ \uprho \colon \Mbar_{g,\mathbf{c},\upbeta}(\scrL_0) \rightarrow \Mbar_{g,\mathbf{c},\upbeta}(S|E).\]
	Moreover, let $\frgt$ denote the morphism forgetting the logarithmic structures
	\[ \frgt : \Mbar_{g,\mathbf{c},\upbeta}(S|E) \rightarrow \MapsRefinedEvals \]
	where the codomain is the moduli space of stable maps to $S$ with \textbf{refined evaluations}:
	\[ \MapsRefinedEvals \colonequals \Mbar_{g,s,\upbeta}(S) \times_{S^s} \Pi_{j=1}^s \mathrm{Ev}_j.\]
	The conservation of number principle \cite[Theorem~1.1]{ACGSDecomposition} and the decomposition theorem \cite[Theorem~1.2]{ACGSDecomposition} give the following identity in the Chow homology of $\Mbar_{g,\mathbf{c},\upbeta}(S|E)$:
\[
		[\Mbar_{g,\mathbf{c},\upbeta}(\calO_S(-D)|\hat{E})]^{\virt} = \sum_{\uptau} \frac{m_{\uptau}}{|\mathrm{Aut}(\uptau)|} \uprho_\star  \upiota_\star [\Mbar_\uptau]^{\virt}.
\]
The sum runs over rigid tropical types $\uptau$ of tropical stable maps to $\Sigma$ of type $(g,\bfc,\upbeta)$. Here $m_{\uptau}$ is the smallest integer such that scaling $\Sigma$ by $m_\uptau$ produces a tropical stable map with integral vertex positions and edge lengths, and  $\upiota \colon \Mbar_\uptau \hookrightarrow \Mbar_{g,\mathbf{c},\upbeta}(\scrL_0)$ is the inclusion of the virtual irreducible component. Capping with the $D$-disjoint insertions $\upgamma$ and pushing forward along $\frgt$ we obtain:
\begin{equation}\label{eqn: decomposition formula} 
		\frgt_\star \big(\upgamma \cap [\Mbar_{g,\mathbf{c},\upbeta}(\calO_S(-D)|\hat{E})]^{\virt}\big) = \sum_{\uptau} \frac{m_{\uptau}}{|\mathrm{Aut}(\uptau)|}  \frgt_\star \big(\upgamma \cap \uprho_\star \upiota_\star [\Mbar_\uptau]^{\virt}\big).	\end{equation}

We now use the degeneration formula as formulated in \cite[Section~6]{RangExpansions} to describe the terms in the sum. Let $\Gamma$ denote the source graph of the tropical type $\uptau$. There is a subdivision of the product
	\[ {\bigtimes_{v \in V(\Gamma)}} \Mbar_v \to {\prod_{v\in V(\Gamma)}} \Mbar_v \]
	and for each edge $e \in E(\Gamma)$ a smooth universal divisor $\hat{\Dcal}_e \to {\bigtimes_v} \Mbar_v$ supporting an evaluation section for each flag $(v \in e)$. We consider the universal doubled divisor
	\[ \hat{\Dcal}_e^{\{2\}} \colonequals \hat{\Dcal}_e \times_{{\bigtimes_v}\Mbar_v} \hat{\Dcal}_e. \]
	There is a universal diagonal $\Delta \colon \hat{\Dcal}_e \hookrightarrow \hat{\Dcal}_e^{\{2\}}$ which is a regular embedding since $\hat{\Dcal}_e \to {\bigtimes_v} \Mbar_v$ is smooth. We obtain a diagram
\begin{equation} \label{eqn: original gluing diagram universal divisor}
	\begin{tikzcd}
	\Mbar_\uptau \ar[r,"\upnu"] & \Nbar_{\uptau} \ar[r] \ar[d] \ar[rd,phantom,"\square"] & {\bigtimes_v} \Mbar_v \ar[d] \\
	& {\prod_e} \hat{\Dcal}_e \ar[r,hook,"\Delta"] & {\prod_e} \hat{\Dcal}_e^{\{2\}}		
	\end{tikzcd}
\end{equation}
	with an equality of classes in the Chow homology of $\Nbar_{\uptau}$
	\[ \upnu_\star [\Mbar_\uptau]^{\virt} = c_\uptau \cdot \Delta^! [\textstyle{\bigtimes_v} \Mbar_v]^{\virt}\]
	where $c_\uptau\in \Q$ is an appropriate combinatorial gluing factor. There is a gluing morphism $\uptheta: \Nbar_{\uptau} \rightarrow \MapsRefinedEvals$ which forgets the logarithmic structures, glues the domain curves indexed by the vertices of $\Gamma$ along the markings as specified by the edges of $\Gamma$, composes the map with the morphism to $S$ and stabilises the domain. This morphism makes the following diagram commute:
\begin{equation}
	\label{eq: guling morphism diagram}
	\begin{tikzcd}
		\Mbar_\uptau \arrow[d, "\upiota"] \arrow[r, "\upnu"] & \Nbar_{\uptau} \arrow[d,"\uptheta"]  \\
		\Mbar_{g,\mathbf{c},\upbeta} (\scrL_0) \ar[r,"\frgt \circ \uprho"] & \MapsRefinedEvals.
	\end{tikzcd}
\end{equation}
We thus rewrite the terms appearing on the right-hand side of equation \eqref{eqn: decomposition formula} as
\begin{equation}
	\label{eq: term degeneration formula rewritten}
	\frgt_\star \big(\upgamma \cap \uprho_\star \upiota_\star [\Mbar_\uptau]^{\virt}\big) = c_\uptau \cdot (\upgamma \cap \uptheta_\star \Delta^! [\textstyle{\bigtimes_v} \Mbar_v]^{\virt}).
\end{equation}

\subsubsection{Reduction outline} We will prove \Cref{thm: local-log} by analysing the contributions of rigid tropical types $\uptau$ to the degeneration formula \eqref{eqn: decomposition formula}. For the rest of \Cref{sec: degeneration formula analysis} we fix a rigid tropical type $\uptau$ whose contribution to \eqref{eqn: decomposition formula} is non-trivial:
\[ \frgt_\star \big(\upgamma \cap \uprho_\star \upiota_\star [\Mbar_\uptau]^{\virt}\big) \neq 0.\]
 We will show that the shape of $\uptau$ is tightly constrained. We proceed via four reductions: 
 \begin{itemize}
 \item In \Cref{sec: first vanishing} we show that $\uptau$ is weakly star-shaped (\Cref{thm: first vanishing}). As we will see the reason for the vanishing of all other rigid tropical types is that the line bundle $\scrL_0$ trivialises over the component $S \subset\scrS_0$.
 \item In \Cref{sec: second vanishing} we show that $\uptau$ has markings confined to $S$ (\Cref{thm: second vanishing}). This step crucially uses the assumption that all insertions are $D$-avoidant and that their codimension equals the virtual dimension of the moduli problem. Moreover, we require a fact about the balancing at $E$ which is specific to our geometric setup (see \Cref{rmk: balancing at E}).
 \item In \Cref{sec: third vanishing} we show that $\uptau$ is star-shaped (\Cref{thm: third vanishing}). This step is purely combinatorial and constrains the shape of $\uptau$ via a careful analysis of tropical balancing. The latter is of course sensitive to the geometric setup.
 \item In \Cref{sec: fourth vanishing} we show that $\uptau$ has a single edge (\Cref{thm: fourth vanishing}). To prove the final reduction we again crucially require the codimension of our insertions to coincide with the virtual dimension of the moduli problem. The assumption that $D$ is rational is only used once in \Cref{prop: only one edge} for solely technical reasons.
 \end{itemize}
 Finally in \Cref{sec: evaluating the integrals} we calculate the contributions of the remaining $\uptau$, arriving at \Cref{thm: local-log}.

\begin{remark}
As in \cite{BFGW} we obtain a reduction to star-shaped graphs (\Cref{thm: third vanishing}). In fact we reduce further to single edge graphs (\Cref{thm: fourth vanishing}), using crucially the fact that $D$ is rational.
\end{remark}
 
\subsubsection{First reduction: weakly star-shaped graphs} \label{sec: first vanishing}
	
	\begin{proposition}[First reduction] \label{thm: first vanishing} The tropical type $\uptau$ is \emph{weakly star-shaped} in the following sense. The source graph $\Gamma$ decomposes into subgraphs
		
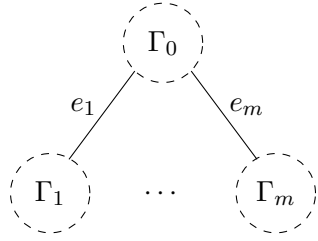
\[
		\begin{tikzpicture}
			
			\draw[dashed] (0,0) circle[radius=15pt];
			\draw (0,0) node{$\Gamma_0$};
			
			\draw (-0.375,-0.375) -- (-1.25,-1.55);
			\draw[dashed] (-1.5,-2) circle[radius=15pt];
			\draw (-0.75,-0.85) node[left]{$e_1$};
			\draw (-1.5,-2) node{$\Gamma_1$};
			
			\draw (0.375,-0.375) -- (1.25,-1.55);
			\draw[dashed] (1.5,-2) circle[radius=15pt];
			\draw (0.75,-0.85) node[right]{$e_m$};
			\draw (1.5,-2) node{$\Gamma_m$};
			
			\draw (0,-2) node{$\cdots$};
		\end{tikzpicture}
		\]
		with the vertices of $\Gamma_0$ mapping to $P,E_1,E_2$ and the vertices of $\Gamma_1,\ldots,\Gamma_m$ mapping to $S,E,D_0,q_1,q_2$.
	\end{proposition}
	
\begin{proof}
	Let $\Gamma^\prime \subseteq \Gamma$ be a maximal connected subgraph contained in the union of the following strata
		\begin{equation} \label{eqn: bottom two third strata} S,E, D_0,q_1,q_2 .\end{equation}
		Let $E^\prime \subseteq E(\Gamma)$ denote the set of edges connecting $\Gamma^\prime$ to the rest of $\Gamma$. It is sufficient to show that $|E^\prime|=1$. The following argument parallels \cite[Lemma~3.1]{vGGR}.
		
The line bundle $\scrL_0$ is trivial when restricted to each of the strata in \eqref{eqn: bottom two third strata}. It follows from the definition of $\Gamma^\prime$ that for each $e \in E^\prime$ the corresponding universal divisor $\hat{\Dcal}_e$ decomposes as
		\[ \hat{\Dcal}_e = \Dcal_e \times \Aone \]
		where $\Dcal_e$ is the universal divisor for the compact degeneration $\scrS_0$ (since $\scrL_0 \to \scrS_0$ is flat and strict, all expansions of $\scrL_0$ are pulled back from expansions of $\scrS_0$). We define
		\[	\Dcal \colonequals \prod_{e \in E(\Gamma) \setminus E^\prime} \hat{\Dcal}_e \times \prod_{e \in E^\prime} \Dcal_e \]
and note that the product of universal divisors appearing in \eqref{eqn: original gluing diagram universal divisor} decomposes as $\Pi_{e \in E(\Gamma)} \hat{\Dcal}_e = \Dcal \times \Aaff^{\! E^\prime}$. We now examine evaluations at the nodes corresponding to edges $e \in E^\prime$.

Given $e \in E^\prime$ the corresponding node has an adjacent irreducible component $C_e$ which is mapped to $P$ and has positive intersection with the divisor $D_0$. Since $D^2 \geq 0$ in $S$ it follows that the pullback of the line bundle $\scrL_0|_P = \OO_P(-D_\infty)$ to this component has negative degree. We conclude that evaluation of $f|_{C_e}$ at the given node factors through the zero section of $\scrL_0$.

On the other hand, the subcurve $C^\prime$ corresponding to $\Gamma^\prime$ is connected and maps to $S$. Since $\scrL_0|_S = \OO_S$ it follows that this subcurve is mapped to a constant section of the bundle $S \times \Aone \to S$. Therefore the evaluations of $f|_{C^\prime}$ at the nodes corresponding to $e \in E^\prime$ all coincide. 

Taken together, we obtain the following cartesian diagram extending \eqref{eqn: original gluing diagram universal divisor}:
		\bcd
		\Nbar_{\uptau} \ar[r] \ar[d] \ar[rd,phantom,"\square",start anchor=center,end anchor=center] & {\bigtimes_v} \Mbar_v \ar[d] \\
		\Dcal \times \Azero \ar[r,hook] \ar[d] \ar[rd,phantom,"\square",start anchor=center,end anchor=center] & \Dcal^{\{2\}} \times \Aone \ar[d] \\
		\Dcal \times \Aaff^{\! E^\prime} \ar[r,hook,"\Delta"] & \Dcal^{\{2\}} \times \Aaff^{\! 2E^\prime}.
		\ecd
		If $|E^\prime| \geq 2$ the excess bundle is a trivial vector bundle of positive rank. Consequently the excess class \cite[Theorem~6.3]{FultonBig} vanishes, and we obtain
		\[ \Delta^![\textstyle{\bigtimes_v}\Mbar_v]^{\virt}=0.\]
		From \eqref{eq: term degeneration formula rewritten} we conclude that $|E^\prime|=1$.\end{proof}
	\subsubsection{Second reduction: markings confined to $S$}\label{sec: second vanishing} The next step is to constrain the shape of the subgraphs $\Gamma_i$ and to confine the unbounded legs.
	\begin{proposition}[Second reduction] \label{thm: second vanishing} Consider a weakly star-shaped tropical type as in \Cref{thm: first vanishing}. Then:
	\begin{enumerate}
	\item The subgraph $\Gamma_i$ is a tree for $i \in \{1,\ldots,m\}$.
	\item Every leaf vertex $v \in V(\Gamma_i)$ satisfies $\ftrop(v) \in S$.
	\item Every marking leg $l \in L(\Gamma)$ is attached to a leaf vertex $v \in V(\Gamma_i)$ for some $i \in \{1,\ldots,m\}$. 
	\end{enumerate}
	\end{proposition}
	
	\begin{proof} Combine Propositions \ref{prop: Gamma_i tree}, \ref{lem: all leaves map to S}, and \ref{prop: v0 no markings} below. \end{proof}
	
	
	The proof proceeds by a sequence of intermediate reductions. We fix a weakly star-shaped tropical type $\uptau$ as in \Cref{thm: first vanishing}.
	
	\begin{notation}\label{notation: vertices e_i} For each $i \in \{1,\ldots,m\}$ we let $v_i \in V(\Gamma_i)$ and $w_i \in V(\Gamma_0)$ denote the endpoints of $e_i$.\end{notation}
	
	\begin{proposition} \label{prop: Gamma_i tree} Each of the graphs $\Gamma_1,\ldots,\Gamma_m$ is a tree.
	\end{proposition}
	\begin{proof} 
		Recall that an edge of a connected graph is \textbf{separating} if deleting it produces a graph with two connected components, and that a graph is a tree if and only if every edge is separating.
		
		Fix $i \in \{1,\ldots,m\}$. We will prove that every edge of $\Gamma_i$ is separating by inducting on the vertices. At each step we prove that the edges adjacent to the current vertex are separating. We pass to the next step by traversing along all adjacent edges, excluding the edge we arrived by. The starting point is the vertex $v_i$ which is connected to $w_i$ along the separating edge $e_i$. When we arrive at a new vertex, there is by induction a path of separating edges connecting it to $w_i$.
		
		Suppose that we arrive at a vertex $v \in V(\Gamma)$ via a separating edge $\tilde{e}$. Let $\tilde{e}_1,\ldots,\tilde{e}_r$ be the other adjacent edges and suppose for a contradiction that $\tilde{e}_1$ is not separating. Since $\tilde{e}$ is separating, the subgraph behind $\tilde{e}_1$ must coincide (without loss of generality) with the subgraph behind $\tilde{e}_2$.
		
		Split the graph $\Gamma$ at the edges $\tilde{e}_1,\tilde{e}_2$. This produces a new combinatorial type $\uptau_{12}$ with four open half-edges corresponding to the previously closed edges $\tilde{e}_1,\tilde{e}_2$. For $i \in \{1,2\}$ let $\hat{\Dcal}_i$ denote the corresponding universal divisor. As in \Cref{sec: first vanishing} this decomposes as $\hat{\Dcal}_{i} = \Dcal_i \times \Aone$ and we have a diagram
		\bcd
		\Mbar_\uptau \ar[r,"\upnu"] & \Nbar_{\uptau} \ar[r] \ar[d] \ar[rd,phantom,"\square"] & \Mbar_{\uptau_{12}} \ar[d] \\
		& \prod_{i=1}^2 (\Dcal_i \times \Aone) \ar[r,hook,"\Delta"] & \prod_{i=1}^2 ( \Dcal_i^{\{2\}} \times \Aaff^{\! 2}).
		\ecd
		We now show that the composite $\Mbar_{\uptau_{12}} \to \Pi_{i=1}^2 ( \Dcal_i^{\{2\}} \times \Aaff^{\! 2}) \to \A^{4}$ 	factors through the linear subspace:
		\begin{align*}
			\upepsilon \colon \Aone & \hookrightarrow \Aaff^{\! 4} \\
			t & \mapsto (0,t,0,t).
		\end{align*}
		The inductive argument connects the current vertex $v$ to the vertex $w_i$ via a path of separating edges, which are hence distinct from $\tilde{e}_1$ and $\tilde{e}_2$. The line bundle $f^\star \scrL_0|_{C_{w_i}}$ is negative and so $f|_{C_{w_i}}$ factors through the zero section of $\scrL_0$. Since $C_{w_i}$ is connected to $C_v$ through nodes which are not split in $\uptau_{12}$ it follows that $f|_{C_v}$ also factors through the zero section of $\scrL_0$. This explains the two zero entries. On the other hand, the two $t$ entries occur because the subgraphs behind $\tilde{e}_1$ and $\tilde{e}_2$ coincide. We thus obtain
		\bcd
		\Mbar_\uptau \ar[r,"\upnu"] & \Nbar_{\uptau} \ar[r] \ar[d] \ar[rd,phantom,"\square",start anchor=center,end anchor=center] & \Mbar_{\uptau_{12}} \ar[d] \\
		& \left(\prod_{i=1}^2 \Dcal_i \right) \times \Azero \ar[r,hook] \ar[d] \ar[rd,phantom,"\square",start anchor=center,end anchor=center] & \left( \prod_{i=1}^2 \Dcal_i^{\{2\}} \right) \times \Aone \ar[d,hook,"\id \times \upepsilon"] \\
		& \left( \prod_{i=1}^2 \Dcal_i \right) \times \Aaff^{\! 2} \ar[r,hook,"\Delta"] & \left( \prod_{i=1}^2 \Dcal_i^{\{2\}} \right) \times \Aaff^{\! 4}.
		\ecd
		Since the codimensions of the lower two horizontal arrows differ, the excess bundle 
		has rank one. Moreover, we see that this bundle must be trivial as it is obtained by pulling back the excess bundle on \smash{$\Azero = \Aaff^{\! 2} \times_{\Aaff^4} \Aone$} along the projection $\smash{(\prod_{i=1}^2 \Dcal_i ) \times \Azero \rightarrow \Azero}$. Hence, the excess class vanishes and, as in the proof of \Cref{thm: first vanishing}, it follows that the contribution of $\uptau$ vanishes. We conclude that the edges $\tilde{e}_1,\ldots,\tilde{e}_r$ adjacent to $v$ are all separating, and this completes the induction step.
	\end{proof}
	
	The arguments now shift from intersection theory to tropical geometry. The background developed in \Cref{sec: tropical types and balancing} is essential. 
	
	By \Cref{prop: Gamma_i tree}, each $\Gamma_i$ is a tree equipped with a root vertex $v_i$. This defines a canonical flow starting at $v_i$. From now on we orient each edge of $\Gamma_i$ according to this flow. The edge $e_i$ is also oriented from $w_i$ to $v_i$. In this way, every vertex of $\Gamma_i$ has a unique incoming edge.
		
		A leaf of $\Gamma_i$ is a vertex adjacent to a single finite edge. Notice that $v \in V(\Gamma_i)$ is a leaf if and only it does not support any outgoing finite edges. Enumerate the leaf vertices in $V(\Gamma_1) \sqcup \ldots \sqcup V(\Gamma_m)$ as
		\[ \ol{v}_1,\ldots,\ol{v}_{\ell} \]
		and write $\ol{e}_j$ for the edge incoming at $\ol{v}_j$.
	
	\begin{lemma} \label{lem: negative vertical slope} Let $e \in E(\Gamma_i) \cup \{ e_i\}$, oriented as above. Suppose that $\ftrop(e)$ is contained in $q_1$,$q_2$, or $D_0$. Then the vertical slope of $e$ is negative.
	\end{lemma}
	
	\begin{proof} Suppose first that $e$ has positive vertical slope. Since $\Gamma_i$ has only a single incoming edge and no outgoing edges, the edge $e$ leads to another vertex of $\Gamma_i$, which must be contained in $q_1,q_2$, or $D_0$. Since the marking legs have zero vertical slope, the balancing condition ensures that there is a finite outgoing edge with positive vertical slope (see in particular \Cref{sec: balancing on D=D0}). Continuing in this way, we produce a path in $\Gamma_i$ consistent with the flow and with positive vertical slope along each edge. This path continues indefinitely, a contradiction.
		
		It remains to consider the case where $e$ has zero vertical slope. Let $\Gamma_e$ denote the subgraph of $\Gamma_i$ behind the oriented edge $e$. By the previous paragraph, no edge of $\Gamma_e$ has positive vertical slope. Following the flow, we see by induction and balancing that no edge of $\Gamma_e$ has negative vertical slope either. Therefore every edge of $\Gamma_e$ has zero vertical slope, and every vertex and edge is mapped to $q_1,q_2$, or $D_0$.
		
		If $e$ has zero horizontal slope then it is contracted by the tropical map $\ftrop$, in which case the tropical type is not rigid. If $e$ has non-zero horizontal slope, we may traverse $\Gamma_e$ using the balancing condition. It is easy to see that eventually we arrive at a vertex $v$ in $q_1$ or $q_2$ around which the image of $\Gamma_e$ takes the following form:
		
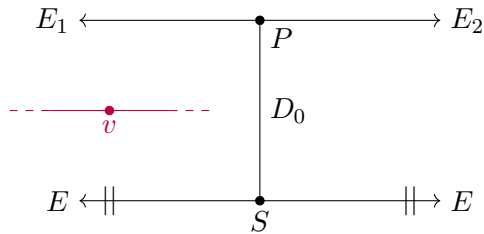
\[
		\begin{tikzpicture}[scale=0.8]
			
			
			\draw[fill=black] (-8,0) circle[radius=2pt];
			\draw (-8,0) node[below]{$S$};
			
			\draw[->] (-8,0) -- (-5,0);
			\draw (-5,0) node[right]{$E$};
			
			\draw (-8,0) -- (-8,3);
			\draw (-8,1.5) node[right]{$D_0$};
			
			\draw[fill=black] (-8,3) circle[radius=2pt];
			\draw (-8,2.7) node[right]{$P$};
			
			\draw[->] (-8,3) -- (-5,3);
			\draw (-5,3) node[right]{$E_2$};
			
			\draw[->] (-8,3) -- (-11,3);
			\draw (-11,3) node[left]{$E_1$};
			
			\draw[->] (-8,0) -- (-11,0);
			\draw (-11,0) node[left]{$E$};
			
			\draw (-5.5,0) node{$\mid\mid$};
			\draw (-10.5,0) node{$\mid\mid$};
			
			\draw[purple] (-9.5,1.5) -- (-11.5,1.5);
			\draw[purple, dashed] (-11.5,1.5) -- (-12.25,1.5);
			\draw[purple, dashed] (-9.5,1.5) -- (-8.75,1.5);
			\draw[purple,fill=purple] (-10.5,1.5) circle[radius=2pt];
			\draw[purple] (-10.5,1.5) node[below]{$v$};


		\end{tikzpicture}
		\]
		Varying the position of $v$ horizontally we see that the tropical type is not rigid.
	\end{proof}

	\begin{lemma} \label{lem: no vertex on E} There is no vertex $v \in V(\Gamma_i)$ with $\ftrop(v) \in E$.
	\end{lemma}
	
	\begin{proof} Suppose for a contradiction that such a vertex $v$ exists. If it has an adjacent edge with non-zero vertical slope, then by balancing (\Cref{sec: balancing on E}) it has at least two adjacent edges with non-zero vertical slope. At most one of these can be the incoming edge, and so at least one is outgoing, i.e. an edge oriented according to the flow with positive vertical slope. This contradicts \Cref{lem: negative vertical slope}. We conclude that all edges adjacent to $v$ have zero vertical slope. It follows immediately that the tropical type is not rigid, as in the proof of \Cref{lem: negative vertical slope}.
	\end{proof}
	
	\begin{remark}
	\label{rmk: balancing at E}
	Since we will resort to the above reduction lemma several times in our subsequent analysis, let us stress that its proof crucially uses the fact that by balancing there can never be a vertex mapping to $E$ adjacent to a single edge with non-zero vertical slope. In our case at hand this is ensured by the fact that $E$ is a curve intersecting $D$ in \textit{at least} two points. One can see that the statement of \Cref{lem: no vertex on E} is generally wrong if $D$ and $E$ intersect in only one point. A counterexample can be found when considering $\mathbb{P}^2$ relative to two lines.
	\end{remark}
	
	\begin{lemma}
		\label{lem: all leaves map to S}
		A vertex $v\in V (\Gamma_i)$ is a leaf if and only if $\ftrop(v) \in S$.
	\end{lemma}
	
	\begin{proof} Suppose $v\in V (\Gamma_i)$ is a leaf. We already argued in \Cref{lem: no vertex on E} that $\ftrop(v) \notin E$. By balancing $v$ cannot map to $q_1$ or $q_2$ since by \Cref{lem: negative vertical slope} we only have incoming edges with strictly positive vertical slope. Suppose $\ftrop(v) \in D_0$. Then the horizontal slope of the incoming edge $e$ must be zero, since otherwise there would be at least one outgoing edge by balancing. We conclude that both $v$ and $e$ are contained in $D_0$ which, however, contradicts the assumption that $\uptau$ is rigid. The only remaining possibility is $\ftrop(v) \in S$.
		
		For the opposite direction, suppose $v\in V(\Gamma_i)$ maps to $S$ but has at least one outgoing finite edge. This edge cannot be contracted (due to rigidity) or flow to a vertex in $E$ (due to \Cref{lem: no vertex on E}). Therefore it must map to $q_1,q_2,$ or $D_0$. In particular it has positive vertical slope, contradicting \Cref{lem: negative vertical slope}.
	\end{proof}
	
	We now switch from tropical geometry back to intersection theory. We will show that all marking legs are adjacent to vertices mapping to $S$. From now on we focus on the entire graph $\Gamma$ rather than the subgraph $\Gamma_i$.
	
	Split $\Gamma$ at the edges $\overline{e}_1,\ldots,\overline{e}_\ell$ incoming to the leaves in $V(\Gamma_1) \sqcup \ldots \sqcup V(\Gamma_m)$. These edges are separating, and we obtain split tropical types:
	\[ \uptau_0,\uptau_1,\ldots,\uptau_\ell.\]
	Each of the types $\uptau_1,\ldots,\uptau_\ell$ constitutes a single vertex supporting a single finite half-edge and a collection of marking legs:
	\[
	\begin{tikzpicture}
		
		\draw[dashed] (0,0) circle[radius=15pt];
		\draw (0,0) node{$\bullet$};
		\draw (0,0.265) node{$\uptau_0$};
		
		\draw (0,0) -- (-1.5,-2);
		\draw[dashed] (-1.5,-2) circle[radius=15pt];
		\draw (-0.75,-0.85) node[left]{$\ol{e}_1$};
		\draw (-1.5,-2) node{$\bullet$};
		\draw (-1.5,-2.265) node{$\uptau_1$};
		
		\draw (0,0) -- (1.5,-2);
		\draw[dashed] (1.5,-2) circle[radius=15pt];
		\draw (0.75,-0.85) node[right]{$\ol{e}_\ell$};
		\draw (1.5,-2) node{$\bullet$};
		\draw (1.5,-2.265) node{$\uptau_\ell$};
		
		\draw (0,-2) node{$\cdots$};
	\end{tikzpicture}
	\]
	
	Passing to a subdivision of $\Sigma$ we may assume that each edge $\overline{e}_i$ is contained in a one dimensional polyhedron (whose corresponding divisor we denote $D_{\overline{e}_i}$) and that the vertices adjacent to $\overline{e}_i$ map to zero-dimensional polyhedral. 
	
	The arguments of \cite[Sections~6.5.2--6.5.3]{RangExpansions} then apply: since the gluing divisors inherit logarithmic structures with one dimensional tropicalisations, all their logarithmic modifications are trivial and hence the gluing takes place over the unexpanded diagonal. The following is an immediate consequence.
	\begin{lemma}\label{lem:fibre product unexpanded divisor} There is a map $\upnu \colon \Mbar_\uptau \to \Nbar_{\uptau}$ with target the fibre product (in the category of schemes) over the unexpanded diagonal 
		\begin{equation*}
			\begin{tikzcd}
				\Mbar_\uptau \arrow[r, "\upnu"] & \Nbar_\uptau \arrow[r] \arrow[d] \ar[rd,phantom,"\square"] & \Mbar_{\uptau_0} \times \prod_{i=1}^{\ell} \Mbar_{\uptau_i} \arrow[d] \\
				& \prod_{i=1}^{\ell} \hat{D}_{\overline{e}_i} \arrow[r,"\Delta",hook] & \prod_{i=1}^{\ell} \hat{D}_{\overline{e}_i}^2
			\end{tikzcd}
		\end{equation*}
	and an equality of classes in the Chow homology of $\Nbar_{\uptau}$:
	\[ \upnu_\star [\Mbar_\uptau]^{\virt} = c_\uptau \cdot \Delta^! \left( [\Mbar_{\uptau_0}]^{\virt} \times \prod_{i=1}^{\ell} [\Mbar_{\uptau_i}]^{\virt}\right).\]
	\end{lemma}
	
	As in the proof of \Cref{thm: first vanishing} we have $\hat{D}_{\overline{e}_i} = D_{\overline{e}_i} \times \Aone$ for all edges incoming to a leaf, and the evaluation morphism
	\[
	\begin{tikzcd}
	\Mbar_{\uptau_0} \ar[r, "\ev_0"] & \prod_{i=1}^{\ell} \big(D_{\overline{e}_i} \times \Aone\big)
	\end{tikzcd}
	\]
	factors through the codimension-$\ell$ subvariety:
	\[
	\begin{tikzcd}
		\Mbar_{\uptau_0} \ar[r,"\ol{\ev}_0"] \ar[rr,bend left=25pt,"\ev_0"] & \prod_{i=1}^{\ell} D_{\overline{e}_i} \ar[r,hook,"\upepsilon"] & \prod_{i=1}^{\ell} \big(D_{\overline{e}_i} \times \Aone\big).
	\end{tikzcd}
	\]
	For $i\in\{1,\ldots,\ell\}$ the moduli space $\Mbar_{\uptau_i}$ parametrises logarithmic maps to some subdivision of $(S \times \Aone \, | \, \hat{D}+\hat{E})$ of type $\uptau_i$. There is a closed embedding
	\[ \begin{tikzcd} \Mbar_{\uptau_i}\lvert_{0} \ar[r,hook] & \Mbar_{\uptau_i} \end{tikzcd}\]
	parametrising logarithmic maps which factor through the zero section. Letting $\updelta$ denote the diagonal $\prod_{i=1}^{\ell} D_{\overline{e}_i}  \hookrightarrow \prod_{i=1}^{\ell} D_{\overline{e}_i}^2$ we obtain the following diagram
	\begin{equation}\label{eqn: big cartesian diagram fourth vanishing}
		\begin{tikzcd}
			\Nbar_{\uptau} \arrow[r] \arrow[d] \ar[rd,phantom,"\square",start anchor=center,end anchor=center] & \Mbar_{\uptau_0} \times \prod_{i=1}^{\ell} \Mbar_{\uptau_i}\lvert_{0}  \arrow[r] \arrow[d] \ar[rd,phantom,"\square",start anchor=center,end anchor=center] & \Mbar_{\uptau_0} \times \prod_{i=1}^{\ell} \Mbar_{\uptau_i} \arrow[d, "\ol{\ev}_0 \times \Pi_{i=1}^\ell \ev_i"] \\
			\prod_{i=1}^{\ell} D_{\overline{e}_i}  \arrow[r, "\updelta"] \arrow[d]  & \prod_{i=1}^{\ell} D_{\overline{e}_i} \times \prod_{i=1}^{\ell} D_{\overline{e}_i}  \arrow[r, "\id \times \upepsilon"] \ar[d,phantom,"\square" right] & \prod_{i=1}^{\ell} D_{\overline{e}_i} \times \prod_{i=1}^{\ell} (D_{\overline{e}_i} \times \Aaff^{\!1}) \arrow[d, "\upepsilon \times \id"] \\
			\prod_{i=1}^{\ell} (D_{\overline{e}_i} \times \Aaff^{\!1}) \arrow[rr, "\Delta"] & \phantom{.} & \prod_{i=1}^{\ell} (D_{\overline{e}_i} \times \Aaff^{\!1})^2.
		\end{tikzcd}
	\end{equation}
	The following explains the appearance of lambda classes in \Cref{thm: local-log}.
	\begin{lemma}\label{prop: simplification of contribution}
		We have
		\begin{equation*}
			\label{eq:simplification term decomposition formula}
			\upnu_\star [\Mbar_\uptau]^{\virt} = c_\uptau \cdot \updelta^! \left( [\Mbar_{\uptau_0}]^{\virt} \times \prod_{i=1}^{\ell} (-1)^{g_i} \uplambda_{g_i} \cap [\Mbar_{\uptau_i}\lvert_{0}]^{\virt} \right).
		\end{equation*}
	\end{lemma}
	
	\begin{proof} From \eqref{eqn: big cartesian diagram fourth vanishing} we obtain
		\[	\Delta^! \left( [\Mbar_{\uptau_0}]^{\virt} \times \prod_{i=1}^{\ell} [\Mbar_{\uptau_i}]^{\virt} \right) = \updelta^! (\id \times \upepsilon)^! \left([\Mbar_{\uptau_0}]^{\virt} \times  \prod_{i=1}^{\ell} [\Mbar_{\uptau_i}]^{\virt} \right) = \updelta^! \left( [\Mbar_{\uptau_0}]^{\virt} \times \upepsilon^! \left( \prod_{i=1}^{\ell} [\Mbar_{\uptau_i}]^{\virt} \right)\right)\]
		where the first equality follows from \cite[Theorems 6.2(c) and 6.5]{FultonBig} and the second equality follows from \cite[Example~6.5.2]{FultonBig}. Since $\scrL_0$ trivialises over $S \subseteq \scrS_0$, we have
		\[ \Mbar_{\uptau_i} = \Mbar_{\uptau_i}\lvert_{0} \times \Aone\] 
		for all $i\in\{1,\ldots,\ell\}$. A direct comparison of obstruction theories then produces
		\[ \upepsilon^! \left( \prod_{i=1}^{\ell} [\Mbar_{\uptau_i}]^{\virt}\right) = \prod_{i=1}^{\ell} (-1)^{g_i} \uplambda_{g_i} \cap [\Mbar_{\uptau_i}\lvert_{0}]^{\virt}\]
		which we combine with \Cref{lem:fibre product unexpanded divisor} to obtain the result.
	\end{proof}
	
	We now show that $\uptau_0$ supports no marking legs, thus completing the proof of \Cref{thm: second vanishing}. Recall that the insertions
	\[ \upgamma = \prod_{j=1}^s \ev_{y_j}^\star (\upgamma_j) \psi_{y_j}^{k_j} \]
	are $D$-avoidant (\Cref{def: E disjoint insertions}). The following vanishing result involves a dimension count, for which it is crucial that the codimension of $\upgamma$ coincides with the virtual dimension of $\Mbar_{g,\bfc,\upbeta}(\OO_S(-D)\, |\, \hat{E})$.
	
	For $i \in \{0,1,\ldots,\ell\}$ let $J(i) \subseteq \{1,\ldots,s\}$ denote the set indexing those markings $y_j$ supported at $\uptau_i$ for which $\upgamma_j\neq \mathbbm{1}_{\mathrm{Ev}_j}$. Define the class
	\[ \upgamma_{J(i)} \coloneqq \prod_{j\in J(i)} \ev_{y_j}^\star (\upgamma_j) \psi_{y_j}^{k_j} \]
	so that $\upgamma=\Pi_{i=0}^\ell \upgamma_{J(i)}$. We continue to write $\uptheta$ for the gluing morphism
	\begin{equation*}
		\Nbar_{\uptau} \rightarrow \MapsRefinedEvals
	\end{equation*}
	making the diagram \eqref{eq: guling morphism diagram} commute.	
	
	\begin{proposition} \label{prop: v0 no markings} Let $\uptau$ be a weakly star-shaped tropical type as in \Cref{thm: first vanishing}. The cycle
		\[ \frgt_\star \big(\upgamma \cap \uprho_\star \upiota_\star [\Mbar_\uptau]^{\virt}\big) \]
		vanishes unless $\uptau_0$ carries no markings. In this case, we have the following equality in $A_0(\MapsRefinedEvals)$:
		\begin{equation}
			\label{eq:term decomposition formula simplification}
			\frgt_\star \big(\upgamma \cap \uprho_\star \upiota_\star [\Mbar_\uptau]^{\virt}\big) = c_\uptau \cdot \uptheta_\star \updelta^! \left( [\Mbar_{\uptau_0}]^{\virt} \times \prod_{i=1}^\ell (-1)^{g_i} \uplambda_{g_i} \upgamma_{J(i)} \cap [\Mbar_{\uptau_i}\lvert_{0}]^{\virt} \right).
		\end{equation}
	\end{proposition}
	
	\begin{proof} It is sufficient to prove the vanishing result. We consider separately the case of markings $y_j$ with trivial and non-trivial insertions. First suppose that $\uptau_0$ carries a marking $y_j$ with $\upgamma_j \neq \mathbbm{1}_{\mathrm{Ev}_j}$. For every marking $y_j$ with $\upgamma_j \neq \mathbbm{1}_{\mathrm{Ev}_j}$, choose a regularly embedded subvariety $Z_j \subseteq \mathrm{Ev}_j$ with $[Z_j] = \upgamma_j$ and $Z_j \cap D = \emptyset$. Let
		\begin{equation} \label{eqn: strict transform of Zi} \scrZ_j \hookrightarrow \scrS \end{equation}
		denote the strict transform of $Z_j \times \Aone \hookrightarrow S \times \Aone$. Since $Z_j \cap D = \emptyset$ it follows that on the central fibre the inclusion \eqref{eqn: strict transform of Zi} factors through the irreducible component $S$. In fact:
		\[ Z_j \hookrightarrow \mathrm{Ev}_j \hookrightarrow S \hookrightarrow \scrS_0.\]
		Consider $J \colonequals \bigsqcup_{i=0}^\ell J(i)$ the set indexing all markings with non-trivial insertion. Take $Z \colonequals \Pi_{j\in J} Z_j$ and $\mathrm{Ev} \colonequals \Pi_{j\in J} \mathrm{Ev}_j$. We combine the above morphisms into a regular embedding
		\[ \upeta_Z \colon Z \hookrightarrow \mathrm{Ev}.\]
		Similarly, for $i \in \{0,\ldots,\ell\}$ we consider the corresponding inclusion
		\[ \upeta_i \colon Z_{J(i)} \colonequals \Pi_{j \in J(i)} Z_j \hookrightarrow \Pi_{j\in J(i)} \mathrm{Ev}_j \equalscolon \mathrm{Ev}_{J(i)}.\]
		We perform the pullback along $\upeta_Z$ to produce a closed substack with constrained evaluation:
		\[
		\begin{tikzcd}
			\Nbar_{\uptau}|_Z \ar[r,hook] \ar[d] \ar[rd,phantom,"\square",start anchor=center,end anchor=center] & \Nbar_{\uptau} \ar[d] \\
			Z \ar[r,hook,"\upeta_Z"] & \mathrm{Ev}.
		\end{tikzcd}
		\]
		Applying $\upeta_Z^!$ to \Cref{eq:simplification term decomposition formula} we obtain
		\begin{align} \label{eqn: term degeneration formula distributed insertions} \upeta_Z^! \upnu_\star [\Mbar_\uptau]^{\virt} & = c_\uptau \cdot \updelta^! \upeta_Z^! \left( [\Mbar_{\uptau_0}]^{\virt} \times \prod_{i=1}^{\ell} (-1)^{g_i} \uplambda_{g_i} \cap [\Mbar_{\uptau_i}\lvert_{0}]^{\virt} \right) \nonumber \\
			& = c_\uptau \cdot \updelta^! \left( \upeta_0^! [\Mbar_{\uptau_0}]^{\virt} \times \prod_{i=1}^{\ell} (-1)^{g_i} \uplambda_{g_i} \cap \upeta_i^! [\Mbar_{\uptau_i}\lvert_{0}]^{\virt} \right)
		\end{align}
		where the first equality follows from \cite[Theorem~6.2(a) and 6.4]{FultonBig} and the second from \cite[Example~6.5.2 and Proposition~6.3]{FultonBig}.
		
		Combining Lemmas~\ref{lem: no vertex on E}~and~\ref{lem: all leaves map to S}, we see that for every vertex $v$ of the graph of $\uptau_0$, the restriction $f|_{C_v}$ factors through $P$. It follows that if a marking $y_j$ belongs to $\uptau_0$ the associated evaluation map factors through $D \subseteq S$. On the other hand $Z_j \cap D = \emptyset$ for all $j\in J$. We conclude that if $J(0)\neq \emptyset$ then the fibre product
		\[ \Mbar_{\uptau_0} \times_{\mathrm{Ev}_{J(0)}} Z_{J(0)} \]
		is empty. Therefore $\upeta_0^! [\Mbar_{\uptau_0}]^{\virt} = 0$ and by \eqref{eqn: term degeneration formula distributed insertions} the contribution vanishes. This is the only point in the paper where we use the assumption of $D$-avoidant insertions.
		
		We conclude that $\uptau_0$ only contains markings $y_j$ with $\upgamma_j=\mathbbm{1}_{\mathrm{Ev}_j}$ and $k_j=0$. Capping \eqref{eqn: term degeneration formula distributed insertions} with the psi classes appearing in the insertions $\upgamma$ produces:
		\begin{equation*}
			\upgamma \cap \upnu_\star [\Mbar_\uptau]^{\virt} = c_\uptau \cdot \updelta^! \left( [\Mbar_{\uptau_0}]^{\virt} \times \prod_{i=1}^{\ell} (-1)^{g_i} \uplambda_{g_i} \upgamma_{J(i)} \cap [\Mbar_{\uptau_i}\lvert_{0}]^{\virt} \right).
		\end{equation*}
		We now perform a dimension count. For $i \in \{0,1,\ldots,\ell\}$ let $s_i$ denote the number of markings $y_j$ contained in $\uptau_i$. Then for $i \in \{1,\ldots,\ell\}$ the virtual dimension of $\Mbar_{\uptau_i}\lvert_{0}$ is $-\upbeta_i \cdot (K_S + D + E) + g_i + s_i$. It follows that the class
		\begin{equation}
			\label{eq:vVerticesCycle}
			\prod_{i=1}^{\ell} (-1)^{g_i} \uplambda_{g_i} \upgamma_{J(i)} \cap [\Mbar_{\uptau_i}\lvert_{0}]^{\virt}
		\end{equation}
		has dimension:
		\begin{align} \label{eqn: dimension count for cycle in fourth vanishing proof}
			& \sum_{i=1}^{\ell} \left( - \upbeta_i \cdot (K_S + D + E) + s_i - \Sigma_{j\in J(i)} (k_j + \mathrm{codim}(Z_j,\mathrm{Ev}_j)) \right) \nonumber \\
			= \ & \uppi_\star\upbeta_0  \cdot (K_S + D + E) - s_0 + \left( - \upbeta \cdot (K_S + D + E) + s - \Sigma_{i=0}^{\ell} \Sigma_{j\in J(i)} (k_j + \mathrm{codim}(Z_j,\mathrm{Ev}_j)) \right) \nonumber \\
			= \ & \uppi_\star\upbeta_0 \cdot (K_S + D + E) - s_0.
		\end{align}
		Here $\uppi_\star\upbeta_0 \in A_1(S;\Z)$ is the projection along $\uppi:P \rightarrow D_0=D \hookrightarrow S$ of the curve class attached to $\uptau_0$. The first equality follows from
		\[ \upbeta = \uppi_\star\upbeta_0 + \sum_{i=1}^{\ell} \upbeta_i, \qquad s= \sum_{i=0}^{\ell} s_i\]
		and the fact that by the previous arguments $k_j + \mathrm{codim}(Z_j,\mathrm{Ev}_j)=0$ for all $j \in J(0)$. The second equality holds because of the assumption that the codimension of $\upgamma$ is equal to the virtual dimension of $\Mbar_{g,\bfc,\upbeta}(\OO_S(-D)\, |\, \hat{E})$.
		
		Note that $\uppi_\star\upbeta_0 $ is necessarily a multiple of $D$, say $\uppi_\star\upbeta_0=k D$ for some $k\geq 0$. By adjunction and $D \cdot E = 2$ we obtain
		\[ \uppi_\star\upbeta_0 \cdot (K_S + D + E) = k D \cdot (K_S + D) + 2k = 2 k g_D = 0. \]
		(This is the first of two points in the argument where we use the assumption that $D$ is rational.) By \eqref{eqn: dimension count for cycle in fourth vanishing proof} this implies that the cycle \eqref{eq:vVerticesCycle} has dimension $-s_0$. The contribution therefore vanishes unless $s_0=0$.
	\end{proof}
	
	\subsubsection{Third reduction: star-shaped graphs}\label{sec: third vanishing} Using tropical arguments we now constrain the shape of the subgraphs $\Gamma_i$ further. 
	
	\begin{proposition}[Third reduction] \label{thm: third vanishing} For $i \in \{0,1,\ldots,m\}$ the subgraph $\Gamma_i$ consists of a single vertex $v_i$. We have $\ftrop(v_0) \in P$ and $\ftrop(v_1),\ldots,\ftrop(v_m) \in S$. The tropical type $\uptau$ therefore takes the following form:
		\[
		\begin{tikzpicture}

			\draw[ ->] (-14,0) -- (-13,0.3);
			\draw[ ->] (-14,0) -- (-13,-0.3);
			\draw[] (-13,0.1) node{$\vdots$};
			\draw[ ->] (-16,0) -- (-17,0.3);
			\draw[ ->] (-16,0) -- (-17,-0.3);
			\draw[] (-17,0.1) node{$\vdots$};
			
			\draw[fill=black] (-15,3) circle[radius=2pt];
			\draw[] (-15,3) node[above]{$v_0$};
			\draw[] (-15,3) -- (-16,0);
			\draw[fill=black] (-16,0) circle[radius=2pt];
			\draw[] (-16,0) node[below]{$v_1$};
			\draw[] (-15,3) -- (-14,0);
			\draw[fill=black] (-14,0) circle[radius=2pt];
			\draw[] (-14,0) node[below]{$v_m$};
			\draw[] (-15,1) node{$\cdots$};

			\draw[->] (-13,1.5) -- (-12,1.5);
			\draw[] (-12.5,1.5) node[above]{$\ftrop$};
			
			\draw[fill=black] (-8,0) circle[radius=2pt];
			\draw (-8,0) node[below]{$S$};
			
			\draw[black,->] (-8,0) -- (-5,0);
			\draw (-5,0) node[right]{$E$};
			
			\draw (-8,0) -- (-8,3);
			\draw (-8,1.5) node[right]{$D_0$};
			
			\draw[fill=black] (-8,3) circle[radius=2pt];
			\draw (-8,2.7) node[right]{$P$};
			
			\draw[black,->] (-8,3) -- (-5,3);
			\draw (-5,3) node[right]{$E_2$};
			
			\draw[black,->] (-8,3) -- (-11,3);
			\draw (-11,3) node[left]{$E_1$};
			
			\draw[black,->] (-8,0) -- (-11,0);
			\draw (-11,0) node[left]{$E$};
			
			\draw[black] (-5.5,0) node{$\mid\mid$};
			\draw[black] (-10.5,0) node{$\mid\mid$};
			
			\draw[black,fill=black] (-8,3) circle[radius=2pt];
			\draw[black,bend left] (-8,3) -- (-8,0);
			\draw[black,fill=black] (-8,0) circle[radius=2pt];
		\end{tikzpicture}
		\]
	\end{proposition}
	
\begin{proof} Combine Lemmas~\ref{lem: Gamma_i a vertex on S} and \ref{lem: Gamma_0 a vertex on P} below. \end{proof}
	
	We prove the statement in several steps.
	
	\begin{lemma} \label{lem: hard step} Let $v \in V(\Gamma_i)$ be such that $\ftrop(v)$ belongs to $q_1$ or $q_2$. Then the incoming edge at $v$ cannot have positive horizontal slope.
	\end{lemma}
	
	\begin{proof} 
		Suppose without loss of generality that $\ftrop(v) \in q_1$, and suppose for a contradiction that the incoming edge has positive horizontal slope. By balancing, there must exist an edge outgoing from $v$ with positive horizontal slope. By \Cref{thm: second vanishing} this edge must be finite, and by \Cref{lem: negative vertical slope} it must have negative vertical slope. Follow this edge to the next vertex. If the vertex also belongs to $q_1$ we repeat the argument. Eventually, we obtain a vertex $w$ with $\ftrop(w) \in E$. This contradicts \Cref{lem: no vertex on E}.
	\end{proof}

	Recall that for $i\in\{1,\ldots,m\}$ we denote by $e_i$ the edge connecting $\Gamma_i$ and $\Gamma_0$ and write $v_i \in V(\Gamma_i)$ and $w_i \in V(\Gamma_0)$ for its endpoints (\Cref{notation: vertices e_i}).

\begin{lemma} \label{lem: v0 on P} We have $\ftrop(w_i) \in P$.
\end{lemma}

\begin{proof} It is equivalent to show that $\ftrop(w_i)$ does not belong to $E_1$ or $E_2$. Lemmas~\ref{lem: no vertex on E} and \ref{lem: hard step} eliminate many cases. The only ones which remain are
	
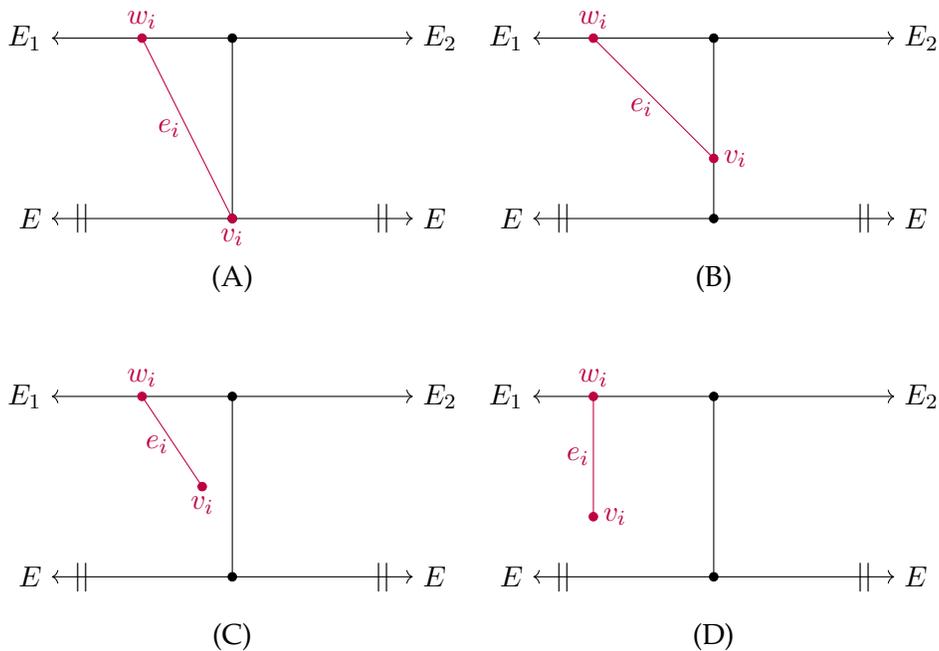
\[
	\begin{tikzpicture}[scale=0.8]
		
		
		\draw[fill=black] (-8,0) circle[radius=2pt];
		
		\draw (-8,-1) node{(A)};
		
		\draw[->] (-8,0) -- (-5,0);
		\draw (-5,0) node[right]{$E$};
		
		\draw (-8,0) -- (-8,3);
		
		\draw[fill=black] (-8,3) circle[radius=2pt];
		
		\draw[->] (-8,3) -- (-5,3);
		\draw (-5,3) node[right]{$E_2$};
		
		\draw[->] (-8,3) -- (-11,3);
		\draw (-11,3) node[left]{$E_1$};
		
		\draw[->] (-8,0) -- (-11,0);
		\draw (-11,0) node[left]{$E$};
		
		\draw (-5.5,0) node{$\mid\mid$};
		\draw (-10.5,0) node{$\mid\mid$};
		
		\draw[purple,fill=purple] (-9.5,3) circle[radius=2pt];
		\draw[purple] (-9.5,3) node[above]{$w_i$};
		\draw[purple] (-9.5,3) -- (-8,0);
		\draw[purple,fill=purple] (-8,0) circle[radius=2pt];
		\draw[purple] (-8,0) node[below]{$v_i$};
		\draw[purple] (-9.05,1.5) node{$e_i$};

		
		\draw[fill=black] (0,0) circle[radius=2pt];
		\draw (0,-1) node{(B)};
		
		\draw[->] (0,0) -- (3,0);
		\draw (3,0) node[right]{$E$};
		
		\draw (0,0) -- (0,3);
		
		\draw[fill=black] (0,3) circle[radius=2pt];
		
		\draw[->] (0,3) -- (3,3);
		\draw (3,3) node[right]{$E_2$};
		
		\draw[->] (0,3) -- (-3,3);
		\draw (-3,3) node[left]{$E_1$};
		
		\draw[->] (0,0) -- (-3,0);
		\draw (-3,0) node[left]{$E$};
		
		
		\draw (2.5,0) node{$\mid\mid$};
		\draw (-2.5,0) node{$\mid\mid$};
		
		\draw[purple,fill=purple] (-2,3) circle[radius=2pt];
		\draw[purple] (-2,3) node[above]{$w_i$};
		\draw[purple] (-2,3) -- (0,1);
		\draw[purple,fill=purple] (0,1) circle[radius=2pt];
		\draw[purple] (0,1) node[right]{$v_i$};
		\draw[purple] (-1.2,1.85) node{$e_i$};
	\end{tikzpicture}
	\] \,
	\[
	\begin{tikzpicture}[scale=0.8]
		
		
		\draw[fill=black] (-8,0) circle[radius=2pt];
		
		\draw (-8,-1) node{(C)};
		
		\draw[->] (-8,0) -- (-5,0);
		\draw (-5,0) node[right]{$E$};
		
		\draw (-8,0) -- (-8,3);
		
		\draw[fill=black] (-8,3) circle[radius=2pt];
		
		\draw[->] (-8,3) -- (-5,3);
		\draw (-5,3) node[right]{$E_2$};
		
		\draw[->] (-8,3) -- (-11,3);
		\draw (-11,3) node[left]{$E_1$};
		
		\draw[->] (-8,0) -- (-11,0);
		\draw (-11,0) node[left]{$E$};
		
		\draw (-5.5,0) node{$\mid\mid$};
		\draw (-10.5,0) node{$\mid\mid$};
		
		\draw[purple,fill=purple] (-9.5,3) circle[radius=2pt];
		\draw[purple] (-9.5,3) node[above]{$w_i$};
		\draw[purple] (-9.5,3) -- (-8.5,1.5);
		\draw[purple,fill=purple] (-8.5,1.5) circle[radius=2pt];
		\draw[purple] (-8.5,1.5) node[below]{$v_i$};
		\draw[purple] (-9.25,2.2) node{$e_i$};

		
		\draw[fill=black] (0,0) circle[radius=2pt];
		\draw (0,-1) node{(D)};
		
		\draw[->] (0,0) -- (3,0);
		\draw (3,0) node[right]{$E$};
		
		\draw (0,0) -- (0,3);
			
		\draw[fill=black] (0,3) circle[radius=2pt];
				
		\draw[->] (0,3) -- (3,3);
		\draw (3,3) node[right]{$E_2$};
		
		\draw[->] (0,3) -- (-3,3);
		\draw (-3,3) node[left]{$E_1$};
		
		\draw[->] (0,0) -- (-3,0);
		\draw (-3,0) node[left]{$E$};
				
		\draw (2.5,0) node{$\mid\mid$};
		\draw (-2.5,0) node{$\mid\mid$};
		
		\draw[purple,fill=purple] (-2,3) circle[radius=2pt];
		\draw[purple] (-2,3) node[above]{$w_i$};
		\draw[purple] (-2,3) -- (-2,1);
		\draw[purple,fill=purple] (-2,1) circle[radius=2pt];
		\draw[purple] (-2,1) node[right]{$v_i$};
		\draw[purple] (-1.9,2) node[left]{$e_i$};
		
	\end{tikzpicture}
	\]
	along with the matching cases for $E_2$. We deal with cases (A), (B), (C) together. In each case balancing at $w_i$ (\Cref{sec: balancing on E1 E2}) ensures that there exists a finite outgoing edge with positive horizontal slope (note that $w_i$ cannot support any marking legs by \Cref{thm: second vanishing}). If this edge has zero vertical slope, we follow it to the next vertex and repeat the argument. Eventually, we arrive at a vertex on $E_1$ supporting an outgoing edge with positive horizontal slope and nonzero vertical slope. This leaves $\Gamma_0$ and enters one of the other subgraphs $\Gamma_j$. But Lemmas~\ref{lem: no vertex on E} and \ref{lem: hard step} preclude this.
	
	It remains to consider (D). By \Cref{lem: negative vertical slope}, \Cref{lem: hard step}, and balancing, all edges outgoing from $v_i$ must have zero horizontal slope and negative vertical slope. Inducting along the path, we eventually arrive at a vertex on $E$, contradicting \Cref{lem: no vertex on E}.
\end{proof}

\begin{lemma} \label{lem: Gamma_i a vertex on S} For $i \in \{1,\ldots,m\}$ the graph $\Gamma_i$ consists only of the vertex $v_i$, with $\ftrop(v_i) \in S$.
\end{lemma}
\begin{proof} We first show $\ftrop(v_i) \in S$. By \Cref{lem: no vertex on E} we have $\ftrop(v_i) \not\in E$. On the other hand \Cref{lem: hard step} and \ref{lem: v0 on P} together imply that $\ftrop(v_i) \not\in q_1$ or $q_2$. It remains to consider the case $\ftrop(v_i) \in D_0$. By \Cref{lem: v0 on P} we have $\ftrop(w_i) \in P$. Since $\Gamma_i$ is a tree, it follows in this case that the corresponding tropical type is not rigid. We conclude that $\ftrop(v_i) \in S$. 
	
	By \Cref{lem: negative vertical slope} the vertex $v_i$ has no outgoing edges with positive vertical slope. By \Cref{lem: no vertex on E} it then has no outgoing edges with positive horizontal slope. It follows that the only outgoing edges have slope zero in both directions. These do not exist because the tropical type is rigid.
\end{proof}

\begin{lemma} \label{lem: Gamma_0 a vertex on P} $\Gamma_0$ consists only of a single vertex $w_0$, with $\ftrop(w_0) \in P$.
\end{lemma}

\begin{proof} By \Cref{lem: Gamma_i a vertex on S} we know that each $\Gamma_i$ consists of a single vertex $v_i$ with $\ftrop(v_i) \in S$ and that $v_i$ is connected to $\Gamma_0$ by an edge $e_i$ with $\ftrop(e_i) \subseteq D_0$. From this we see that if $\Gamma_0$ contains any vertices along $E_1$ or $E_2$ then the tropical type is not rigid. On the other hand if $\Gamma_0$ contains more than one vertex over $P$ then it must contain contracted edges, and again the tropical type is not rigid.
\end{proof}

\subsubsection{Fourth reduction: single-edge graphs} \label{sec: fourth vanishing}

By \Cref{thm: third vanishing} we know that the tropical type $\uptau$ is star-shaped, with a vertex $v_0$ mapping to $P$ and vertices $v_1,\ldots,v_m$ mapping to $S$ and supporting all the marking legs. We now establish the fourth and final reduction:

\begin{proposition}[Fourth reduction] \label{thm: fourth vanishing} Let $\uptau$ be a star-shaped tropical type, as in \Cref{thm: third vanishing}. Then $m=1$ and $v_1$ contains all of the markings:
	
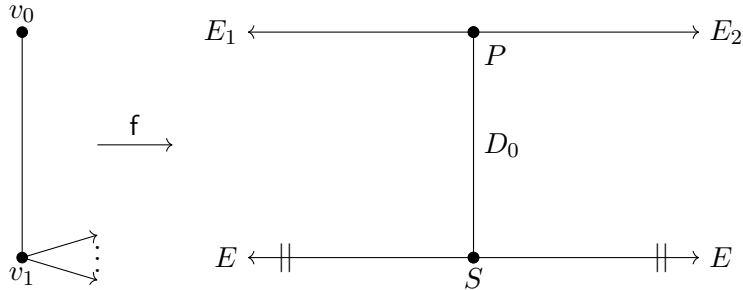
\[
		\begin{tikzpicture}
			\draw[ ->] (-14,0) -- (-13,0.3);
			\draw[ ->] (-14,0) -- (-13,-0.3);
			\draw[] (-13,0.1) node{$\vdots$};
			
			\draw[fill=black] (-14,3) circle[radius=2pt];
			\draw[] (-14,3) node[above]{$v_0$};
			\draw[] (-14,3) -- (-14,0);
			\draw[fill=black] (-14,0) circle[radius=2pt];
			\draw[] (-14,0) node[below]{$v_1$};
				
			\draw[->] (-13,1.5) -- (-12,1.5);
			\draw[] (-12.5,1.5) node[above]{$\ftrop$};
			
			\draw[fill=black] (-8,0) circle[radius=2pt];
			\draw (-8,0) node[below]{$S$};
			
			\draw[black,->] (-8,0) -- (-5,0);
			\draw (-5,0) node[right]{$E$};
			
			\draw (-8,0) -- (-8,3);
			\draw (-8,1.5) node[right]{$D_0$};
			
			\draw[fill=black] (-8,3) circle[radius=2pt];
			\draw (-8,2.7) node[right]{$P$};
			
			\draw[black,->] (-8,3) -- (-5,3);
			\draw (-5,3) node[right]{$E_2$};
			
			\draw[black,->] (-8,3) -- (-11,3);
			\draw (-11,3) node[left]{$E_1$};
			
			\draw[black,->] (-8,0) -- (-11,0);
			\draw (-11,0) node[left]{$E$};
			
			\draw[black] (-5.5,0) node{$\mid\mid$};
			\draw[black] (-10.5,0) node{$\mid\mid$};
			
			\draw[black,fill=black] (-8,3) circle[radius=2pt];
			\draw[black,bend left] (-8,3) -- (-8,0);
			\draw[black,fill=black] (-8,0) circle[radius=2pt];
		\end{tikzpicture}
		\]
\end{proposition}

	\begin{proof} This follows from \Cref{prop: only one edge} below.	
	\end{proof}

%

The proof necessitates a careful analysis of the gluing appearing in the degeneration formula, similar to \Cref{sec: second vanishing}. We first note that if $\uptau$ is as in \Cref{thm: third vanishing} then there is no need to pass to a subdivision of the target, since all finite edges map to $D_0$ and all vertices map to $S$ or $P$. Consequently the identity \eqref{eq:term decomposition formula simplification} simplifies to\footnote{There is a typo in the statement of \cite[Lemma~6.4.4]{RangExpansions}: the gluing factor on the right-hand side must be divided by $m_\uptau$. We thank Dhruv~Ranganathan for confirming this. This typo is corrected in \cite{MRLogCurveSheaves}. The above identity is consistent with \cite[Theorem~1.1]{GrossGluing}, \cite[Corollary~7.10.4]{ChenDegeneration}, and {\cite[(1.4)]{KLR}}.}
\begin{equation}
	\label{eqn:single term degeneration formula}
	\frgt_\star \big(\upgamma \cap \uprho_\star \upiota_\star [\Mbar_\uptau]^{\virt} \big) = \left( \dfrac{\prod_{i=1}^m m_i}{\operatorname{lcm}(m_i)} \right) \uptheta_\star \updelta^! \left( [\Mbar_{\uptau_0}]^{\virt} \times \prod_{i=1}^m (-1)^{g_i} \uplambda_{g_i} \upgamma_{J(i)} \cap [\Mbar_{\uptau_i}(S\,|\,D+E)]^{\virt} \right).
\end{equation}
Note that the graph $\Gamma_0$ underlying $\uptau_0$ consists of a single vertex $v_0$ supporting no marking legs. The nodes $q_1,\ldots,q_m$ at this vertex have tangency zero with respect to $E_1$ and $E_2$, and the balancing condition determines the associated curve class $\upbeta_0$ as
\[ \upbeta_0 = (D \cdot\upbeta) \, F \]
where $F \in A_1(P)$ is the class of a fibre. Consequently the evaluation morphism $\Mbar_{\uptau_0} \to \prod_{i=1}^m D_0$ factors through the small diagonal $D_0 \hookrightarrow \Pi_{i=1}^m D_0$. Let
\[ \uppsi : \Mbar_{\uptau_0} \to \Mbar_{g_0,m} \times D_0 \]
denote the morphism remembering only the stabilised source curve and the evaluation. We obtain a diagram
\begin{equation*}
	\begin{tikzcd}
		\MapsRefinedEvals & \Nbar_{\uptau} \arrow[d] \arrow[r] \arrow[l, swap, "\uptheta"] \ar[rd,phantom,"\square",start anchor=center,end anchor=center] & \Mbar_{\uptau_0} \times \prod_{i=1}^m \Mbar_{\uptau_i}(S \, | \, D+E) \arrow[d, "\uppsi \times \id"] \\
		& \Mbar_{g_0,m}\times \ol{L}_\uptau \arrow[d] \arrow[r] \arrow[ul,"\upphi"] \ar[rd,phantom,"\square",start anchor=center,end anchor=center] & \Mbar_{g_0,m}\times D_0 \times \prod_{i=1}^m \Mbar_{\uptau_i}(S \, | \, D+E) \arrow[d]\\
		& \ol{L}_\uptau \arrow[r] \arrow[d] \ar[rd,phantom,"\square",start anchor=center,end anchor=center] &  D_0 \times \prod_{i=1}^m \Mbar_{\uptau_i}(S \, | \, D+E) \arrow[d]\\
		& \prod_{i=1}^m D_0 \arrow[r, "\updelta"] & \prod_{i=1}^m D_0^2
	\end{tikzcd}
\end{equation*}
in which $\ol{L}_\uptau$ is defined via the bottom cartesian square.\footnote{If $g_0=0$ and $m<3$ we adopt the convention $\Mbar_{0,1}=\Mbar_{0,2}=\Speck$.} It parametrises logarithmic maps $f_i \colon C_i \to (S\, |\, D+E)$ for $i \in \{1,\ldots,m\}$, such that $f_1(q_1) = \ldots = f_m(q_m)$. The gluing morphism $\upphi$ is defined similarly to $\uptheta$, see \eqref{eq: guling morphism diagram}.

For $g \geq 0$ and $d > 0$ consider the moduli space
\[ \Mbar_{g,(d),d}(\OO_{\PP^1}(-1)\, |\, \hat{0}) \]
of stable logarithmic maps with maximal tangency at a single marking along the fibre over $0\in\PP^1$, and consider the pushforward of its virtual fundamental class to the moduli space of curves:
\[ C_{g,d} \colonequals \uppi_\star [\Mbar_{g,(d),d}(\OO_{\PP^1}(-1)\, |\, \hat{0})]^{\virt} \in A_{g}(\Mbar_{g,1}). \]

\begin{proposition}	\label{prop: only one edge} Let $\uptau$ be a star-shaped tropical type as in \Cref{thm: third vanishing}. The contribution $\frgt_\star (\upgamma \cap \uprho_\star \upiota_\star [\Mbar_\uptau]^{\virt})$ vanishes unless $m=1$. In this case $\ol{L}_\uptau = \Mbar_{\uptau_1}(S\, |\, D+E)$ and we have
\begin{equation} \label{eqn: prop only one edge}	\frgt_\star \big(\upgamma \cap \uprho_\star \upiota_\star [\Mbar_\uptau]^{\virt}\big) = (-1)^g \upphi_\star \left( \left(\uplambda_{g_0} \cap C_{g_0,D \cdot \upbeta} \right) \times \left(\uplambda_{g_1} \upgamma \cap [\Mbar_{\uptau_1} (S \, | \, D+E)]^{\virt} \right)\right).\end{equation}
\end{proposition}

\begin{proof} Using the compatibility of proper push forward and outer product \cite[Proposition~1.10]{FultonBig}, the identity \eqref{eqn:single term degeneration formula} simplifies to
	\begin{equation} \label{eq:simplifyDecTerm1}
		\frgt_\star \big(\upgamma \cap \uprho_\star \upiota_\star [\Mbar_\uptau]^{\virt}\big)  = \left( \dfrac{\prod_{i=1}^m m_i}{\operatorname{lcm}(m_i)} \right) \upphi_\star \updelta^! \left( \uppsi_\star [\Mbar_{\uptau_0}]^{\virt} \times \prod_{i=1}^m (-1)^{g_i} \uplambda_{g_i} \upgamma_{J(i)} \cap [\Mbar_{\uptau_i}(S \, | \, D+E)]^{\virt} \right).
	\end{equation}
Since $D_0$ is a rational curve, \cite[Example~1.10.2]{FultonBig} ensures that the outer product morphism
	\[ A_\star(\Mbar_{g_0,m}) \otimes A_\star(D_0) \to A_\star(\Mbar_{g_0,m}\times D_0)\]
	is surjective (this is the second and final point in the argument where we use the assumption that $D$ is rational). We may thus write
	\begin{equation} \label{eqn: product decomposition fourth vanishing}
		\uppsi_\star [\Mbar_{\uptau_0}]^{\virt} = (A \times [D_0]) + (B \times [\pt])
	\end{equation}
	where $A,B \in A_\star(\Mbar_{g_0,m})$ and $[\pt]$ is the generator of $A_0(D_0)$. Starting with the second term, we have by \cite[Example~6.5.2]{FultonBig}:
\[ \updelta^! \left( B \! \times \! [\pt] \times \! \prod_{i=1}^m (-1)^{g_i} \uplambda_{g_i}  \upgamma_{J(i)} \cap [\Mbar_{\uptau_i}(S \, | \, D+E)]^{\virt} \right) = B \times \updelta^! \left( [\pt] \! \times \! \prod_{i=1}^m (-1)^{g_i} \uplambda_{g_i}  \upgamma_{J(i)} \cap [\Mbar_{\uptau_i}(S \, | \, D+E)]^{\virt} \right).\]
A dimension count similar to \eqref{eqn: dimension count for cycle in fourth vanishing proof} then shows that
	\begin{equation*}
		\updelta^! \left( [\pt] \times \prod_{i=1}^m (-1)^{g_i} \uplambda_{g_i}  \upgamma_{J(i)} \cap [\Mbar_{\uptau_i}(S \, | \, D+E)]^{\virt} \right) \in A_{-m} ( \ol{L}_\uptau ).
	\end{equation*}
	We always have $m \geq 1$ and hence this contribution always vanishes. Considering the second term of \eqref{eqn: product decomposition fourth vanishing}, a similar analysis shows that
	\begin{equation*}
		\updelta^!  \left( [D_0] \times \prod_{i=1}^m (-1)^{g_i} \uplambda_{g_i} \upgamma_{J(i)} \cap [\Mbar_{\uptau_i}(S \, | \, D+E)]^{\virt} \right) \in A_{1-m} ( \ol{L}_\uptau )
	\end{equation*}
	and so the contribution vanishes unless $m = 1$. This proves the vanishing statement.
	
	Now suppose that $m=1$.  In this case the gluing factor is $(\Pi_{i=1}^m m_i)/\operatorname{lcm}(m_i)=1$. We obtain:
	\begin{equation} \label{eq:simplifyDecTerm2}
		\frgt_\star \big(\upgamma \cap \uprho_\star \upiota_\star [\Mbar_\uptau]^{\virt}\big) = \upphi_\star  \left( A \times \updelta^! \left([D_0] \times  (-1)^{g_1} \uplambda_{g_1} \upgamma \cap [\Mbar_{\uptau_1}(S \, | \, D+E)]^{\virt} \right)\right).
	\end{equation}
It remains to determine the cycle $A$. Recall that $\uppsi_\star [\Mbar_{\uptau_0}]^{\virt} = (A \times [D_0]) + (B \times [\pt])$. Choose a closed point $\upxi \colon \Speck \hookrightarrow D_0$ and form the fibre product:
	\begin{equation*}
		\begin{tikzcd}
			\Mbar_{g_0,(D \cdot \upbeta), D \cdot \upbeta} (\calO_{\mathbb{P}^1}(-1) \,|\, \hat{0} ) \arrow[r,hook] \arrow[d,"\uppi"] \ar[rd,phantom,"\square",start anchor=center,end anchor=center] & \Mbar_{\uptau_0} \arrow[d,"\uppsi"] \\
			\Mbar_{g_0,1} \arrow[r, hook] \arrow[d] \ar[rd,phantom,"\square",start anchor=center,end anchor=center] &  \Mbar_{g_0,1} \times D_0  \arrow[d] \\
			\Speck \arrow[r, hook, "\upxi"] & D_0.
		\end{tikzcd}
	\end{equation*}
We then obtain
\begin{align*}	\label{eq:aplhacalc}
A & = \upxi^! \uppsi_\star [\Mbar_{\uptau_0}]^{\virt}	 = \uppi_\star \upxi^! [\Mbar_{\uptau_0}]^{\virt} = \uppi_\star \left( (-1)^{g_0} \uplambda_{g_0} \cap [\Mbar_{g_0,(D \cdot \upbeta), D \cdot \upbeta} (\calO_{\mathbb{P}^1}(-1) \,|\, \hat{0} )]^{\virt} \right) = (-1)^{g_0} \uplambda_{g_0} \cap C_{g_0,D \cdot \upbeta}
\end{align*}
where in the final step we use the fact that lambda classes are preserved by pullback along stabilisation morphisms. The identity \eqref{eq:simplifyDecTerm2} then immediately implies \eqref{eqn: prop only one edge}.
\end{proof}

We now combine \Cref{prop: only one edge} with \eqref{eqn: decomposition formula}. For single-edge star-shaped tropical types we have $m_\uptau = D\cdot \upbeta$ and $|\operatorname{Aut}(\uptau)|=1$. We therefore obtain:
\begin{theorem}
	\label{thm: local-log cycle} The following identity holds in $A_0(\MapsRefinedEvals)$
	\begin{equation*}
		\frgt_\star\big(\upgamma \cap [\Mbar_{g,\mathbf{c},\upbeta}(\calO_S(-D)|\hat{E})]^{\virt}\big) = (-1)^g (D \cdot \upbeta) \!\! \sum_{g_1 + g_2 = g} \!\! \upphi_\star \left( (\uplambda_{g_{1}} \cap C_{g_{1},D \cdot \upbeta}) \times (\uplambda_{g_2} \upgamma \cap [\Mbar_{g_2,\hat{\bfc},\upbeta}(S \, | \, D+E)]^{\virt}) \right)
	\end{equation*}
	where the tangency data $\hat{\bfc}$ is obtained from $\bfc$ by introducing one additional marked point with maximal tangency along $D$ (see \Cref{sec: setup}).
\end{theorem}

\subsubsection{Evaluating the integrals} \label{sec: evaluating the integrals} Finally, we use the cycle-theoretic \Cref{thm: local-log cycle} to deduce the numerical \Cref{thm: local-log}.

\begin{proof}[Proof of \Cref{thm: local-log}]
Push the formula in \Cref{thm: local-log cycle} forward to a point, and form the generating function by summing over $g$. We obtain:
\[ \begin{split}
	\sum_{g \geq 0} \GW_{g,\bfc,\upbeta}( \OO_S(-D)\, |\, \hat{E})\langle \upgamma \rangle \cdot \hbar^{2g-2} = (D \cdot \upbeta) &\left( \sum_{g_1 \geq 0} \GW_{g_1,(D\cdot \upbeta),D\cdot \upbeta}(\Ocal_{\mathbb{P}^1}(-1)\, | \, \hat{0}) \langle (-1)^{g_1} \uplambda_{g_1}\rangle \cdot \hbar^{2g_1-1} \right)\\
	 & \left( \sum_{g_2 \geq 0} \GW_{g_2,\hat{\bfc},\upbeta}(S\, | \, D+E) \langle (-1)^{g_2} \uplambda_{g_2} \upgamma \rangle \cdot \hbar^{2g_2-1} \right).
\end{split}\]
The \namecref{thm: local-log cycle} follows immediately from \cite[Lemma~6.3]{BryanPandharipandeLocalCurves} which gives:
\[
\sum_{g_1 \geq 0} \GW_{g_1,(D \cdot \upbeta),D \cdot \upbeta}(\Ocal_{\mathbb{P}^1}(-1)\, | \, \hat{0}) \langle (-1)^{g_1} \uplambda_{g_1}\rangle \cdot \hbar^{2g_1-1} = \left( \dfrac{1}{D \cdot \upbeta} \right) \dfrac{(-1)^{D \cdot \upbeta+1}}{2 \sin\left( \frac{D \cdot \upbeta}{2} \hbar \right)}.\qedhere
\]
\end{proof}

\subsection{Nef pairs} \label{sec: nef pairs} In this section, we adopt the setup of \Cref{sec: setup} and impose the following additional assumptions:
\begin{itemize}
\item $E \cdot \upbeta > 0$.
\item $E^2 \geq 0$.	
\end{itemize}
This includes nef Looijenga pairs as studied in \cite{BBvG2}. The assumptions ensure that the local theory of $\OO_S(-E)$ is well-defined, see \Cref{sec: setup} for the analogous argument for $D$.

We consider stable logarithmic maps with two markings of maximal tangency to $D$ and $E$ and possibly additional interior markings. Denote this contact data by $\hat{\bfc}$ and by $\bfc$ the result of deleting the two tangency markings. We obtain the following correspondence in genus zero:
\begin{theorem}[\Cref{thm: nef correspondence introduction}] \label{thm: nef correspondence} The following identity holds between the genus zero Gromov--Witten invariants with $D$-avoidant insertions $\upgamma$:
\[	\GW_{0,\hat{\bfc},\upbeta}(S\, | \, D+E)\langle \upgamma \rangle = (-1)^{(D+E)\cdot \upbeta} (D\cdot\upbeta)(E\cdot\upbeta) \cdot \GW_{0,c,\upbeta}(\OO_S(-D) \oplus \OO_S(-E))\langle \upgamma \rangle.\]\end{theorem}
\begin{proof} Take \Cref{thm: local-log} in genus zero and apply the logarithmic-local correspondence for smooth pairs \cite[Theorem~1.1]{vGGR}.
\end{proof}

This extends \cite[Theorem~5.2]{BBvG2}, which gives the above result when $(S \, | \, D+E)$ is logarithmically Calabi--Yau and has stationary insertions supported at a single interior marking.

\begin{remark} \label{rmk: nef pairs higher genus} A higher genus analogue of \Cref{thm: nef correspondence} may be obtained by combining \Cref{thm: local-log} with \cite[Theorem~2.7]{BFGW}. The graph sum simplifies, since the logarithmic invariants carry an insertion of $\lambda_g^2$ which vanishes unless $g=0$. The resulting correspondence will thus compare the genus zero invariants of $(S\, |\, D+E)$ to the higher genus invariants of the local geometry. We leave a detailed analysis to future work.
\end{remark}

\subsection{Root stacks and self-nodal pairs} \label{sec: exceptional divisors} In this section we adopt the setup of \Cref{sec: setup} with the following modifications, which apply only to this section:
\begin{itemize}
\item \textbf{Stricter:} $g=0$ and $E \cdot \upbeta=0$.
\item \textbf{Looser:} $D$ is no longer required to be rational, insertions $\upgamma$ are no longer required to avoid $D$.
\end{itemize}
Since $E \cdot \upbeta=0$ there are no markings with tangency along $E$. A case of particular interest is resolutions of irreducible self-nodal curves, where $\upbeta$ is a curve class pulled back along the blowup.
 
In this setting, we establish the analogue of \Cref{thm: local-log}. 
\begin{theorem} \label{thm: exceptional divisors correspondence} The following identity holds between the genus zero Gromov--Witten invariants:
\[ \GW_{0,\hat{\bfc},\upbeta}(S\, | \, D+E)\langle \upgamma \rangle = (-1)^{D \cdot \upbeta +1}(D \cdot \upbeta) \cdot \GW_{0,\bfc,\upbeta}(\OO_S(-D)\, |\, \hat{E})\langle \upgamma \rangle.\]
\end{theorem}
Instead of using the degeneration formula, the result is proved via the enumerative geometry of root stacks. We pass through the following correspondences of genus zero Gromov--Witten theories:
\[
\begin{tikzcd}
\mathsf{Log}(S\, |\, D+E) \ar[r,<->,dashed,"\text{\cite{RootLog}}"] & \mathsf{Orb}(S\, |\, D+E) \ar[r,<->,"\text{\cite{BNTY}}"] & \mathsf{Orb}(\OO_S(-D)\, |\, \hat{E}) \ar[r,<->,"\text{\cite{AbramovichCadmanWise}}"] & \mathsf{Log}(\OO_S(-D)\, |\, \hat{E}).
\end{tikzcd}
\]
The second correspondence follows from \cite[Theorem~1.2]{BNTY} (the result is stated for $D$ nef, but in fact $\upbeta$-nef is all that is required in the proof). The third correspondence follows from \cite[Theorem~1.1]{AbramovichCadmanWise}. Only the first correspondence requires further justification. \Cref{thm: exceptional divisors correspondence} thus reduces to the following:

\begin{proposition}[\Cref{thm: log orbifold introduction}] \label{prop: log orb coincide} There is an equality of logarithmic and orbifold Gromov--Witten invariants:
	\[ \GW^{\mathsf{log}}_{0,\hat{\bfc},\upbeta}(S\, | \, D+E)\langle \upgamma \rangle = \GW^{\mathsf{orb}}_{0,\hat{\bfc},\upbeta}(S\, | \, D+E)\langle \upgamma \rangle.\]
\end{proposition}

\begin{proof} Let $\Sigma$ be the tropicalisation of $(S \, | \, D+E)$. A kaleidoscopic double cover is depicted in Figure~\ref{fig: tropicalisation of bicyclic pair}. Given a tropical type of map to $\Sigma$, the balancing conditions of Sections~\ref{sec: balancing on E}~and~\ref{sec: balancing on D=D0} apply verbatim.
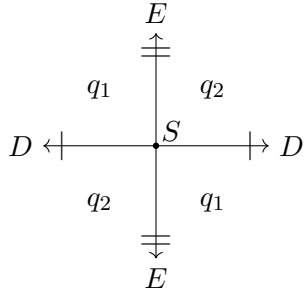
\begin{figure}[h]
	\begin{tikzpicture}[scale=0.5,baseline=(current  bounding  box.center)]
			\draw[fill=black] (0,0) circle[radius=2pt];
			\draw (0.4,0.4) node{$S$};
			
			\draw[->] (0,0) -- (3,0);
			\draw (3,0) node[right]{$D$};
			
			\draw[->] (0,0) -- (0,3);
			\draw (0,3) node[above]{$E$};
			
			\draw[->] (0,0) -- (-3,0);
			\draw (-3,0) node[left]{$D$};
			
			\draw[->] (0,0) -- (0,-3);
			\draw (0,-3) node[below]{$E$};
				
			\draw (2.5,0) node{$\mid$};
			\draw (-2.5,0) node{$\mid$};
			
			\draw (0,2.5) node[rotate=90]{$\mid\mid$};
			\draw (0,-2.5) node[rotate=90]{$\mid\mid$};
			
			\draw (1.5,1.5) node{$q_2$};
			\draw (-1.5,1.5) node{$q_1$};
			\draw (1.5,-1.5) node{$q_1$};
			\draw (-1.5,-1.5) node{$q_2$};
		\end{tikzpicture}
		\caption{Representing $\Sigma$ as a quotient of a double cover.}
		\label{fig: tropicalisation of bicyclic pair}
\end{figure}

By \cite[Theorem~X]{RootLog} it is sufficient to show that $(S\, |\, D+E)$ is slope-sensitive with respect to the given numerical data \cite[Section~4.1]{RootLog}. Fix a naive type of tropical map to $\Sigma$ as in \cite[Section~3]{RootLog}. We must show that there is no oriented edge $\vec{e}$ of the source graph $\Gamma$ such that the associated cone $\upsigma_e \in \Sigma$ is maximal and the slope $m_{\vec{e}} \in N_{\upsigma_e}$ belongs to the positive quadrant.

	We begin with a useful construction. Given an oriented edge $\vec{e} \in \vec{E}(\Gamma)$ terminating at a vertex $v \in V(\Gamma)$ with $\upsigma_e \in \{q_1,q_2\}$ and $\upsigma_v \in \{q_1,q_2,D\}$, we let
	\[ \Gamma(\vec{e}\,) \subseteq \Gamma \]
	denote the maximal connected subgraph which contains $\cev{e}$ as an outgoing half-edge and is such that all vertices and half-edges have associated cones $q_1,q_2$, or $D$. This means that all outgoing half-edges besides $\cev{e}$ are either unbounded marking legs, or finite edges terminating at $E$ or $S$ (the vertical dividing line in Figure~\ref{fig: tropicalisation of bicyclic pair}).
	
	Now suppose for a contradiction that there exists an oriented edge $\vec{e}_1 \in \vec{E}(\Gamma)$ such that $\upsigma_{e_1} = q_1$ and $m_{\vec{e}_1} \in N_{q_1}$ belongs to the positive quadrant. The assumption on the slope ensures that the subgraph $\Gamma(\vec{e}_1)$ is well-defined. We claim that $\Gamma(\vec{e}_1)$ contains the (unique) marking leg with positive tangency to $D$. Since $D^2 \geq 0$ it follows by balancing (Section~\ref{sec: balancing on D=D0}) that at every vertex of $\Gamma(\vec{e}_1 )$ the sum of the outgoing slopes in the $D$-direction is non-negative. Summing over all the vertices, we see that the sum of outgoing slopes from $\Gamma(\vec{e}_1 )$ in the $D$-direction is non-negative. The slope of $\cev{e}_1 $ in the $D$-direction is negative, hence there exists an outgoing edge whose slope in the $D$-direction is positive. Such an edge cannot terminate at $E$ or $S$ and so, by the definition of $\Gamma(\vec{e}_1 )$, it must be the marking leg with positive tangency to $D$.
	
	A similar argument shows that $\Gamma(\vec{e}_1 )$ also has an outgoing edge with positive slope in the $E$-direction. Since $E \cdot \upbeta=0$ there are no marking legs with tangency to $E$, and so this must be a finite edge terminating at a vertex $v_0$ with $\upsigma_{v_0} = E$. By balancing (Section~\ref{sec: balancing on E}) we see that $v_0$ supports an outgoing edge $\vec{e}_2$ with $\upsigma_{e_2} = q_2$.
	
	The same argument as above now shows that $\Gamma(\vec{e}_2 )$ also contains the marking leg with positive tangency to $D$. However, $\Gamma(\vec{e}_1)$ and $\Gamma(\vec{e}_2 )$ are disjoint: deleting the vertex $v_0$ separates them, since $\Gamma$ has genus zero.
\end{proof}

\begin{remark} \label{rmk: orbifold correspondence strong} While restricted to genus zero, \Cref{thm: exceptional divisors correspondence} is strong in that it establishes an equality of virtual fundamental classes. This contrasts with \Cref{thm: local-log} which also establishes an equality of Chow classes, but only after capping with suitable insertions. In the latter case, we expect that even in genus zero, the counterexamples of \cite[Sections~1~and~3.7]{MaxContacts} can be adapted to produce pathological insertions (such as naked psi classes) violating the correspondence.
\end{remark}

\section{Toric and open geometries}\label{sec: toric pairs}
\noindent In this section we restrict to toric targets. The main result (\Cref{thm: local open correspondence}) equates the Gromov--Witten theories of the local and open geometries:
\[ (\OO_S(-D)\, |\, \hat{E}) \leftrightarrow \OO_S(-D)|_{S\setminus E}.\]
The proof proceeds via torus localisation. For the open geometry, the computation is controlled by the topological vertex \cite{LiLiuLiuZhouTopVert}. The difficult step is to show that the contributions of certain localisation graphs vanish.

\subsection{Setup} \label{sec: setup toric} We retain the setup of \Cref{sec: setup} and introduce the following additional assumptions:
\begin{itemize}
	\item $S$ is a toric surface.
	\item $E$ is a toric hypersurface.
	\item $D+E \in |\!-\!\!K_S|$.
	\item $E \cdot \upbeta=0$.
\end{itemize}
We do not require that $D$ is toric. An important example is the resolution of an irreducible self-nodal cubic in the plane (Section~\ref{sec: nodal cubic}).

With these assumptions, the enumerative setup of Section~\ref{sec: setup} specialises. There is a single marked point $x$ with tangency $D \cdot \upbeta$ along $D$, and no marked points with tangency along $E$. The space of stable logarithmic maps has virtual dimension $g$ and we consider the Gromov--Witten invariant with a lambda class and no additional insertions:
\[ \GW_{g,(D \cdot \upbeta,0),\upbeta}(S\, | \, D+E)\langle (-1)^g \uplambda_g \rangle \colonequals (-1)^g \uplambda_g \cap [\Mbar_{g,(D \cdot \upbeta,0),\upbeta}(S\, |\, D+E)]^{\virt} \in \Q.\]
The space of stable logarithmic maps to the local target $(\OO_S(-D)\, | \, \hat{E})$ has virtual dimension zero and we consider the Gromov--Witten invariant with no insertions:
\[ \GW_{g,0,\upbeta}(\OO_S(-D)\, | \, \hat{E}) \colonequals [\Mbar_{g,0,\upbeta}(\OO_S(-D)\, |\, \hat{E})]^{\virt} \in \Q.\]
Theorem~\ref{thm: local-log} furnishes a correspondence between these invariants. In this section we relate the latter to the invariants of the open target $\OO_S(-D)|_{S \setminus E}$.

\subsection{Open invariants} \label{sec: open invariants} We establish conventions for toric geometry. We consider fans $\Sigma$ not contained in a proper linear subspace of the ambient lattice; these correspond to toric varieties with no torus factors. We write $\Sigma(k)$ for the set of $k$-dimensional cones. Letting $n$ denote the dimension of the ambient lattice, we introduce notation for closed toric strata appearing in critical dimensions:
\begin{itemize}
	\item For $\uprho \in \Sigma(1)$ we denote the corresponding toric hypersurface $D_\uprho$.
	\item For $\uptau \in \Sigma(n-1)$ we denote the corresponding toric curve $L_\uptau$.
	\item For $\upsigma \in \Sigma(n)$ we denote the corresponding torus-fixed point $P_\upsigma$.
\end{itemize}
Write $\Sigma_S$ for the fan of $S$ and $\uprho_E \in \Sigma_S(1)$ for the cone corresponding to $E$. Since $D+E \in |\!-\!\!K_S|$ we have the following identity in the class group of $S$:
\begin{equation} \label{eqn: D in terms of toric boundary} D = \sum_{\substack{\uprho \in \Sigma_S(1)\\ \uprho \neq \uprho_E}} D_\uprho.\end{equation}
Set $X \colonequals \OO_S(-D)$ and consider the open subvariety:
\begin{equation*} X^{\circ} \coloneqq \calO_{S}(-D)|_{S \setminus E}.\end{equation*}
Equip $X^\circ$ with the trivial logarithmic structure and $X$ with the logarithmic structure induced by $E$. The open embedding $\upiota \colon X^\circ \hookrightarrow X$ is strict.

Following the formalism of the topological vertex \cite{LiLiuLiuZhouTopVert} we define Gromov--Witten invariants of the open manifold $X^\circ$ by localising with respect to an appropriate torus.

\begin{definition}[\!{\cite[Section~3.1]{LiLiuLiuZhouTopVert}}]
Let $P \in X^\circ$ be a torus-fixed point. Consider the action of the three-dimensional dense torus on $\bigwedge\!\!{}^3\, T_P X^{\circ}$ and let $\upchi_P$ denote the associated character. The \textbf{Calabi--Yau torus} is denoted and defined
\[ T\coloneqq \operatorname{Ker}\upchi_P.\]
It is a two-dimensional subtorus of the dense torus. The definition is independent of the choice of $P$ because $X^\circ$ is Calabi--Yau. 
\end{definition}

While the moduli space of stable maps to $X^{\circ}$ is non-proper, its $T$-fixed locus is proper. This is used to define Gromov--Witten invariants, via localisation. Let $Q_T$ denote the localisation of $A^\star_T(\mathrm{pt})$ at the set of homogeneous elements of non-zero degree, and let $Q_{T,k}$ denote its $k$th graded piece.

\begin{definition} The \textbf{$T$-localised Gromov--Witten invariant} of $X^\circ$ is denoted and defined:
\[
	\GW_{g,0,\upiota^\star \upbeta}^T(X^{\circ}) \coloneqq \int_{[\Mbar_{g,0,\upiota^\star \upbeta}(X^{\circ})^T]^{\virt}_T} \frac{1}{e^{T}(N^{\virt})} \in Q_{T,0}.
\]
A priori this is a rational function in the equivariant weights, with numerator and denominator homogeneous polynomials of the same degree. However \cite[Theorem 4.8]{LiLiuLiuZhouTopVert} shows that the numerator and denominator are in fact constant, so that:
\[ 	\GW_{g,0,\upiota^\star \upbeta}^T(X^{\circ}) \in \Q.\]
For this it is crucial to restrict to the Calabi--Yau torus.
\end{definition}
The main result of this section is the following:
	\begin{theorem}[\Cref{thm: local open introduction}] \label{thm: local open correspondence} For all $g\geq 0$ we have:
		\begin{equation}
			\label{eq:log loc intermediate}
			\GW_{g,0,\upbeta} (X|\hat{E}) = \GW_{g,0,\upiota^\star \upbeta}^T(X^{\circ}).
		\end{equation}    
	\end{theorem}

\subsection{Localisation calculation} \label{sec: localisation calculation}

We prove \Cref{thm: local open correspondence} via virtual torus localisation \cite{GraberPandharipande} which expresses the left-hand side of \eqref{eq:log loc intermediate} as a sum of contributions from the $T$-fixed loci of the moduli stack. First, in \Cref{sec: comparison fixed loci} we identify the contributions of torus fixed points associated to stable maps factoring through $X^\circ \hookrightarrow X$ with the $T$-localised Gromov--Witten invariant of $X^\circ$. To establish equation \eqref{eq:log loc intermediate} it is therefore sufficient to show that all remaining fixed loci contribute trivially. We prove the vanishing in \Cref{sec: localisation vanishing of remaining contributions}. This last step crucially uses that we localised with respect to the Calabi--Yau torus $T$ (see the proof of \Cref{thm: local open correspondence}) and not with respect to the dense open torus which generally does not yield the required vanishing.

\subsubsection{Fixed loci} \label{sec:torusfixedlocus}

The action $T \curvearrowright X$ lifts to actions on $\Mbar_{g,0,\upbeta}(X)$ and $\Mbar_{g,0,\upbeta}(X|\hat{E})$ (in the latter case, this is because $T$ sends $\hat{E}$ to itself and hence lifts to a logarithmic action $T \curvearrowright (X|\hat{E})$).

Restricting from the dense torus of $X$ to the subtorus $T$ does not change the zero- and one-dimensional orbits in $X$. It follows that it also does not change the fixed locus in the moduli space of stable maps. Since $X$ is a toric variety, this fixed locus is well-understood, see e.g. \cite[Section~6]{SpielbergArxiv}, \cite[Section~5.2]{LiuLocalizationGW}, \cite[Section~4]{GraberPandharipande}, \cite[Section~9.2]{CoxKatz}, or \cite[Section~4]{BehrendLocalizationGW}.

Briefly, the fixed locus decomposes into a union of connected components indexed by \textbf{localisation graphs}. A localisation graph $\Gamma$ is a graph equipped with marking legs, degree labelings $d_e > 0$ for every edge $e \in E(\Gamma)$ and genus labelings $g_v \geq 0$ for every vertex $v \in V(\Gamma)$. Furthermore every vertex $v \in V(\Gamma)$ is assigned a cone $\upsigma(v) \in \Sigma_X(3)$ and every edge $e \in E(\Gamma)$ is assigned a cone $\upsigma(e) \in \Sigma_X(2)$. The corresponding connected component of the fixed locus is denoted
\[ F_\Gamma(X) \]
and generically parametrises stable maps with components $C_v$ contracted to torus-fixed points and components $C_e$ forming degree $d_e$ covers of toric curves, totally ramified over the torus-fixed points.

We let $\Omega_{g,0,\upbeta}(X)$ denote the set of localisation graphs, so that:
\[ \Mbar_{g,0,\upbeta}(X)^T = \bigsqcup_{\Gamma \in \Omega_{g,0,\upbeta}(X)} F_\Gamma(X).\]

\subsubsection{Comparison of fixed loci}\label{sec: comparison fixed loci} Consider the morphism forgetting the logarithmic structures:
\[ \Mbar_{g,0,\upbeta}(X|\hat{E}) \to \Mbar_{g,0,\upbeta}(X).\]
This is $T$-equivariant, and hence restricts to a morphism between $T$-fixed loci. For each $\Gamma \in \Omega_{g,0,\upbeta}(X)$ we define $F_\Gamma(X|\hat{E})$ via the fibre product
\begin{equation*}
	\begin{tikzcd}
		F_{\Gamma}(X|\hat{E}) \arrow[r,hookrightarrow] \arrow[d] \ar[rd,phantom,"\square"] & \Mbar_{g,0,\upbeta}(X|\hat{E})^T \arrow[d] \\
		F_{\Gamma}(X) \arrow[r,hookrightarrow] & \Mbar_{g,0,\upbeta}(X)^T
	\end{tikzcd}
\end{equation*}
and this produces a decomposition of $\Mbar_{g,0,\upbeta}(X|\hat{E})^T$ into clopen substacks:
\[ \Mbar_{g,0,\upbeta}(X|\hat{E})^T = \bigsqcup_{\Gamma \in \Omega_{g,0,\upbeta}(X)} F_\Gamma(X|\hat{E}). \]
Note that we do not claim that each $F_\Gamma(X|\hat{E})$ is connected, nor that $F_\Gamma(X|\hat{E}) \to F_\Gamma(X)$ is virtually birational. Virtual localisation \cite{GraberPandharipande} gives:\footnote{Virtual localisation for spaces of stable logarithmic maps presents conceptual difficulties, as the obstruction theory is defined over the Artin fan which is typically singular. Since the divisor $E \subseteq X$ is smooth, we circumvent these issues by passing to Kim's space of expanded logarithmic maps \cite{KimLog}, which has an absolute obstruction theory and arises as a logarithmic modification of the Abramovich--Chen--Gross--Siebert space (see e.g. \cite[Section~2.1]{BNRGenus1}). The arguments of this section are insensitive to the choice of birational model of the moduli space. See \cite{MolchoRoutis} for a treatment of localisation in this setting.}
\begin{equation}
	\GW_{g,0,\upbeta}(X|\hat{E}) = \sum_{\Gamma \in \Omega_{g,0,\upbeta}(X)} \int_{[F_\Gamma(X|\hat{E})]^{\virt}_T} \frac{1}{e^T(N^{\virt}_{F_\Gamma(X|\hat{E})})}.
\end{equation}
Turning to $X^\circ$ we note that there is an inclusion
\[ \Omega_{g,0,\upiota^\star \upbeta}(X^\circ) \subseteq \Omega_{g,0,\upbeta}(X) \]
consisting of localisation graphs which do not interact with cones in $\Sigma_X \setminus \Sigma_{X^\circ}$. Since the logarithmic structure on $(X|\hat{E})$ is trivial when restricted to $X^\circ$ it follows that for $\Gamma \in \Omega_{g,0,\upiota^\star \upbeta}(X^\circ)$ we have
\[  F_\Gamma(X|\hat{E}) = F_\Gamma(X) = F_\Gamma(X^\circ).\]
The perfect obstruction theories coincide when restricted to these loci, producing an identification of the induced virtual fundamental classes and virtual normal bundles. We conclude:
\begin{proposition}
	\label{prop: local open localisation comparisons}
	We have:
	\begin{equation*}
		\GW_{g,0,\upbeta} (X|\hat{E}) = \GW_{g,0,\upiota^\star \upbeta}^T(X^{\circ}) + \sum_{\Gamma \in \Omega_{g,0,\upbeta}(X) \setminus \Omega_{g,0,\upiota^\star \upbeta}(X^{\circ})} \int_{[F_\Gamma(X|\hat{E})]^{\virt}_T} \frac{1}{e^T(N^{\virt}_{F_\Gamma(X|\hat{E})})}.
	\end{equation*}
\end{proposition}

\subsubsection{Vanishing of remaining contributions} \label{sec: localisation vanishing of remaining contributions}
Fix a localisation graph $\Gamma \in \Omega_{g,0,\upbeta}(X) \setminus \Omega_{g,0,\upiota^\star \upbeta}(X^\circ)$. To prove \Cref{thm: local open correspondence} it remains to show
\[\int_{[F_\Gamma(X|\hat{E})]^{\virt}_T} \frac{1}{e^T(N^{\virt}_{F_\Gamma(X|\hat{E})})}=0.\]
This requires a detailed analysis of the shape of the localisation graph and its contribution.

\begin{notation} Local to $E \subseteq S$ the toric boundary takes the following form
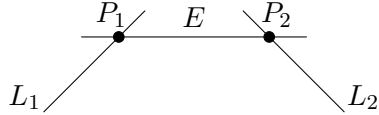
\[
\begin{tikzpicture}

\draw (-1.5,0) -- (1.5,0);
\draw (0,0) node[above]{$E$};

\draw (-0.65,0.35) -- (-2,-1);
\draw (-1.9,-0.8) node[left]{$L_1$};
\draw[fill=black] (-1,0) circle[radius=2pt];
\draw (-1.1,0) node[above]{$P_1$};

\draw (0.65,0.35) -- (2,-1);
\draw (1.9,-0.8) node[right]{$L_2$};
\draw[fill=black] (1,0) circle[radius=2pt];
\draw (1.1,0) node[above]{$P_2$};
	
\end{tikzpicture}
\]
The zero section gives a closed embedding $S \hookrightarrow X$ as a union of toric boundary strata. Let
\[ \uptau_E,\uptau_1,\uptau_2 \in \Sigma_X(2)\]
denote the cones corresponding to the toric curves $E,L_1,L_2 \hookrightarrow S \hookrightarrow X$. Similarly let 
\[ \upsigma_1,\upsigma_2 \in \Sigma_X(3)\]
denote the cones corresponding to the torus-fixed points $P_1,P_2 \in S \hookrightarrow X$.
\end{notation}

\begin{lemma} \label{lem:ExistsExcComp} There exists an edge $\tilde{e} \in E(\Gamma)$ with $\upsigma(\tilde{e}) = \uptau_E$.
\end{lemma}
\begin{proof}
	Suppose for a contradiction that $\upsigma(e) \neq \uptau_E$ for all $e \in E(\Gamma)$. Since $\Gamma \not\in \Omega_{g,0,\iota^\star \upbeta}(X^\circ)$ there exists a vertex $v \in V(\Gamma)$ with $\upsigma(v) \in \{\upsigma_1,\upsigma_2\}$. This vertex is adjacent to an edge $\tilde{e} \in E(\Gamma)$, and since $\upsigma(\tilde{e}) \neq \uptau_E$ we must have $\upsigma(\tilde{e}) \in \{ \uptau_1,\uptau_2\}$. We then find
	\[ E \cdot \upbeta = \sum_{e \in E(\Gamma)} d_e \left(E \cdot L_{\upsigma(e)}\right) = \sum_{\substack{e \in E(\Gamma)\\ \upsigma(e) \in \{\uptau_1,\uptau_2\}}} d_e \geq d_{\tilde{e}} > 0 \]
	which contradicts $E \cdot \upbeta=0$.
\end{proof}

\begin{lemma}
	\label{lem:weightsTanSp}
	The following relation holds in $A^1_T(\mathrm{pt})$: 
	\begin{equation*}
		c_1^T(T_{P_1} E) + c^{T}_1 \left(\calO_{S}(-D)|_{P_1} \right) = 0.
	\end{equation*}
\end{lemma}

\begin{proof}
	\Cref{fig:toricskel} illustrates the toric skeleton of $X$ in a neighbourhood of $L_1 \cup E$. The horizontal edges index boundary curves contained in the zero section $S \hookrightarrow X$ while the vertical edges index fibres of the projection $X \to S$ over torus-fixed points. We define:
	\begin{equation*}
		u_1 \coloneqq c_1^T(T_{P_1} E), \qquad \qquad u_2\coloneqq  c_1^T(T_{P_1} L_1).
	\end{equation*}
We now calculate the weights of the $T$-action on $T_{P_0}L_0$ and $T_{P_0}L_1$. The standard theory of torus actions on projective lines gives:
	\begin{equation} \label{eqn: wt on L1} c_1^T(T_{P_0}L_1) = - c_1^T(T_{P_1}L_1) = -u_2. \end{equation}
Turning to $T_{P_0}L_0$ we have natural identifications:
\[ T_{P_0}L_0 = N_{L_1|S}|_{P_0}, \qquad T_{P_1}E = N_{L_1|S}|_{P_1}.\]
Let $a_1 \colonequals \deg N_{L_1|S}$ denote the self-intersection of the divisor $L_1 \subseteq S$. We then have:
\begin{equation} \label{eqn: wt on L0} c_1^T(T_{P_0}L_0) = c_1^T(N_{L_1|S}|_{P_0}) = c_1^T(N_{L_1|S}|_{P_1}) - a_1 c_1^T(T_{P_1}L_1) = u_1 - a_1u_2.\end{equation}
	From the definition of the Calabi--Yau torus and \eqref{eqn: wt on L1}, \eqref{eqn: wt on L0} we obtain
\[ 0 = c^{T}_1 \left(\textstyle\bigwedge^{\!3} T_{P_0} X \right) =  c^{T}_1 (T_{P_0} F_0 ) + (-u_2) + (u_1 - a_1 u_2) \]
from which we deduce:
\[ c^{T}_1 (T_{P_0} F_0) = -u_1 + (a_1+1) u_2.\]
By \eqref{eqn: D in terms of toric boundary} we have $\calO_{S}(-D)|_{L_1} \cong \calO_{\PP^1}(-a_1-1)$ from which we conclude
\[
		c^{T}_1 (T_{P_1} F_1) = c^{T}_1 (T_{P_0} F_0) + ( a_1 + 1) c_1^T (T_{P_0} L_1) = -u_1 = - c_1^T(T_{P_1}E)
\]
	which completes the proof.
\end{proof}

	\begin{figure}[t]
		\centering
		\begin{tikzpicture}[smooth, baseline={([yshift=-.5ex]current bounding box.center)}]%
			\draw[thick] (0,0) to (3,0);
			\draw[thick] (0,0) to (-3,0);
			\draw[thick] (0,0) to (0,1.5);
			\draw[thick,dashed] (0,1.5) to (0,2);
			\draw[thick] (3,0) to (5,0);
			\draw[thick] (3,0) to (3,1.5);
			\draw[thick,dashed] (3,1.5) to (3,2);
			\draw[thick] (-3,0) to (-3,1.5);
			\draw[thick,dashed] (-3,1.5) to (-3,2);
			\draw[thick] (-3,0) to (-5,0);
			\node[above] at (1.5,0) {$E$};
			\node[above] at (-1.5,0) {$L_1$};
			\node[above] at (-4,0) {$L_0$};
			\node[above] at (0,2) {$F_1$};
			\node[above] at (-3,2) {$F_0$};
			\node[above] at (3,2) {$F_2$};
			\node[below] at (0,0) {$P_1$};
			\node[below] at (3,0) {$P_2$};
			\node[below] at (-3,0) {$P_0$};
			\node at (0,0) {$\bullet$};
			\node at (3,0) {$\bullet$};
			\node at (-3,0) {$\bullet$};
			\node[below,purple] at (0.7,0.05) {$u_1$};
			\node[left,purple] at (0.1,1) {$-u_1$};
			\node[below,purple] at (-0.7,0.05) {$u_2$};
			\node[below,purple] at (2.4,0.1) {$-u_1$};
			\node[left,purple] at (3.1,1) {$u_1$};
			\node[below,purple] at (-2.3,0.1) {$-u_2$};
			\node[below,purple] at (-4.2,0.1) {$u_1 - a u_2$};
			\node[left,purple] at (-2.9,1) {$-u_1 + ( 1 + a) u_2$};
		\end{tikzpicture}%
		\caption{The toric skeleton of $X$ locally around $L_1 \cup E$, used in the proof of \Cref{lem:weightsTanSp}. Edges represent boundary curves and vertices represent torus-fixed points. The purple label at a flag $(P,L)$ records the weight $c^{T}_1(T_{P} L)$.}
		\label{fig:toricskel}
	\end{figure}
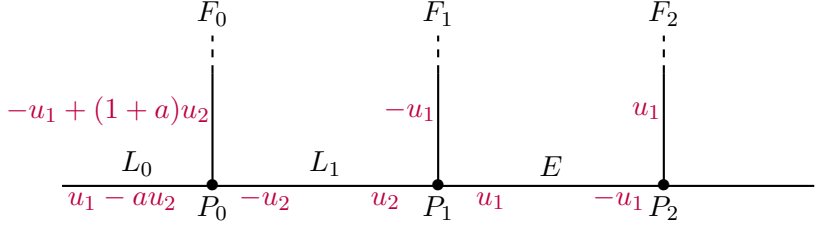

\begin{proof}[Proof of \Cref{thm: local open correspondence}] Since $D^2 \geq 0$ and $D \cdot \upbeta > 0$ it follows that
\[ \Mbar_{g,0,\upbeta}(X|\hat{E}) = \Mbar_{g,0,\upbeta}(S|E). \]
Since $(S|E) \hookrightarrow (X|\hat{E})$ is strict, there is a short exact sequence
\[ 0 \to T^{\log}_{S|E} \to T^{\log}_{X|E}\big\rvert_S \to N_{S|X} \to 0\]
and using $N_{S|X} = \OO_S(-D)$ we obtain:
\[ [\Mbar_{g,0,\upbeta}(X|\hat{E})]^{\virt}_T = e^T (\mathbf{R}^{\!1} \uppi_\star f^\star \OO_S(-D))\cap [\Mbar_{g,0,\upbeta}(S|E)]^{\virt}_T.\]
Fix a graph $\Gamma \in \Omega_{g,0,\upbeta}(X) \setminus \Omega_{g,0,\upiota^\star \upbeta}(X^\circ)$. By \Cref{prop: local open localisation comparisons} it suffices to show that the contribution of $\Gamma$ vanishes. We will prove that the $T$-equivariant vector bundle
\[ \mathbf{R}^{\!1} \uppi_\star f^\star \OO_S(-D)|_{F_\Gamma(X|\hat{E})} \]
has a weight zero summand in $K$-theory. This ensures that the $T$-equivariant Euler class vanishes, and the claim follows.

Perform the partial normalisation of the source curve at the nodes forced by the localisation graph $\Gamma$ (such nodes correspond to flags based at either a vertex of valency at least three, or a vertex of valency two which is the intersection of two bounded edges). The normalisation sequence produces a surjection:
\[  H^1(C,f^\star \OO_S(-D)) \twoheadrightarrow \bigoplus_{e \in E(\Gamma)} H^1(C_e,f^\star \OO_S(-D)).\]
It suffices to show that one summand of the codomain has vanishing equivariant Euler class. By \Cref{lem:ExistsExcComp} there exists an edge $\tilde{e} \in E(\Gamma)$ with $\upsigma(\tilde{e}) = \uptau_E$. Every point of $F_\Gamma(X|\hat{E})$ parametrises a stable logarithmic map whose underlying curve contains an irreducible component $C_{\tilde{e}}$ which maps to $E$ with positive degree $d_{\tilde{e}}$ and is totally ramified over the torus-fixed points. Using \Cref{lem:weightsTanSp} we write 
	\begin{equation*}
		u_1 = - c^{T}_1 \left(\calO_{S}(-D)|_{P_1}\right) = c_1^T(T_{P_1} E).
	\end{equation*}
A Riemann--Roch calculation (see e.g. \cite[Example~19]{LiuLocalizationGW}) then shows that:
	\begin{equation*}
		\mathrm{ch}_T \left(H^1(C_{\tilde{e}},f^\star \calO_{S}(-D)) \right) = \sum_{j=1}^{2 d_{\tilde{e}} - 1} \exp(-u_1 + j \tfrac{u_1}{d_{\tilde{e}}}).
	\end{equation*}
	Taking $j=d_{\tilde{e}}$ we see that $H^1(C_{\tilde{e}},f^\star \OO_S(-D))$ has a vanishing Chern root, and hence its equivariant Euler class vanishes as claimed.
\end{proof}

\subsection{Applications: GV, BPS, and quiver DT invariants} \label{sec: GV BPS DT} \label{sec: appli} \Cref{thm: local open correspondence} has several consequences for adjacent curve counting theories, which we now elaborate. 

\subsubsection{Logarithmic BPS and open Gopakumar--Vafa} We fix a saturated additive subset
\[ \Rcal \subseteq A_1(S;\Z) \]
consisting of effective curve classes and such that for every $\upbeta \in \Rcal$ we have $D \cdot \upbeta > 0$ and $E \cdot \upbeta = 0$ as in \Cref{sec: setup toric}. In particular, $0\notin \Rcal$.

We introduce the invariants of interest. By \cite[Lemma 8.4]{BousseauQuantumTropicalVertex} for each $\upbeta \in \Rcal$ there is a rational function
\[ \ol{\Omega}_\upbeta^{(S|D+E)}\big(q^{\frac{1}{2}}\big)\in\Q\big(q^{\pm\frac{1}{2}}\big) \]
such that after the change of variables $q=\mathrm{e}^{\mathrm{i}\hbar}$ we have:
\[
\ol{\Omega}^{(S|D+E)}_\upbeta\big(q^{\frac{1}{2}}\big)= (-1)^{D \cdot \upbeta -1} \left( 2 \sin\left( \frac{\hbar}{2} \right) \right) \sum_{g \geq 0} \GW_{g,(D\cdot\upbeta,0),\upbeta}(S\, | \, D+E) \langle (-1)^g \uplambda_g \rangle \, \hbar^{2g-1}.
\]
Moreover, if we consider the identity
\begin{equation} \label{eqn: relating Omega and Omegabar}
\ol{\Omega}^{(S|D+E)}_\upbeta\big(q^{\frac{1}{2}}\big) = \sum_{\substack{\ell \in \Z_{>0},\, \upbeta'\in\Rcal \\ \ell \upbeta^\prime = \upbeta}} \dfrac{1}{\ell} \, \dfrac{q^{\frac{1}{2}}-q^{-\frac{1}{2}}}{q^{\frac{\ell}{2}}-q^{-\frac{\ell}{2}}} \, \Omega_{\upbeta^\prime}^{(S|D+E)}\big(q^{\frac{\ell}{2}}\big)
\end{equation}
then the resulting implicitly-defined function $\Omega^{(S|D+E)}_{\upbeta}(q^{\frac{1}{2}})$ is a palindromic Laurent polynomial with integer coefficients \cite[Theorem 8.5]{BousseauQuantumTropicalVertex}:
\begin{equation} \label{eqn: refined BPS} \Omega^{(S|D+E)}_\upbeta\big(q^{\frac{1}{2}}\big)\in\Z\big[q^{\pm\frac{1}{2}}\big]. \end{equation}
Following \cite{BousseauQuantumTropicalVertex} we refer to this as the \textbf{refined BPS invariant} of $(S\,|\,D+E)$. The identity \eqref{eqn: relating Omega and Omegabar} is equivalent to:
\[
\sum_{\upbeta \in \Rcal} \ol{\Omega}^{(S|D+E)}_\upbeta\big(q^{\frac{1}{2}}\big) z^{\upbeta}= \sum_{\upbeta \in \mathcal{R}} \sum_{k\geq1} \dfrac{1}{k} \, \dfrac{q^{\frac{1}{2}}-q^{-\frac{1}{2}}}{q^{\frac{k}{2}}-q^{-\frac{k}{2}}} \, \Omega^{(S|D+E)}_{\upbeta}\big(q^{\frac{k}{2}}\big) \, z^{k\upbeta}.
\]

Turning to the open geometry $X^\circ = \OO_S(-D)|_{S \setminus E}$ the \textbf{Gopakumar--Vafa invariants} $n_{g,\upbeta}^T(X^\circ)$ are defined recursively as the unique (a priori) rational numbers that satisfy the following equation:
\begin{equation} \label{eqn: definition of GV}
\sum_{\substack{g \geq 0 \\ \upbeta \in \Rcal}} \, \GW_{g,\upbeta}^T(X^{\circ}) \, \hbar^{2g-2} \, z^\upbeta = \sum_{\substack{g \geq 0 \\ \upbeta \in \Rcal }} \sum_{k\geq 1} \dfrac{n_{g,\upbeta}^T(X^\circ)}{k} \, \left( 2 \sin\left( \frac{k\hbar}{2} \right) \right)^{2g-2} z^{k\upbeta}.
\end{equation}
Under the change of variables $q=\mathrm{e}^{\mathrm{i}\hbar}$ the right-hand side becomes:
\[ \sum_{\substack{g \geq 0 \\ \upbeta \in \Rcal}} \sum_{k \geq 1} \dfrac{n_{g,\upbeta}^T(X^\circ)}{k} \, (-1)^{g-1} \, \left( q^{\frac{k}{2}}-q^{-\frac{k}{2}} \right)^{2g-2} z^{k\upbeta}. \]

The Gopakumar--Vafa conjecture \cite{GopakumarVafaII} (proven for toric Calabi--Yau threefolds in \cite{KonishiGV} and in general in \cite{IonelParkerGV}) states that $n^T_{g,\upbeta}(X^\circ)\in\Z$, and that for fixed $\beta$ we have $n^T_{g,\upbeta}(X^\circ)=0$ for $g$ sufficiently large. Consider the generating function
\begin{equation} \label{eqn: generating function GV}
\Omega^{X^\circ}_\upbeta\!\big(q\big):=\sum_{g\geq0} \, n^T_{g,\upbeta}(X^\circ) \, (-1)^g \, \left( q^{\frac{1}{2}} - q^{-\frac{1}{2}} \right)^{2g}\in\Z\big[q^{\pm1}\big].
\end{equation}
Note that $\Omega^{X^\circ}_\upbeta \!(q)$ is a palindromic polynomial. Finally consider the quantum number:
\begin{equation}
[m]_q := \dfrac{q^{\frac{m}{2}}-q^{-\frac{m}{2}}}{q^{\frac{1}{2}} - q^{-\frac{1}{2}}} =  q^{\frac{m-1}{2}} + q^{\frac{m-3}{2}} + \cdots + q^{-\frac{m-1}{2}}.
\end{equation}
Our first application relates the refined BPS invariants \eqref{eqn: refined BPS} to the generating function for Gopakumar--Vafa invariants \eqref{eqn: generating function GV}.

\begin{corollary} \label{cor: refined} There is an equality of palindromic Laurent polynomials:
\[
\Omega^{(S|D+E)}_{\upbeta}\big(q^{\frac{1}{2}}\big)= [D \cdot \upbeta ]_q \, \Omega_\upbeta^{X^\circ}\!\big(q\big).
\]
In particular, $\Omega^{(S|D+E)}_{\upbeta}\big(q^{\frac{1}{2}}\big)$ is divisible by the quantum number $[D \cdot \upbeta]_q$.
\end{corollary}

\begin{proof}
Taking the exponential of \eqref{eqn: definition of GV} and using our notation \eqref{eqn: generating function GV} we obtain
\[
\label{eq:exp-local-GW}
\exp\left(\sum_{\substack{g \geq 0 \\ \upbeta \in \Rcal}} \, \GW_{g,\upbeta}^T(X^{\circ}) \, \hbar^{2g-2} \, z^\upbeta \right)= {\rm Exp} \left( \sum_{\upbeta \in \Rcal} \dfrac{-1}{\left( q^{\frac{1}{2}} -q^{-\frac{1}{2}} \right)^2 } \, \Omega^{X^\circ}_\upbeta\!\big(q\big) \,  z^{\upbeta} \right)
\]
where ${\rm Exp}$ is the plethystic exponential:
\[
{\rm Exp}\big(f(q,z)\big)=\exp\left(\sum_{k\geq1}\frac{1}{k}f(q^k,z^k)\right).
\]

Combining \Cref{thm: local-log,thm: local open correspondence}, the left-hand side becomes:
\[
\begin{split}
\label{eq:comp-the-pleth}
\exp\left(\sum_{\substack{g \geq 0 \\ \upbeta \in \Rcal}} \, \GW_{g,\upbeta}^T(X^{\circ}) \, \hbar^{2g-2} \, z^\upbeta \right)
=&\exp\left( \sum_{\substack{g \geq 0 \\ \upbeta \in \Rcal}} \, \dfrac{(-1)^{D \cdot \upbeta -1}}{2 \sin\left( \frac{D \cdot \upbeta}{2} \hbar \right)} \, \GW_{g,(D\cdot\upbeta,0),\upbeta}(S\, | \, D+E) \langle (-1)^g \uplambda_g \rangle \, \hbar^{2g-1} z^{\upbeta} \right) \\
=& \exp\left(  \sum_{\upbeta \in \Rcal} \, \dfrac{-1}{\left(q^{\frac{D \cdot \upbeta}{2}}-q^{-\frac{D \cdot \upbeta}{2}}\right) \left(q^{\frac{1}{2}}-q^{-\frac{1}{2}}\right)} \, \ol{\Omega}^{(S|D+E)}_\upbeta\big(q^{\frac{1}{2}}\big) \, z^{\upbeta} \right) \\
=& \exp\left(\sum_{\upbeta \in \Rcal} \sum_{k \geq 1} \, \dfrac{1}{k} \, \dfrac{-1}{\left( q^{k\frac{D \cdot \upbeta}{2}}-q^{-k\frac{D \cdot \upbeta}{2}}\right) \left(q^{\frac{k}{2}}-q^{-\frac{k}{2}}\right)} \, \Omega^{(S|D+E)}_{\upbeta}\big(q^{\frac{k}{2}}\big) \, z^{k\upbeta} \right) \\
=& {\rm Exp}\left( \sum_{\upbeta \in \Rcal} \, \dfrac{-1}{\left(q^{\frac{D \cdot \upbeta}{2}}-q^{-\frac{D \cdot \upbeta}{2}}\right) \left(q^{\frac{1}{2}}-q^{-\frac{1}{2}}\right)} \, \Omega^{(S|D+E)}_{\upbeta}\big(q^{\frac{1}{2}}\big) \, z^{\upbeta} \right).
\end{split}
\]
Taking the plethsytic logarithm we obtain for all $\upbeta \in \Rcal$:
\[
\Omega_\upbeta^{(S|D+E)}\big(q^{\frac{1}{2}}\big) = \dfrac{q^{\frac{D \cdot \upbeta}{2}} - q^{-\frac{D \cdot \upbeta}{2}}}{q^{\frac{1}{2}}-q^{-\frac{1}{2}}} 
\, \Omega_\upbeta^{X^\circ}\!\big(q\big).\qedhere
\]
\end{proof}

\subsubsection{Quiver DT invariants}
\label{sec: quiver DT}
We follow the construction in \cite[Section~8.5]{BousseauQuantumTropicalVertex} to associate a quiver to the pair $(S | D+E)$ (see also \cite{ABQuiver,MeinRei,ReinekeWeist}). A concrete example is studied in \Cref{sec: scattering}.

First we use \cite[Proposition~1.3]{GHK} to pass to a toric model:
\begin{equation*}
\begin{tikzcd}
	& (\widetilde{S} \, | \, \widetilde{D}) \ar[dl, "\varphi", swap] \ar[dr, "\uppi"]& \\
	(S \, | \, D+E) & & (\ol{S} \, | \, \ol{D}).
\end{tikzcd}
\end{equation*}
Here, the morphisms $\varphi$ and $\uppi$ are logarithmic modifications where $\varphi$ is a sequence of blowups along zero dimensional strata and $\uppi$ is the blowup of $\smash{\ol{S}}$ at distinct smooth points of $\smash{\ol{D}}$. The pair $(\ol{S} | \ol{D})$ is a toric surface together with its toric boundary.

We now associate a quiver $Q$ to the toric model $\uppi$ as follows. The set of vertices is in bijection with the points $p_1,\ldots,p_n$ of $\smash{\ol{D}}$ which are blown up under $\uppi$. The number of arrows from the vertex corresponding to $p_i$ to the vertex corresponding to $p_j$ is taken to be $\max (\uprho_i \wedge \uprho_j, 0)$ where $\uprho_i$ and $\uprho_j$ are the primitive generators of the rays in the fan of $\smash{\ol{S}}$ corresponding to the toric divisors on which $p_i$ and $p_j$ lie respectively. Finally, given a curve class $\upbeta\in \Rcal$ on $S$ we fix a dimension vector $d(\upbeta)$ of $Q$ by declaring its entries to be the intersection numbers of $\upbeta$ with the pushforward of the exceptional loci $L_i = \pi^{-1}(p_i)$ under $\varphi$:
\begin{equation*}
	d(\beta) \coloneqq \big(\varphi_\star L_1 \cdot \upbeta , \ldots, \varphi_\star  L_n \cdot \upbeta \big).
\end{equation*}
We consider the associated moduli space 
\[ M_{d(\upbeta)}^{\uptheta-\text{ss}}(Q) \]
of $\uptheta$-semistable quiver representations, where $\uptheta$ is the so-called generic anti-attractor stability condition. 
We define the associated \textbf{refined quiver Donaldson--Thomas invariant} as the shifted Poincar\'e polynomial
\begin{align*}
\Omega^{Q}_{d(\upbeta)}\big(q^{\frac{1}{2}}\big)\colonequals & \big(- q^{-\frac{1}{2} }\big)^{\dim M_{d(\upbeta)}^{\uptheta-\text{ss}}(Q)} \sum_{j \geq 0} \dim H^{2j}\left(M_{d(\upbeta)}^{\uptheta-\text{ss}}(Q) ,\iota_{!\star}\Q\right) q^j 
\end{align*}
where $\upiota$ is the inclusion of the stable locus.

\begin{theorem}[\!{\cite[Theorem~8.13]{BousseauQuantumTropicalVertex}}]
\label{thm: quiver DT to log BPS}
If $Q$ is acyclic we have
\[ \Omega^{Q}_{d(\upbeta)}\big(q^{\frac{1}{2}}\big) = \Omega^{(S|D+E)}_{\upbeta}\big(q^{\frac{k}{2}}\big).\]
\end{theorem}

\begin{remark}
	Strictly speaking \cite{BousseauQuantumTropicalVertex} only establishes an equality between the refined Donaldson--Thomas invariants of $Q$ and the refined BPS invariants of $\smash{(\widetilde{S} | \widetilde{D})}$. The statement of \Cref{thm: quiver DT to log BPS} then follows as a consequence of birational invariance of logarithmic Gromov--Witten invariants \cite{AbramovichWiseBirational} which identifies the refined BPS invariants of $(S|D+E)$ and $\smash{(\widetilde{S} | \widetilde{D})}$.
\end{remark}


Combining \Cref{thm: quiver DT to log BPS} with \Cref{cor: refined} we obtain:

\begin{corollary}
\label{cor:divisibility-quiver}
If $Q$ is acyclic we have
\[
\Omega^{Q}_{d(\upbeta)}\big(q^{\frac{1}{2}}\big) = [D \cdot \upbeta]_q \, \Omega_\upbeta^{X^\circ}\!\big(q\big).
\]
\end{corollary}

Crucially, this corollary implies that $M_{\mathrlap{d(\upbeta)}}{\mathstrut}^{\uptheta-\text{ss}}(Q)$ virtually admits the structure of a projective bundle. This is surprising since, excluding small cases, the moduli space does not admit the structure of a true projective bundle.

\begin{remark} We must restrict to pairs $(S|D+E)$ such that the resulting quiver $Q$ is acyclic. The above results can be extended to cyclic quivers in genus zero by introducing a superpotential, see \cite{ABQuiver}.\end{remark}

\section{Self-nodal plane curves}\label{sec: nodal cubic}

\noindent In this section we focus on an important special case. Fix $r \geq 1$ and consider the toric variety
\[ S_r \colonequals \PP(1,1,r).\footnote{We can also take $r=0$ in which case we have $S_0 \colonequals \PP^1 \times \PP^1$ and $D_0$ the union of a $(1,0)$ curve and a smooth $(1,2)$ curve. In this case, the results of this section follow from a direct calculation, using \cite[Proposition~6.1]{GPS}. Similarly we can take $r=-1$ which results in the local geometry of a $(-1)$-curve in a surface.}\]
Let $D_r \in |\!-\!\!K_{S_r}|$ be an irreducible curve with a single nodal toric singularity at the singular point of $S_r$ (or at one of the torus-fixed points if $r=1$). The pair $(S_r|D_r)$ is logarithmically smooth. By a curve in $S_r$ of degree $d$ we mean a curve whose class is $d$ times the class of the toric hypersurface with self-intersection $r$. Given a curve in $S_r$ of degree $d$, its intersection number with $D_r$ is $d(r+2)$. We consider the genus zero maximal tangency Gromov--Witten invariants:
\[ \GW_{0,(d(r+2)),d}(S_r|D_r) \in \Q.\]
We begin by deriving an explicit formula for these invariants (\Cref{thm: nodal cubic invariants}) and applying it to deduce a formula for the invariants of local $\PP^1$ (\Cref{thm: local P1 invariants}).

We then specialise to $r=1$ and establish a relationship between the invariants of $(\PP^2|D_1)$ and $(\PP^2|E)$ for $E$ a smooth cubic (\Cref{thm: nodal cubic invariants are contributions}). We apply this to prove a conjecture of Barrott and the second-named author (\Cref{thm: BN conjecture holds}).

\subsection{Genus zero invariants} \label{sec: scattering} The main result of this section is:

\begin{theorem}[\Cref{thm: nodal cubic introduction}] \label{thm: nodal cubic invariants}
	We have:
	\[ \GW_{0,(d(r+2)),d}(S_r|D_r) = \dfrac{r+2}{d^2} {(r+1)^2 d-1 \choose d-1}.\] 
\end{theorem}

The following result, already known in the physics literature \cite[Equation~(4.53)]{CaporasoGriguoloMarinoPasquettiSeminara}, is a direct consequence of \Cref{thm: exceptional divisors correspondence} and \Cref{thm: local open correspondence}.

\begin{theorem}[\Cref{thm: local P1 introduction}] \label{thm: local P1 invariants}
	We have
	\[ \GW^T_{0,0,d}\big(\Ocal_{\PP^1}(r)\oplus\Ocal_{\PP^1}(-r-2)\big) = \dfrac{(-1)^{rd - 1}}{d^3} {(r+1)^2 d-1 \choose d-1} \] 
where the left-hand side is defined via localisation with respect to the Calabi--Yau torus, as in Section~\ref{sec: open invariants}.
\end{theorem}

\begin{proof}
There is a toric resolution of singularities $\mathbb{F}_r \to \PP(1,1,r)$. Consider $D_r \subseteq \mathbb{F}_r$ the strict transform and $E \subseteq \mathbb{F}_r$ the exceptional divisor.  Realise the Hirzebruch surface as a $\PP^1$-bundle
\[ \mathbb{F}_r \cong \PP_{\PP^1}(\OO_{\PP^1}(r) \oplus \OO_{\PP^1}) \]
with $E \subseteq \mathbb{F}_r$ the zero section and $E_\infty \subseteq \mathbb{F}_r$ the infinity section. By \cite{AbramovichWiseBirational} the Gromov--Witten invariants of $(S_r|D_r)$ are identified with the Gromov--Witten invariants of the bicyclic pair $(\mathbb{F}_r\, | \, D_r + E)$. We then have
\begin{align} \label{eqn: identification of log open geometry for WPS} \GW_{0,(d(r+2)),d}(S_r|D_r) & = \GW_{0,(d(r+2),0),d E_\infty}(\mathbb{F}_r \, | \, D_r + E) \nonumber \\
& = (-1)^{d(r+2)+1} d(r+2) \cdot \GW_{0,0,d\, \upiota^\star E_\infty}^T(\OO_{\mathbb{F}_r}(-D_r)|_{\mathbb{F}_r \setminus E})
\end{align}
where the second equality follows by combining Theorems~\ref{thm: exceptional divisors correspondence}~and~\ref{thm: local open correspondence}. Now note that $\mathbb{F}_r \setminus E$ is the total space of $N_{E_\infty|\mathbb{F}_r}$ which is a degree $r$ line bundle on $E_\infty \cong \mathbb{P}^1$. Similarly, we have $\deg \OO_{\mathbb{F}_r}(-D_r)|_{E_\infty} = -r-2$. This establishes the identity
\[ \GW_{0,0,d\, \upiota^\star E_\infty}^T(\OO_{\mathbb{F}_r}(-D_r)|_{\mathbb{F}_r \setminus E}) = \GW^T_{0,0,d}\big(\Ocal_{\PP^1}(r)\oplus\Ocal_{\PP^1}(-r-2)\big) \]
and the result follows by combining \eqref{eqn: identification of log open geometry for WPS} and \Cref{thm: nodal cubic invariants}.
\end{proof}

It remains to prove \Cref{thm: nodal cubic invariants}. This proceeds via an application of the Gromov--Witten/quiver correspondence of \Cref{sec: quiver DT}. To construct a quiver for the pair $(S_r|D_r)$ we first need to choose a toric model. In the following we will identify a divisor with its strict transform (respectively image) under a blowup (respectively blowdown).

As in the proof of \Cref{thm: local P1 invariants} we consider the resolution of singularities $\mathbb{F}_r \rightarrow S_r$, writing $D_r\subseteq \mathbb{F}_r$ for the strict transform of the self-nodal curve and $E\subseteq \mathbb{F}_r$ for the exceptional divisor. We now further blowup at the two intersection points $\{q_1,q_2\}=D_r \cap E$ and let $F_1,F_2$ denote the resulting exceptional divisors. We write
\[\varphi\,:\,(\widetilde{S}_r \,|\, \widetilde{D}_r)\coloneqq(\mathrm{Bl}_{q_1,q_2}\mathbb{F}_r \, |\, D_r + F_1 + E + F_2) \longrightarrow (S_r | D_r)\]
for the composition of these blowups. Let $L_1,L_2 \subseteq \smash{\widetilde{S}_r}$ denote the strict transforms of the tangent lines at the singularity of $D_r \subseteq S_r$. These are $(-1)$-curves which we blow down to produce:
\[\uppi \,:\, (\widetilde{S}_r \, |\, \widetilde{D}_r) \rightarrow (\ol{S}_r\, |\, D_r + F_1 + E + F_2). \]
A quick calculation shows that the self-intersection numbers of $D_r,F_1,E,F_2\subseteq \ol{S}_r$ are $(r+2),0,-(r+2),0$. By \cite[Lemma~2.10]{FriedmanAnticanonical} this identifies the pair $(\ol{S}_r|D_r + F_1 + E + F_2)$ with $(\mathbb{F}_{r+2}\, |\, \partial \mathbb{F}_{r+2})$ meaning we have indeed found a toric model $\uppi$ for $(S_r \, |\, D_r)$.

Now as $\uppi$ is the blowup of points on $F_1$ and $F_2$ and $|\uprho_{F_1} \wedge \uprho_{F_2}| = r+2$ the quiver $Q_{r+2}$ associated to $(S_r| D_r)$ (respectively $(\mathbb{F}_r | D_r + E)$) is the $(r+2)$-Kronecker quiver:
\begin{equation*}
	\begin{tikzcd}
		Q_{r+2} \, : &[-2.55em] \bullet \arrow[r,bend left=40] \arrow[r,bend left=25] \arrow[r,bend left=-40] \arrow[r,phantom,"\vdots"] &[2em] \bullet
	\end{tikzcd}
\end{equation*}
Since the section $E_\infty$ of $\mathbb{F}_r$ intersects both $L_1$ and $L_2$ in a single point, as an application of \Cref{thm: quiver DT to log BPS} we obtain
\begin{equation}
	\label{eq: Fr GW to Kronecker DT}
	\Omega^{(\mathbb{F}_r | D_r + E)}_{d E_\infty}\big(q^{\frac{1}{2}}\big) = \Omega^{Q_{r+2}}_{(d,d)}\big(q^{\frac{1}{2}}\big)\,.
\end{equation}
The $q^{\frac{1}{2}}=1$ specialisation of the right-hand side Donaldson--Thomas invariant $\smash{\Omega^{Q_{r+2}}_{(d,d)}(1)} \in \Z$ has been calculated in \cite{Reineke,MeinRei} (see also \cite[Section 3 and 4]{Reineke:QuantumDilog}).

\begin{lemma}[\!{\cite[Theorem~5.2]{Reineke} and \cite[Theorem~4.6]{MeinRei}}]	\label{lem: DT Kronecker}
	Let $\upmu$ be the M\"{o}bius function. Then for all $d>0$ we have
	\begin{equation*}
		\Omega^{Q_{r+2}}_{(d,d)}\big(1\big) = \frac{1}{r d^2} \sum_{\ell|d} \upmu(d/\ell) \,(-1)^{(r+2)\ell+1} \,\binom{(r+1)^2 \ell - 1}{\ell}.
	\end{equation*}
\end{lemma}

Via the Gromov--Witten/quiver correspondence \eqref{eq: Fr GW to Kronecker DT}, this \namecref{lem: DT Kronecker} determines all genus-zero Gromov--Witten invariants of $(S_r| D_r)$.

\begin{proof}[Proof of \Cref{thm: nodal cubic invariants}]
	Specialising \eqref{eqn: relating Omega and Omegabar} to $q^{\frac{1}{2}}=1$ and using birational invariance of logarithmic Gromov--Witten invariants \cite{AbramovichWiseBirational} we get
	\begin{equation*}
		\GW_{0,(d(r+2)),d}(S_r|D_r) = (-1)^{(r+2)d-1} \sum_{\ell d' = d} \frac{1}{\ell^2}\, \Omega^{(\mathbb{F}_r | D_r + E)}_{d' E_\infty}\big(1\big).
	\end{equation*}
	Combining equation \eqref{eq: Fr GW to Kronecker DT} and \Cref{lem: DT Kronecker} the right-hand side evaluates to
	\begin{align*}
		\GW_{0,(d(r+2)),d}(S_r|D_r) & = (-1)^{(r+2)d-1} \sum_{\ell d' = d} \frac{1}{\ell^2} \frac{1}{r {d'}^2} \sum_{\ell'|d'} \upmu(d'/\ell') \,(-1)^{(r+2)\ell'+1} \,\binom{(r+1)^2 \ell' - 1}{\ell'} \\
		& = \frac{1}{rd^2} \binom{(r+1)^2 d - 1}{d}
	\end{align*}
	where the last identity follows from the M\"obius inversion formula. The statement of \Cref{thm: nodal cubic invariants} now follows from the identity
	\begin{equation}
		\label{eqn: compare binomial coeffs} \binom{(r+1)^2d-1}{d} = r(r+2)\binom{(r+1)^2d-1}{d-1}.\qedhere
	\end{equation}
\end{proof}

\subsection{Nodal cubic versus smooth cubic} \label{sec: degenerating hypersurfaces} In this final section we set $r=1$. We have 
\[ D_1 = D \subseteq \PP^2\]
an irreducible cubic with a single nodal singularity. Let $E \subseteq \PP^2$ be a smooth cubic. For each of the two pairs we consider the moduli space of genus zero stable logarithmic maps, with maximal tangency at a single marking (see Figure~\ref{fig: nodal and smooth cubic}).
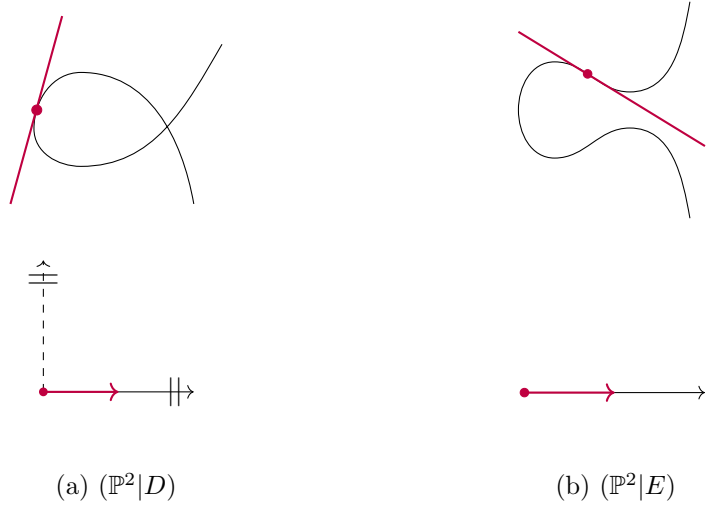
\begin{figure}[h]
\centering
\begin{subfigure}[b]{0.3\textwidth}
\[
\begin{tikzpicture}[scale=2.5]
\draw (0.55,4.35) to [out=240,in=0] (-0.2,3.7) to [out=180, in=270] (-0.45,3.9) to [out=90,in=180] (-0.2,4.2) to [out=0, in=100] (0.4,3.5);	
\draw[purple,thick] (-0.3,4.5) - -(-0.575,3.5);
\draw[purple,fill=purple] (-0.435,4) circle[radius=0.75pt];

\draw[->] (-0.4,2.5) -- (0.4,2.5);
\draw (0.3,2.5) node{$\mid\mid$};
\draw[->,dashed] (-0.4,2.5) -- (-0.4,3.2);
\draw (-0.4,3.1) node[rotate=90]{$\mid\mid$};
\draw[purple,fill=purple,thick] (-0.4,2.5) circle[radius=0.5pt];
\draw[purple,thick,->] (-0.4,2.5) -- (0,2.5);
\end{tikzpicture}
\]
\label{(P2,D) tangent line picture}
\caption{$(\PP^2|D)$}
\end{subfigure}\qquad\qquad
\begin{subfigure}[b]{0.3\textwidth}
\[
\begin{tikzpicture}[scale=0.8]
\draw (0.25,-0.5) to[out=260,in=0] (-0.75,-2) to[out=180,in=0] (-2,-1.5) to[out=180,in=90] (-2.6,-2.3) to[out=270,in=180] (-2,-3.1) to[out=0,in=180] (-0.75,-2.6) to[out=0,in=100] (0.25,-4.1);
\draw[purple,thick] (-2.6,-1) -- (0.5,-2.9);
\draw[purple,fill=purple] (-1.45,-1.7) circle[radius=2pt];
\draw[->] (-2.5,-7) -- (0.5,-7);
\draw[purple,fill=purple] (-2.5,-7) circle[radius=2pt];
\draw[purple,thick,->] (-2.5,-7) -- (-1,-7);
\draw (0.5,-7.8) node{$$};
\end{tikzpicture}
\]
\caption{$(\PP^2|E)$}
\label{(P2,E) tangent line picture}
\end{subfigure}
\caption{Tangent curves to plane cubics, nodal and smooth.}
\label{fig: nodal and smooth cubic}
\end{figure}

In genus zero, each moduli space has virtual dimension zero and produces a system of enumerative invariants indexed by $d$. Both theories are completely solved: the case of $(\PP^2|D)$ is solved in Theorem~\ref{thm: nodal cubic invariants}, while the case of $(\PP^2|E)$ is solved in \cite[Example~2.2]{GathmannMirror} with inspiration from \cite{TakahashiMirror}. The numbers do not agree, as the following table demonstrates:
\[
\begin{tabu}{| c || c | c |}
\hline
d & \GW_{0,(3d),d}(\PP^2|D) & \GW_{0,(3d),d}(\PP^2|E) \\
\hline\hline
1 & 3 & 9 \\
2 & 21/4 & 135/4 \\
3 & 55/3 & 244 \\
4 & 1,365/16 & 36,999/16 \\
5 & 11,628/25 & 635,634/25 \\
6 & 33,649/12 & 307,095 \\
\hline
\end{tabu}
\]
Experimentally, we always have
\begin{equation} \label{eqn: nodal cubic less than smooth cubic invariants} \GW_{0,(3d),d}(\PP^2|D) < \GW_{0,(3d),d}(\PP^2|E).\end{equation}
In this section we provide a conceptual explanation for this defect, via the geometry of degenerating hypersurfaces (Theorem~\ref{thm: nodal cubic invariants are contributions}). We then settle a conjecture posed in \cite{BarrottNabijou} (Theorem~\ref{thm: BN conjecture holds}).

The paper \cite{BarrottNabijou} degenerates a smooth cubic to the toric boundary, and studies the resulting logarithmic Gromov--Witten theory on the central fibre. The following construction is similar, except that our starting point is a nodal cubic. We explain how the arguments adapt to this setting, assuming some familiarity with \cite{BarrottNabijou}. Let $\Delta \subseteq \PP^2$ denote the toric boundary and consider a degeneration $D \rightsquigarrow \Delta$, i.e. a divisor
\[ \scrD \subseteq \PP^2 \times \Aone \]
whose general fibre is an irreducible nodal cubic and whose central fibre is $\Delta$. We can choose $\scrD$ to be irreducible with normal crossings singularities, and such that $\uppi^{-1}(t) \cap \scrD^{\mathrm{sing}}$ is the nodal point of $\scrD_t$ for $t \neq 0$, and
\[ \uppi^{-1}(0) \cap \scrD^{\mathrm{sing}} = p_0\]
where $p_0=[1,0,0]$. Consider the logarithmically regular logarithmic scheme
\[ \scrY = (\PP^2 \times \Aone\,  |\, \scrD).\]
Equip $\Aone$ with the trivial logarithmic structure and consider the logarithmic morphism $\scrY \to \Aone$. This is not logarithmically smooth, but is logarithmically flat; the proof is similar to \cite[Lemma~3.7]{BarrottNabijou}. The general fibre $\scrY_t$ is the logarithmic scheme associated to the smooth pair $(\PP^2|\scrD_t)$. The central fibre, on the other hand, is logarithmically singular. The stalks of the ghost sheaf of $\scrY_0$ are
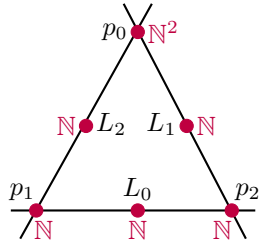
\[
\begin{tikzpicture}[scale=2.5]
\draw [thick] (4.6,5.2) -- (5.3,6.45);
\draw [thick] (5.15,6.45) -- (5.8,5.2);
\draw [thick] (4.55,5.35) -- (5.85,5.35);

\draw[purple,fill=purple] (5.225,6.3) circle[radius=1pt];
\draw[purple] (5.225,6.3) node[right]{\small$\N^2$};
\draw (5.225,6.3) node[left]{\small$p_0$};

\draw[purple,fill=purple] (4.685,5.35) circle[radius=1pt];
\draw (4.732,5.45) node[left]{\small$p_1$};
\draw[purple] (4.725,5.35) node[below]{\small$\N$};

\draw[purple,fill=purple] (5.725,5.35) circle[radius=1pt];
\draw (5.69,5.45) node[right]{\small$p_2$};
\draw[purple] (5.675,5.35) node[below]{\small$\N$};

\draw[purple,fill=purple] (5.225,5.35) circle[radius=1pt];
\draw[purple] (5.225,5.35) node[below]{\small$\N$};
\draw (5.225,5.35) node[above]{\small$L_0$};

\draw[purple,fill=purple] (4.95,5.8) circle[radius=1pt];
\draw (4.95,5.81) node[right]{\small$L_2$};
\draw[purple] (4.95,5.8) node[left]{\small$\N$};

\draw[purple,fill=purple] (5.485,5.8) circle[radius=1pt];
\draw (5.485,5.81) node[left]{\small$L_1$};
\draw[purple] (5.485,5.8) node[right]{\small$\N$};

\end{tikzpicture}
\]
where $L_0,L_1,L_2 \subseteq \PP^2$ are the coordinate lines and $p_0,p_1,p_2 \in \PP^2$ the coordinate points.

As in \cite[Section~2]{BarrottNabijou} we construct, in genus zero, a virtual fundamental class on the space of stable logarithmic maps to the central fibre $\scrY_0$. Integrals against this class recover the invariants of $(\PP^2|D)$ by the conservation of number principle. We now study the moduli space
\[ \Mbar_{0,\bfc,d}(\scrY_0).\]
As in \cite[Lemma~4.5]{BarrottNabijou} we find that every logarithmic map to $\scrY_0$ must factor through $\Delta$. In fact, in our new setting we obtain a stronger constraint.

\begin{lemma} \label{lem: factors through L0}Given a logarithmic map to $\scrY_0$ the underlying schematic map to $\PP^2$ factors through $L_0$.\end{lemma}

\begin{proof} Let $C \to \scrY_0$ be a logarithmic map and consider its tropicalisation $\tropf \colon \Sigma C \to \Sigma \scrY_0$. The target $\Sigma \scrY_0$ is the cone complex:
\[
\begin{tikzpicture}[scale=0.7]
			\draw[fill=black] (0,0) circle[radius=2pt];
			
			\draw[->] (0,0) -- (3,0);
			\draw (3,0) node[right]{$\Delta$};
			
			\draw[->] (0,0) -- (0,3);
			\draw (0,3) node[above]{$\Delta$};
			
			\draw (0.4,0) node[above]{$\PP^2$};
			
			\draw (2.5,0) node{$\mid\mid$};
			\draw (0,2.5) node[rotate=90]{$\mid\mid$};
			
			\draw (2,2) node{$p_0$};
\end{tikzpicture}
\]
Following \cite[proof of Lemma~4.5]{BarrottNabijou} it suffices to show that the image of $\tropf$ does not intersect the interior of the maximal cone $p_0$.

We first describe the balancing condition. For $v \in V(\Gamma)$ let $C_v \subseteq C$ denote the corresponding irreducible component. Since $f(C_v) \subseteq \Delta$ we have that $\tropf(v)$ belongs to either $\Delta$ or $p_0$. In the latter case, the balancing condition states that the sum of the outgoing slope vectors is zero. The interesting case is $\tropf(v) \in \Delta$. We distinguish two possibilities:
\begin{itemize}
\item If $f(C_v) \subseteq L_0$ then all outgoing edges from $v$ are contained in $\Delta$ and the sum of their slopes is equal to $3d_v$.
\item If $f(C_v) \subseteq L_1$ or $L_2$ then the slope of each outgoing edge from $v$ can be broken into components tangent to $\Delta$ and normal to $\Delta$. The sum of the slopes tangent to $\Delta$ is equal to $2d_v$, while the sum of the slopes normal to $\Delta$ is equal to $d_v$. Moreover, edges with positive slope normal to $\Delta$ must all enter the same neighbourhood of $v$ in the above chart. Which neighbourhood they enter depends on which of $L_1$ or $L_2$ the component $C_v$ maps to.
\end{itemize}
Due to the balancing condition, every tropical map $\Sigma C \to \Sigma \scrY_0$ admits a lift to the standard cover of the target:
\[
\begin{tikzcd}
\, & \RR^2_{\geq 0} \ar[d] \\
\Sigma C \ar[ru] \ar[r] & \Sigma \scrY_0.
\end{tikzcd}
\]
Choose such a lift, and suppose for a contradiction that the image of $\tropf$ intersects the interior of $p_0$. Then there exists an edge $e \in E(\Gamma)$ such that $\tropf(e)$ intersects the interior of $p_0$ and such that one of the end vertices of $e$ is mapped to $\Delta$. On the lift we have:
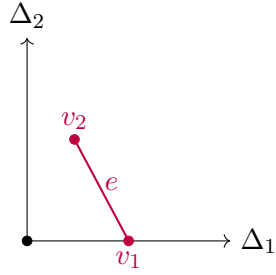
\[
\begin{tikzpicture}[scale=0.9]
			\draw[fill=black] (0,0) circle[radius=2pt];
			
			\draw[->] (0,0) -- (3,0);
			\draw (3,0) node[right]{$\Delta_1$};
			
			\draw[->] (0,0) -- (0,3);
			\draw (0,3) node[above]{$\Delta_2$};
			
			
			
			\draw[purple,fill=purple] (1.5,0) circle[radius=2pt];
			\draw[purple] (1.5,0) node[below]{$v_1$};
			\draw[purple,thick] (1.5,0) -- (0.7,1.5);
			\draw[purple] (1,0.85) node[right]{$e$};
			\draw[purple,fill=purple] (0.7,1.5) circle[radius=2pt];
			\draw[purple] (0.7,1.5) node[above]{$v_2$};
\end{tikzpicture}
\]
Define $\Gamma_i \subseteq \Gamma$ by cutting $\Gamma$ at $e$ and taking the subgraph containing $v_i$. Start with $\Gamma_1$. By balancing we have $d_{v_1} > 0$ and so $v_1$ supports an outgoing edge with positive slope in the $\Delta_1$-direction. Traversing along this edge to the next vertex, we again conclude by balancing that there exists an outgoing edge with positive slope in the $\Delta_1$-direction. Continuing in this way, we eventually arrive at the marking leg. It follows that the marking leg $\Gamma_1$.

Now consider $\Gamma_2$. By balancing the vertex $v_2$ supports an outgoing edge with positive slope in the $\Delta_2$-direction (this occurs both if $\ftrop(v_2) \in p_0$ and if $\ftrop(v_2) \in \Delta_2$). As above, we inductively traverse the graph and produce a path consisting of edges with positive slope in the $\Delta_2$-direction. Eventually we arrive at the marking leg. It follows that the marking leg is contained in $\Gamma_2$.

The marking leg is thus contained in both $\Gamma_1$ and $\Gamma_2$. But these subgraphs are disjoint: since $\Gamma$ has genus zero, the edge $e$ is separating.
\end{proof}

Using \Cref{lem: factors through L0} we can show as in \cite[Proposition~4.11]{BarrottNabijou} that
\[ \Mbar_{0,\bfc,d}(\scrY_0) = \Mbar_{0,1,d}(L_0) \cong \Mbar_{0,1,d}(\PP^1). \]
This space carries a virtual fundamental class of dimension zero, arising from logarithmic deformation theory. In \cite[Section~4]{BarrottNabijou} this is expressed as an obstruction bundle integral, and in \cite[Section~5]{BarrottNabijou} it is computed via localisation. These calculations apply verbatim in our new setting, because the logarithmic scheme $\scrY_0$ agrees with the logarithmic scheme $\Xcal_0$ of \cite[Section~3.1]{BarrottNabijou} in a neighbourhood of $L_0$. We conclude:

\begin{theorem}[\Cref{thm: nodal cubic invariants are contributions introduction}] \label{thm: nodal cubic invariants are contributions} The invariant of $(\PP^2|D)$ is precisely the central fibre contribution to the invariant of $(\PP^2|E)$ arising from multiple covers of a single line of $\Delta$. In the notation of \cite[Section~5.5]{BarrottNabijou}:
\[ \GW_{0,(3d),d}(\PP^2|D) = \operatorname{C}_{\operatorname{ord}}(d,0,0).\]
Thus the invariants of $(\PP^2|D)$ constitute one contribution, amongst many, to the invariants of $(\PP^2|E)$. Experimentally, all contributions are positive: this explains the inequality \eqref{eqn: nodal cubic less than smooth cubic invariants}.
\end{theorem}

\Cref{thm: nodal cubic invariants are contributions} also allows us to compute $\operatorname{C}_{\operatorname{ord}}(d,0,0)$. In \cite[Section~5]{BarrottNabijou} these were computed up to $d=8$ by computer-assisted torus localisation. Based on these numerics, the following formula was proposed:
\begin{conjecture}[\!\!{\cite[Conjecture 5.9(39)]{BarrottNabijou}}] \label{conj: multiple cover contributions} We have the following hypergeometric expression for the contribution of degree $d$ covers of $L_0$:
\[ \operatorname{C}_{\operatorname{ord}}(d,0,0) = \dfrac{1}{d^2\ } {4d-1 \choose d}.\]\end{conjecture}

It was then shown \cite[Proposition~5.13]{BarrottNabijou} that Conjecture~\ref{conj: multiple cover contributions} is equivalent to the following conjecture in pure combinatorics, which is verified by computer up to $d=50$:

\begin{conjecture}[\!{\cite[Conjecture~5.12]{BarrottNabijou}}] \label{conj: combinatorial} Fix an integer $d \geq 1$. Then we have
\[	\sum_{(d_1,\ldots,d_r) \vdash d} \dfrac{2^{r-1} \cdot d^{r-2}}{\#\! \operatorname{Aut}(d_1,\ldots,d_r)} \prod_{i=1}^r \dfrac{(-1)^{d_i-1}}{d_i} {3d_i \choose d_i} = \dfrac{1}{d^2} {4d-1 \choose d} \]
where the sum is over strictly positive unordered partitions of $d$ (of any length).\end{conjecture}

Using \Cref{thm: nodal cubic invariants are contributions} we can now prove both these conjectures.

\begin{theorem} \label{thm: BN conjecture holds} Conjecture~\ref{conj: multiple cover contributions}, and hence also Conjecture~\ref{conj: combinatorial}, holds.	
\end{theorem}

\begin{proof} By \Cref{thm: nodal cubic invariants are contributions} it is equivalent to show that
\[ \GW_{0,(3d),d}(\PP^2|D) = \dfrac{1}{d^2\ } {4d-1 \choose d}. \qquad\] 
This follows from \Cref{thm: nodal cubic invariants} for $r=1$ and the identity \eqref{eqn: compare binomial coeffs}. See also \cite[Remark~5.10]{BarrottNabijou}.
\end{proof}

\footnotesize
\bibliographystyle{alpha}
\bibliography{Bibliography.bib}\medskip

Michel van Garrel. University of Birmingham. \href{mailto:m.vangarrel@bham.ac.uk}{m.vangarrel@bham.ac.uk}

Navid Nabijou. Queen Mary University of London. \href{mailto:n.nabijou@qmul.ac.uk}{n.nabijou@qmul.ac.uk}

Yannik Schuler. University of Sheffield. \href{mailto:yschuler1@sheffield.ac.uk}{yschuler1@sheffield.ac.uk}

\end{document}